\documentclass[12pt,english,british]{article}
\usepackage[T1]{fontenc}
\usepackage[latin9]{inputenc}
\usepackage{geometry}
\usepackage{color}
\usepackage{amsmath}
\usepackage{amsthm}
\usepackage{amssymb}

\makeatletter
\numberwithin{equation}{section}
\numberwithin{figure}{section}
  \theoremstyle{remark}
  \newtheorem*{rem*}{\protect\remarkname}
\theoremstyle{plain}
\newtheorem{thm}{\protect\theoremname}[section]
  \theoremstyle{plain}
  \newtheorem{prop}[thm]{\protect\propositionname}
  \theoremstyle{remark}
  \newtheorem{rem}[thm]{\protect\remarkname}
  \theoremstyle{plain}
  \newtheorem{lem}[thm]{\protect\lemmaname}
\newtheorem{cor}[thm]{\protect\corollaryname}

\makeatother

\usepackage{babel}
\usepackage{ulem}
  \addto\captionsbritish{\renewcommand{\lemmaname}{Lemma}}
  \addto\captionsbritish{\renewcommand{\propositionname}{Proposition}}
  \addto\captionsbritish{\renewcommand{\remarkname}{Remark}}
  \addto\captionsbritish{\renewcommand{\theoremname}{Theorem}}
  \addto\captionsenglish{\renewcommand{\lemmaname}{Lemma}}
  \addto\captionsenglish{\renewcommand{\propositionname}{Proposition}}
  \addto\captionsenglish{\renewcommand{\remarkname}{Remark}}
  \addto\captionsenglish{\renewcommand{\theoremname}{Theorem}}
  \providecommand{\lemmaname}{Lemma}
  \providecommand{\propositionname}{Proposition}
  \providecommand{\remarkname}{Remark}
\providecommand{\theoremname}{Theorem}
\providecommand{\corollaryname}{Corollary}

\def\diff{\mathop{\mathrm d\null}\!}

\let\four=\mathbf
\def\fourmetric{{\four g}}
\def\fourRicci{{\four R}}
\def\orbitmetric{\tilde g}
\def\p#1#2{f_{#1}^{#2}}
\def\tr{\mathop{\mathrm{tr}}\nolimits}

\begin{document}

\begin{center}
\Large The equivalence problem for generic four-dimensional metrics 
with two commuting Killing vectors
\\[2ex]
\normalsize D. Catalano Ferraioli%
\footnote{Instituto de Matem\'atica e Estat\'{\i}stica, Universidade Federal da Bahia, 
Campus de Ondina, Av. Adhemar de Barros, S/N, Ondina -- CEP 40.170.110 -- Salvador, 
BA -- Brazil, e-mail: \tt diego.catalano@ufba.br}
and  M. Marvan%
\footnote{Mathematical Institute in Opava, Silesian
  University in Opava, Na Rybn\'{\i}\v{c}ku 626/1, 746\,01 Opava, Czech Republic, 
  email: \tt Michal.Marvan@math.slu.cz}
\\[3ex]

February 20, 2019
\end{center}


\begin{abstract}
We consider the equivalence problem of four-dimensional semi-Riemannian metrics with the
$2$-dimensional Abelian Killing algebra. 
In the generic case we determine a semi-invariant frame and a fundamental set of first-order scalar differential invariants suitable for solution of the equivalence problem. 
Genericity means that the Killing leaves are not null, the metric is not orthogonally transitive (i.e., the distribution orthogonal to the Killing leaves is non-integrable), and two explicitly constructed scalar 
invariants $C_\rho$ and $\ell_{\mathcal C}$ are nonzero.
All the invariants are designed to have tractable coordinate expressions. 
Assuming the existence of two functionally independent invariants, we solve the equivalence problem in two ways.
As an example, we invariantly characterise the Van den Bergh metric.
To understand the non-generic cases, we also find all $\Lambda$-vacuum metrics that are generic in the above sense, 
except that either $C_\rho$ or $\ell_{\mathcal C}$ is zero. 
In this way we extend the Kundu class to $\Lambda$-vacuum metrics.
The results of the paper can be exploited for invariant characterisation of classes of metrics and for extension of the set of known solutions of the Einstein equations.

\vspace{0.5cm}
\noindent
\footnotesize {\bf Keywords:} differential invariants,
metric equivalence problem, Kundu class \\
\footnotesize {\bf MSC:} 83C20, 35Q76. 

\end{abstract}

\section{Introduction}

Scalar differential invariants have multiple uses in general relativity. 
Scalar {\it polynomial\/} invariants
\cite{C-M} arise as scalar contractions of $g,R$ and the 
covariant derivatives $\nabla R, \dots, \nabla^m R$. 
Besides being a tool to detect true singularities irremovable by coordinate
transformations, scalar differential invariants provide a basis for solving the 
equivalence problem, i.e., the problem of classifying spacetime metrics with respect to
local isometries.
Scalar differential invariants can, in principle, solve the equivalence problem 
except for metrics of the Kundt class~\cite{C-H-P}, but not in an effective way.
Here the Cartan--Karlhede invariants, see~\cite{Ka, Ka2006} or
\cite[Ch.~9]{S-K-M-H-H}, come to the rescue.
Cartan--Karlhede invariants, defined as components of the Riemann tensor and its 
covariant derivatives with respect to suitably chosen frames, lie in the heart 
of a workable algorithm to decide about 
equivalence of space-time metrics~\cite{A-Ka,Ka-M,P-S-dI3}. 
Another useful application is that of finding solutions of Einstein's equations by 
imposing additional invariant constraints~\cite{B-M,L-Yu 1,L-Yu 2,P-P-C-M}.

All invariants mentioned so far started at the second order,
a strict lower bound for scalar invariants of metrics~\cite{Zor}.
To enable first-order metric invariants, one would have to reduce the pseudogroup
of diffeomorphisms. 
One important case when this is easily done is when the metric has Killing 
fields. 
The semi-Riemannian manifold then becomes a  submersion in a natural way and 
instead of the equivalence of space-times we can consider the equivalence 
of the semi-Riemannian submersions. 
This removes the ban on the first-order invariants, while not causing any harm to 
solution of the equivalence problem, if the submersion is taken with respect to 
the full Killing algebra.

More precisely, in the case of a semi-Riemannian manifold 
$(\mathcal M,\mathbf{g})$ with the Killing algebra 
$\mathfrak{Kill}(\mathbf{g})=\mathcal{G}$ we consider the
semi-Riemannian submersion $\mathcal M \to \mathcal M/G$, where $\mathcal M/G$ 
is the orbit space of the Lie group  
$G$ of the transformations generated  by $\mathcal{G}$  on $\mathcal M$.
Obviously, two semi-Riemannian manifolds are isometric if and only if the
corresponding semi-Riemannian submersion structures are isomorphic.

In the present paper we apply the above scheme to the particular class, 
denoted $\mathrm{G}_{2}$, 
of four-dimensional semi-Riemannian metrics whose Killing algebra is two-dimensional
and generated by two commuting Killing fields $\xi_1,\xi_2$ 
(this class is denoted by $\mathrm{G}_2\mathrm{I}$ in~\cite[Ch.~17]{S-K-M-H-H}).
We also require that the Killing leaves (orbits) of the foliation are non-null,
i.e., the metric restricted to the leaves is not degenerated.
For these metrics we obtain a fundamental system of functionally independent scalar 
differential invariants and solve the problem of equivalence.

By considering the distribution  $\Xi^{\perp}$, orthogonal to
the Killing leaves, we shall distinguish two cases:
 the case when $\Xi^{\perp}$ is not integrable, which will be referred to as the 
 \textit{orthogonally intransitive case}; the case when $\Xi^{\perp}$
is integrable, which will be referred to as the 
\textit{orthogonally transitive case}~\cite{Ca2}.

In the orthogonally intransitive case we construct a fundamental system
of $6$ functionally independent first-order scalar differential invariants  
as well as a first-order (semi-)invariant frame.
Moreover, these six invariants admit very simple explicit expressions in terms
of metric coefficients, in sharp contrast to the curvature invariants 
mentioned above. 

In the orthogonally transitive case only
$4$ functionally independent first-order scalar differential invariants exist; 
this case has been completely solved in~\cite{M-S} 
(without constructing an invariant frame).

We have chosen the class $\mathrm{G}_{2}$ because it is rather rich in explicit  solutions of Einstein equations, 
especially in the orthogonally transitive subcase (when the vacuum, 
electro-vacuum and some other cases are integrable in the sense of soliton theory), 
see \cite{Aleks,C2,S-K-M-H-H,K-R,Ve} and references therein.
At the same time, only a handful of orthogonally intransitive metrics are known (e.g., \cite{Gaff, K-S, Bergh, W-R, WCL}).
The treatment of non-generic cases is generally left aside.
In this paper we just characterise metrics with vanishing invariants $C_\rho$ or 
$\ell_{\mathcal C}$. 
In the latter case (Kundu class) we obtained new explicit
$\Lambda$-vacuum solutions.

The paper is organized as follows. In Section~\ref{sec2} we introduce the Lie pseudogroup $\mathfrak{G}$ acting on four-dimensional semi-Riemannian manifolds of class $\mathrm{G}_{2}$. In Section \ref{sec3}, denoting by $\mathfrak{G}_{\tau}$ the natural extension of $\mathfrak{G}$ to the bundle $\tau$ of metrics, we describe the infinitesimal generators of $\mathfrak{G}_{\tau}$ and determine the number of functionally independent differential invariants of jet orders $0,\,1$ and $2$.  In Section~\ref{sec4} we introduce the metric  $\tilde{\mathbf{g}}$ on the orbit space $\mathcal{S}=\mathcal{M}/G_2$; this metric is referred to as the \textit{orbit metric} throughout the paper. In Section~\ref{sec5} we introduce a maximal set of $6$ generically functionally independent scalar differential invariants $C_{\rho},\, C_{\chi},\,
Q_{\chi},\,Q_{\gamma},\, \ell_{\mathcal{C}},\,(\Theta_{\mathrm{I}})^{2}$ of the first order. In the generic case, when $C_{\rho}$ and $\ell_{\mathcal{C}}$ do not vanish, we also provide a semi-invariant orthogonal frame. 
Moreover, we obtain a number of further first-order scalar differential invariants and discuss their functional dependence on the $6$ independent invariants. In Section~\ref{sec6}, we provide a maximal set of $20$ generically functionally independent scalar differential invariants of the second order. 
In Section~\ref{sect:Lambda vacuum eqs}, we derive the $\Lambda$-vacuum Einstein equations for $\mathrm{G}_{2}$-metrics, and find their solutions in non-generic cases $C_\rho=0$ and $ \ell_{\mathcal{C}}=0$. In the first case we find that all $\Lambda$-vacuum solutions of Einstein equations are \textit{pp}-waves with all first-order invariants identically zero. In the second case we extend to the $\Lambda$-vacuum case the explicit solutions originally presented by Kundu in the case when $\Lambda=0$. In particular we present there two new solutions of  $\Lambda$-vacuum Einstein equations. In Section~\ref{sec8}, we answer the question of how many invariants are functionally independent on solutions to the  $\Lambda$-vacuum Einstein equations.
Finally, in Section \ref{sect:equivalence}, we address the equivalence problem of $\mathrm G_2$-metrics in the generic case.

\section{The pseudogroup and the metric }
\label{sec2}

Let $\mathcal{M}$ be a four-dimensional manifold, endowed with a
two-dimensional Abelian algebra of vector fields $\mathcal{G}_{2}$.
We denote by $\Xi$ the vector distribution generated on $\mathcal{M}$
by vector fields of $\mathcal{G}_{2}$ and by $G_2$ the Lie group of transformations 
generated by  $\mathcal{G}_{2}$ on $\mathcal{M}$. 
Throughout the paper we  assume that the {\it orbit space} 
$\mathcal{S}=\mathcal{M}/G_2$ is a $2$-dimensional manifold,
with $\pi:\mathcal{M}\to\mathcal{S}$ being the natural projection.

One can always choose local coordinates $\{t^{1},t^{2},z^{1},z^{2}\}$
on $\mathcal{M}$ such that: 
\begin{description}
\item [{(1)}] $\mathcal{G}_{2}$ is generated by the coordinate vector
fields $\xi_{(i)}=\partial/\partial z^{i}$, $i=1,2$; 
\item [{(2)}] the leaves of $\Xi$ are the surfaces characterized by the
constancy of $t^{1}$ and $t^{2}$. 
\end{description}
We  refer to such a kind of coordinates $\{t^{1},t^{2},z^{1},z^{2}\}$
as local \textit{adapted coordinates}, and denote by $\mathfrak{G}$
the Lie pseudogroup of adapted coordinates transformations. By
a $\mathfrak{G}$-transformation we mean an element of $\mathfrak{G}$;
by definition $\mathfrak{G}$-transformations are coordinate transformations
$\bar{t}^{\,i}=\bar{t}^{i}(t^{1},t^{2},z^{1},z^{2})$, $\bar{z}^{\,i}=\bar{z}^{i}(t^{1},t^{2},z^{1},z^{2})$
which preserve (1)--(2), i.e., such that $\mathcal{G}_{2}$ is generated
by $\partial/\partial\bar{z}^{i}$, $i=1,2$, and the leaves of $\Xi$
are surfaces characterized by the constancy of $\bar{t}^{1}$ and
$\bar{t}^{2}$. 


\begin{prop}
\label{prop:pseudogroup}The Lie pseudogroup $\mathfrak{G}$ is formed
by transformations $P:\mathcal{M}\rightarrow\mathcal{M}$ which in
adapted coordinates have the form
\begin{equation}
\bar{t}^{\,i}=\phi^{i}(t^{1},t^{2}),\qquad\bar{z}^{\,i}=\alpha_{j}^{i}\,z^{j}+\psi^{i}(t^{1},t^{2}),\label{eq:psgroup_fin}
\end{equation}
where $\phi^{i}(t^{1},t^{2})$ and $\psi^{i}(t^{1},t^{2})$ are arbitrary
differentiable functions satisfying 
\begin{equation}
J_{\phi}=\left|\begin{array}{@{}ll@{}}
\partial_{t^{1}}\phi^{1} & \partial_{t^{2}}\phi^{1}\\
\partial_{t^{1}}\phi^{2} & \partial_{t^{2}}\phi^{2}
\end{array}\right|\neq0,\label{eq:psgroup_fin_nondeg}
\end{equation}
and $\alpha_{j}^{i}\in\mathbb{R}$, with $(\alpha_{j}^{i})\in \mathrm{GL}(2, \mathbb{R})$. 

Infinitesimal generators of $\mathfrak{G}$ have the form 
\begin{equation}
U=\Phi^{i}(t^{1},t^{2})\frac{\partial}{\partial t^{i}}+\left(A_{l}^{k}z^{l}+\Psi^{k}(t^{1},t^{2})\right)\frac{\partial}{\partial z^{k}},\label{eq:ps_group_inf}
\end{equation}
where $\Phi^{i}(t^{1},t^{2})$ and $\Psi^{k}(t^{1},t^{2})$ are arbitrary
differentiable functions and $A_{l}^{k}\in\mathbb{R}$ are arbitrary constants. 

In particular, $\mathfrak{G}$ can be decomposed as 
\begin{equation}
\mathfrak{G}=\mathfrak{G}_{+,+}\cup\mathfrak{G}_{+,-}\cup\mathfrak{G}_{-,+}\cup\mathfrak{G}_{-,-},\label{eq:psgroup_decomp}
\end{equation}
 where $\mathfrak{G}_{\epsilon_{1},\epsilon_{2}}$ are the connected
components, with $\epsilon_{1}=\mathop{\mathrm{sgn}} J_{\phi}$ 
and $\epsilon_{2}=\mathop{\mathrm{sgn}}(\mathop{\mathrm{det}}\alpha_{j}^{i})$. 
\end{prop}


\begin{proof}
Under a $\mathfrak{G}$-transformation $\partial/\partial z^{j}=\alpha_{j}^{i}\,\partial/\partial\bar{z}^{\,i}$,
with $(\alpha_{j}^{i})\in \mathrm{GL}(2, \mathbb{R}).$ On the
other hand, since $\partial/\partial z^{j}=\left(\partial\bar{z}^{\,i}/\partial z^{j}\right)\partial/\partial\overline{z}^{\,i}+\left(\partial\bar{t}^{\,i}/\partial z^{j}\right)\partial/\partial\overline{t}^{\,i}$,
one gets $\partial\bar{z}^{\,i}/\partial z^{j}=\alpha_{j}^{i}$, $\partial\bar{t}^{\,i}/\partial z^{j}=0$.
Hence, the $\mathfrak{G}$-transformations have the required form.

On the other hand, a vector field 
$U=T^{i}(t,z)\,\partial/\partial t^{i}+Z^{k}(t,z)\,\partial/\partial z^{k}$
is an infinitesimal generator of $\mathfrak{G}$ iff $U$ is an infinitesimal
symmetry of the Lie algebra generated by $\xi_{(1)}=\partial/\partial z^{1}$
and $\xi_{(2)}=\partial/\partial z^{2}$. Therefore, 
$\left[\partial/\partial z^{l},U\right]=A_{l}^{k}\,\partial/\partial z^{k}$,
with $A_{l}^{k}$ arbitrary constants. Hence, $\partial T^{i}(t,z)/\partial z^{l}=0$,
$\partial Z^{k}(t,z)/\partial z^{l}=A_{l}^{k}$, and the statement
readily follows. 
\end{proof}
Assume now that $\mathcal{M}$ is endowed with a Riemannian or pseudo-Riemannian
metric $\mathbf{g}$ and that the algebra of Killing vector fields of $\mathbf{g}$ 
is the two-dimensional
Abelian algebra $\mathcal{G}_{2}$, i.e., 
$$
\mathfrak{Kill}(\mathbf{g})=\mathcal{G}_{2}
$$
In particular, there are no Killing vectors outside $\mathcal{G}_{2}$.

The 2-dimensional integral submanifolds of $\Xi$ are called the
\textit{Killing leaves}.
In adapted coordinates, the metric $\mathbf{g}$ takes the form 
\begin{equation}
\mathbf{g}=b_{ij}(t^{1},t^{2})\,dt^{i}\,dt^{j}+2f_{ik}(t^{1},t^{2})\,dt^{i}\,dz^{k}+h_{kl}(t^{1},t^{2})\,dz^{k}\,dz^{l},\label{gg}
\end{equation}
with 
\[
b_{21}=b_{12},\quad h_{21}=h_{12}.
\]
It is worth noting here that, under $\mathfrak{G}$-transformations
(\ref{eq:psgroup_fin}), $\mathbf{g}$ transforms
to 
\begin{equation}\label{ggbis}
\bar{\mathbf{g}}=\bar{b}_{mn}(\bar{t}^{1},\bar{t}^{2})\,d\bar{t}^{m}\,d\bar{t}^{n}+2\bar{f}_{mr}(\bar{t}^{1},\bar{t}^{2})\,d\bar{t}^{m}\,d\bar{z}^{r}+\bar{h}_{rs}(\bar{t}^{1},\bar{t}^{2})\,d\bar{z}^{r}\,d\bar{z}^{s},
\end{equation}
with
\begin{equation}
\begin{array}{l}
b_{ij}=\bar{b}_{mn}{\displaystyle \frac{\partial\phi^{m}}{\partial t^{i}}}{\displaystyle \frac{\partial\phi^{n}}{\partial t^{j}}}+2\bar{f}_{mr}{\displaystyle \frac{\partial\phi^{m}}{\partial t^{i}}}{\displaystyle \frac{\partial\psi^{r}}{\partial t^{j}}}+\bar{h}_{rs}{\displaystyle \frac{\partial\psi^{r}}{\partial t^{i}}\frac{\partial\psi^{s}}{\partial t^{j}}},\vspace{5pt}\\
f_{ik}={\displaystyle \bar{f}_{mr}\alpha_{k}^{r}\frac{\partial\phi^{m}}{\partial t^{i}}+\bar{h}_{rs}\alpha_{k}^{r}\frac{\partial\psi^{s}}{\partial t^{i}}},\vspace{5pt}\\
h_{kl}=\bar{h}_{rs}\alpha_{k}^{r}\alpha_{l}^{s}.
\end{array}
\label{cond_equiv}
\end{equation}
In particular
\begin{equation}
\mathop{\mathrm{det}}\mathbf{\bar{g}}=(\mathop{\mathrm{det}}\alpha_{j}^{i})^{2}\,\left(J_{\phi}\right)^{2}\,\mathop{\mathrm{det}}\mathbf{g}\neq0.\label{eq:psgroup_ametr}
\end{equation}

\begin{prop}
\label{prop:PG acts on tau}The pseudogroup $\mathfrak{G}$ naturally
extends to the bundle of symmetric $(0,2)$-tensor fields on
$\mathcal{M}$ and its action preserves the sub-bundle $\tau:E\rightarrow\mathcal{M}$
of metrics of the form {\rm(\ref{gg})} on $\mathcal{M}$.
\end{prop}

\begin{proof}
See formulas \eqref{ggbis} and \eqref{cond_equiv}.
\end{proof}

The extension of $\mathfrak{G}$
to $\tau$ will be denoted by $\mathfrak{G}_{\tau}$.

\section{Pseudogroup prolongation and differential invariants}
\label{sec3}

In view of Proposition \ref{prop:PG acts on tau}, the classification
problem for metrics with an Abelian $2$-dimensional Killing algebra
$\mathcal{G}_{2}$ reduces to identifying orbits of the action of
$\mathfrak{G}_{\tau}$ on the bundle $\tau:E\rightarrow\mathcal{M}$
of metrics $\mathbf{g}$ of the form (\ref{gg}); indeed these orbits
consist of mutually equivalent metrics.

Following Lie's classical method, the classification problem for these
metrics can be solved by using a sufficient number of independent
scalar differential invariants of $\mathfrak{G}_{\tau}$. These invariants
are defined to be functions on the jet prolongations $J^{m}\tau$,
$m=0,1,2,...$, that are invariant with respect to the action of the
corresponding prolonged pseudogroups $\mathfrak{G}_{\tau}^{(m)}$. 

The problem of finding the $m$-th order scalar differential invariants
becomes linear if written in terms of the infinitesimal action of
$\mathfrak{G}_{\tau}^{(m)}$ on $J^{m}\tau$. This fact is at the
heart of Lie's infinitesimal method of computing differential invariants
and also permits a simple determination of the dimensions $N_{m}$
of the orbit spaces $J^{m}\tau/\mathfrak{G}_{\tau}^{(m)}$ for $m=0,1,2,...$
\begin{prop}
By using the coordinate representation {\rm(\ref{gg})}, the pseudogroup
$\mathfrak{G}_{\tau}$ is infinitesimally generated by vector fields
\begin{equation}
\begin{array}{l}
\displaystyle 
U^{\tau} 
= \Phi^{i} \frac{\partial}{\partial t^{i}}
+ \left(A_{l}^{k}z^{l}+\Psi^{k}\right) \frac{\partial}{\partial z^{k}}
- \left(b_{is} \frac{\partial\Phi^{s}}{\partial t^{j}}
      + f_{is} \frac{\partial\Psi^{s}}{\partial t^{j}}
      + b_{js} \frac{\partial\Phi^{s}}{\partial t^{i}}
      + f_{js} \frac{\partial\Psi^{s}}{\partial t^{i}}\right)
  \frac{\partial}{\partial b_{ij}}
\vspace{8pt}\\
\qquad\displaystyle 
-\left(f_{sk} \frac{\partial\Phi^{s}}{\partial t^{i}}
     + h_{sk} \frac{\partial\Psi^{s}}{\partial t^{i}}
     + f_{is}A_{k}^{s}\right) \frac{\partial}{\partial f_{ik}}
-\left(h_{ks}A_{l}^{s}+h_{sl}A_{k}^{s}\right)
 \frac{\partial}{\partial h_{kl}}.\vspace{8pt}
\end{array}\label{gener_U_tau}
\end{equation}
where $\Phi^{i},\Psi^{i},A_{l}^{k}$ are as in Proposition~{\rm\ref{prop:pseudogroup}},
formula~\eqref{eq:ps_group_inf}.
\end{prop}
\begin{proof}
Since $U^{\tau}$ projects to $U$, it has the form
\[
U^{\tau}=\Phi^{i}(t){\displaystyle \frac{\partial}{\partial t^{i}}}+\left(A_{l}^{k}z^{l}+\Psi^{k}(t)\right){\displaystyle \frac{\partial}{\partial z^{k}}}+B_{ij}\frac{\partial}{\partial b_{ij}}+F_{ik}\frac{\partial}{\partial f_{ik}}+H_{kl}\frac{\partial}{\partial h_{kl}},
\]
with $B_{ij}$, $F_{ik}$ and $H_{kl}$ differentiable functions of
$t^{1}$, $t^{2}$, $z^{1}$, $z^{2}$ and $f_{ij}$, $b_{ij}$, $h_{kl}$.
Then (\ref{gener_U_tau}) follows by imposing the Lie invariance condition 
\[
L_{U^{\tau}}(\mathbf{g})=0.
\]
\vskip-1.6\baselineskip
\end{proof}
Recall that $J^{0}\tau=\tau$ and that the formal derivatives of $b_{ij}$,
$f_{ik}$ and $h_{kl}$ of orders $m=0,1,2,...$ with respect to $t^{1},t^{2}$,
can serve as coordinates along the fibers of $J^{m}\tau$ in an obvious
way. We  denote such coordinates as $b_{ij,\mathbf{I}}$,
$f_{ij,\mathbf{I}}$ and $h_{kl,\mathbf{I}}$, for any symmetric multi-index
$\mathbf{I}$ when $m>1$, and $b_{ij,s}$, $f_{ij,s}$ and $h_{kl,s}$
when $m=1$. 

Prolongation formulas of $U^{\tau}$ to $U^{J^{m}\tau}$ on $J^{m}\tau$,
$m=1,2,...$, are well known (\cite{A-V-L,Ol,Ov}). Alternatively,
from the commutator 
\[
\left[\frac{\partial}{\partial t^{s}},U\right]=\frac{\partial\Phi^{j}}{\partial t^{s}}\frac{\partial}{\partial t^{j}}+\left(\frac{\partial A_{l}^{k}}{\partial t^{s}}z^{l}+\frac{\partial\Psi^{k}}{\partial t^{s}}\right)\frac{\partial}{\partial z^{k}}
\]
valid on the base manifold $\mathcal{S}$ one can infer the relations
$\left[U^{J^{\infty}\tau},D_{s}\right]=-\left(\partial\Phi^{1}/\partial t^{s}\right)D_{1}-\left(\partial\Phi^{2}/\partial t^{s}\right)D_{2}$
on $J^{\infty}\tau$, where 
\[
D_{s}=\frac{\partial}{\partial t^{s}}+b_{ij,\mathbf{I}+\mathbf{1}_{s}}\frac{\partial}{\partial b_{ij,\mathbf{I}}}+f_{ij,\mathbf{I}+\mathbf{1}_{s}}\frac{\partial}{\partial f_{ij,\mathbf{I}}}+h_{kl,\mathbf{I}+\mathbf{1}_{s}}\frac{\partial}{\partial h_{kl,\mathbf{I}}},\quad s=1,2,
\]
denote the usual total derivatives and $\mathbf{I}$ stands for an
arbitrary symmetric multi-index. These relations reflect the way how the action
on metric coefficients extends to the action on derivatives thereof. Thus,
for any symmetric multi-index $\mathbf{I}$ of order $m$, we have
\[
U^{J^{m+1}\tau}\left(b_{ij,\mathbf{I}+\mathbf{1}_{s}}\right)=D_{s}\left(U^{J^{m}\tau}\left(b_{ij,\mathbf{I}}\right)\right)-\frac{\partial\Phi^{1}}{\partial t^{s}}b_{ij,\mathbf{I}+\mathbf{1}_{1}}-\frac{\partial\Phi^{2}}{\partial t^{s}}b_{ij,\mathbf{I}+\mathbf{1}_{2}},
\]
and analogously for $f_{ik}$ and $h_{kl}$.

Now, scalar differential invariants can be identified with functions
on $J^{\infty}\tau$ invariant with respect to the fields $U^{J^{\infty}\tau}$,
for all admissible choices of the coefficients $\Phi^{i},A_{l}^{k},\Psi^{k}$.
These invariants form a commutative associative $\mathbb{R}$-algebra,
which can be thought of as algebra of functions on the orbit space
$J^{\infty}\tau/\mathfrak{G}_{\tau}^{(\infty)}$. 

By using the infinitesimal generators of $\mathfrak{G}_{\tau}^{(m)}$,
obtained for all admissible choices of the coefficients $\Phi^{i}, A_{l}^{k},\Psi^{k}$,
one can determine the dimensions of $\mathfrak{G}_{\tau}^{(m)}$ and
the corresponding (generic) dimensions of the orbit spaces $J^{m}\tau/\mathfrak{G}_{\tau}^{(m)}$.
For $m=0,1,2$ one has the following proposition.

\begin{prop}
\label{Nr} In the generic case, when $\Xi^{\perp}$ is not integrable
{\rm(}the orthogonally intransitive case{\rm)}, the generic dimension $N_{m}$ of
the orbit space $J^{m}\tau/\mathfrak{G}_{\tau}^{(m)}$ for $m=0,1,2$
is provided in the following table: 
\[
\begin{array}{r|rrrrrr}
m & 0 & 1 & 2 & \\
\hline N_{m} & 0 & 6 & 20 & 
\end{array}
\]
In the special case, when $\Xi^{\perp}$ is integrable {\rm(}orthogonally
transitive case{\rm),} the generic dimension $N_{m}$ of the orbit space
$J^{m}\tau/\mathfrak{G}_{\tau}^{(m)}$ for $m=0,1,2$ is provided
in the following table: 
\[
\begin{array}{r|rrrrrr}
m & 0 & 1 & 2 & \\
\hline N_{m} & 0 & 4 & 14 & 
\end{array}
\]
\end{prop}

It is worth noting here that the generic dimensions $N_{m}$ refer
to non-singular strata of the orbit space $J^{\infty}\tau/\mathfrak{G}_{\tau}^{(\infty)}$,
and hence there may exist singular strata of lower dimension where
the maximum number of functionally independent scalar differential
invariants is lower than the generic value $N_{m}$.
\begin{rem}
\label{not} A comment on a possible source of misunderstanding is
due. Proposition~\ref{Nr} refers to scalar differential invariants
as functions on the jet space $J^{\infty}\tau$. If such a function,
say $F$, is evaluated for a particular metric $\mathbf{g}$, then
it becomes a function on the orbit space $\mathcal{S}$, which we shall denote as
$F|_{\mathbf{g}}$ (formally $F|_{\mathbf{g}}=F\circ j^{\infty}\sigma_{\mathbf{g}}$,
where $j^{\infty}$ denotes a jet prolongation of a section of the
bundle $\tau$ and $\sigma_{\mathbf{g}}$ is the section associated
with $\mathbf{g}$). Analogous correspondences hold for other geometric
objects such as forms and vector fields. Hence another interpretation
of scalar differential invariants as functions on $\mathcal{S}$.
Both interpretations are natural and important. For instance, the
order of an invariant can only be seen in the context of jet spaces,
while the most natural way to construct an invariant consists in combining
various invariant geometric constructions on $\mathcal{S}$ \cite{A-V-L}.
It is usually harmless to use one and the same notation with both
interpretations and omit the symbol $|_{\mathbf{g}}$. However, one
should bear in mind that independence of functions on $\mathcal{S}$
is very different from that on $J^{\infty}\tau$. The maximal number
of independent functions is two on $\mathcal{S}$, and unlimited on
$J^{\infty}\tau$.
\end{rem}

\section{Orbit metric\label{subsec:Orbit-metric}}
\label{sec4}

The restriction of $\mathbf{g}$ to the orbits of the Killing algebra
$\mathcal{G}_{2}$, generated by $\xi_{(1)}=\partial_{z^{1}}$, $\xi_{(2)}=\partial_{z^{2}}$,
is described by the $2\times2$ symmetric matrix $H=(h_{ij})$ with
elements 
\[
h_{kl}=\mathbf{g}(\xi_{(k)},\xi_{(l)})=\mathbf{g}_{ab}\xi_{(k)}^{a}\xi_{(l)}^{b}.
\]
In view of assumption (ii) (see Introduction), $\det (h_{ij}) \ne 0$ everywhere, since otherwise 
the restriction
of $\mathbf{g}$ to orbits would be degenerate at some point. 

We found more convenient to rewrite the metric in the form
\begin{equation}
\mbox{\ensuremath{\mathbf{g}}}=\tilde{g}_{ij}\,dt^{i}\,dt^{j}+h_{kl}(dz^{k}+f_{i}^{k}\,dt^{i})(dz^{l}+f_{j}^{l}\,dt^{j}),\label{eq:gG}
\end{equation}
where 
\begin{equation}
\tilde{g}_{ij}=b_{ij}-f_{ik}f_{jl}h^{kl},\qquad f_{j}^{k}=f_{js}h^{sk},\label{eq:gg-gG}
\end{equation}
and $h^{kl}$ denote the elements of the inverse matrix $H^{-1}$. 
Notice that relations (\ref{eq:gg-gG}) directly connect components of (\ref{gg}) to those of (\ref{eq:gG}). 

In terms of variables $\tilde{g}_{ij},\,f_{i}^{k},h_{kl}$, expression 
(\ref{gener_U_tau}) for $U^{\tau}$ simplifies to
\begin{equation}
\begin{array}{l}
\displaystyle 
U^{\tau} 
= \Phi^{i} \frac{\partial}{\partial t^{i}}
+ (A_{l}^{k}z^{l} + \Psi^{k}) \frac{\partial}{\partial z^{k}}
- \left(\tilde g_{is} \frac{\partial\Phi^{s}}{\partial t^{j}}
      + \tilde g_{js} \frac{\partial\Phi^{s}}{\partial t^{i}}\right)
  \frac{\partial}{\partial \tilde g_{ij}}
\vspace{8pt}\\
\qquad\displaystyle 
+ \left(f_{i}^{s} A_{s}^{k}
     - \frac{\partial\Psi^{k}}{\partial t^{i}}
     - f_{s}^{k} \frac{\partial\Phi^{s}}{\partial t^{i}}
     \right) \frac{\partial}{\partial f_{i}^{k}}
- (A_{l}^{s} h_{ks} + A_{k}^{s} h_{sl})
 \frac{\partial}{\partial h_{kl}}.\vspace{8pt}
\end{array}\label{gener_U_tau new}
\end{equation}

An important advantage of (\ref{eq:gG}) is that
$\tilde{\mathbf{g}}=\tilde{g}_{ij}\,dt^{i}\,dt^{j}$ defines a natural
metric on the orbit space $\mathcal{S}$ such that
\[
\tilde{\mathbf{g}}(X,Y)=\mathbf{g}(X,Y)-h^{kl}\mathbf{g}(\xi_{(k)},X)\mathbf{g}(\xi_{(l)},Y),
\]
for any pair of vector fields $X,Y$ on $\mathcal{S}$.  

\begin{prop}[Geroch~\cite{Ger1}]
The $(0,2)$-tensor field $\tilde{\mathbf{g}}$ defines a metric tensor
on the orbit space $\mathcal{S}=\mathcal{M}/G_2$.
\end{prop}
\begin{proof}
The components of $\tilde{\mathbf{g}}$ only depend on $(t^{1},t^{2})$
and, since $\tilde{\mathbf{g}}(\xi_{(i)},\text{--})=0$ for any $i=1,2$,
we have $\tilde{\mathbf{g}}(X,\text{--})=0$ for every vector
field $X\in\Xi$. It follows that $\tilde{\mathbf{g}}$ is
a well-defined $(0,2)$ tensor on the two-dimensional orbit space
$\mathcal{S}=\mathcal{M}/G_{2}$, and 
\[
\tilde{g}_{ij}=b_{ij}-f_{ik}f_{jl}h^{kl},\quad i,j=1,\dots,2.
\]
Moreover, it is easily checked that 
\[
\det\tilde{\mathbf g} = \frac{\det\mathbf{g}}{\det \mathbf h}.
\]
Hence, $\tilde{\mathbf{g}}$ is nondegenerate and defines a metric
on $\mathcal{S}$.
\end{proof}
Another reason why we prefer (\ref{eq:gG}) to  (\ref{gg}) is that 
explicit expressions of differential invariants of 
$\mathbf{g}$ are relatively simple in terms of $\tilde{g}_{ij},\,f_{i}^{k},h_{kl}$,
whereas they swell in $b_{ij},\,f_{ik},h_{kl}$.

\section{First-order invariants}
\label{sec5}

According to Proposition \ref{Nr}, on $J^{1}\tau$ there are at most
$4$ functionally independent scalar invariants in the orthogonal
transitive case (when $\Xi^{\perp}$ is integrable), and at most $6$ 
functionally independent invariants in the orthogonally intransitive case
(when $\Xi^{\perp}$ is not integrable).
Such a maximal system of functionally independent invariants generates
the whole algebra of differential invariants of the first order, since
any first-order scalar differential invariant must be a function of
them. In this section we  provide an explicit construction of
a maximal system of $6$ functionally independent scalar invariants
for the orthogonally intransitive case, which extends the already 
known~\cite{M-S} maximal system
$I_{1}=C_{\rho}$, $I_{2}=C_{\chi}$, $I_{3}=Q_{\chi}$,
$I_{4}=Q_{\gamma}$ of invariants for the orthogonally transitive
case.
As a matter of fact, we obtain and explore mutual dependence of a number 
of additional first-order invariants which vanish when $\Xi^{\perp}$ is integrable,. 



\subsection{\label{subsec:Invariants C and Q}Scalar invariants {\normalsize{}$C_{\rho}$,
$C_{\chi}$, $Q_{\chi}$, $Q_{\gamma}$} and the semi-invariant orthogonal 
frame $\{ \mathcal{X}, \mathcal{X}^\perp \}$ on $\mathcal S$}

The first-order invariants presented in this subsection essentially coincide with those 
of~\cite{M-S} except that the metric coefficients $g_{ij}$ of~\cite{M-S} have been replaced with the
coefficients of the orbit space metric $\tilde{\mathbf g}_{ij}$.

\begin{lem}
\label{lem:sigma}For any metric $\mathbf{g}$ of the form {\rm(\ref{eq:gG}),}
the pseudogroup action leaves invariant the $1$-form 
\[
\sigma = d \ln (\det\mathbf h) = \frac{d(\det\mathbf h)}{\det\mathbf h}
\]
 and the symmetric $(0,2)$-tensors
\[
\qquad\rho=\sigma^{2},\qquad\chi=\frac{1}{(\det\mathbf h)}\left(dh_{11}\,dh_{22}-dh_{12}\,dh_{12}\right),\qquad\gamma=\chi-\tfrac{1}{4}\rho.
\]
\end{lem}
\begin{proof} 
Under pseudogroup transformations of $\mathfrak{G}_{\tau}$, $\det\mathbf h$
transforms as $\det\mathbf h\mapsto (\det\mathbf h)/(\mathop{\mathrm{det}}\alpha_{j}^{i})^{2}$,
with $(\alpha_{j}^{i})\in \mathrm{GL}(2, \mathbb{R})$. Therefore,
the $1$-form $\sigma={d(\det\mathbf h)}/{\det\mathbf h}$ is $\mathfrak{G}_{\tau}^{(1)}$-invariant.
The invariance of $\chi$ follows from the transformation
rule $\left(dh_{11}\,dh_{22}-dh_{12}\,dh_{12}\right)$ $\mapsto\left(dh_{11}\,dh_{22}-dh_{12}\,dh_{12}\right)/(\mathop{\mathrm{det}}\alpha_{j}^{i})^{2}$ under the pseudogroup
transformations of $\mathfrak{G}_{\tau}$. 
\end{proof}
\begin{rem}
We shall call $\gamma$ the \textit{Cosgrove form}, since it was 
introduced in the paper \cite[Eq.~(2.3)]{C1} in the orthogonally transitive case.
\end{rem}

An easy construction of the first-order scalar differential invariants follows from the 
consideration of the determinant $Q_{\mu}$ and the trace $C_{\mu}$ of the self-adjoint 
$(1,1)$-tensor field related to a symmetric bilinear form $\mu$ on $\mathcal{S}$.
In coordinates,
if $\mu=\mu_{ij}\,dt^{i}\,dt^{j}$, then the corresponding $(1,1)$-tensor field 
has the components $\mu_{i}^{j} = \mu_{is}\tilde{\mathbf g}^{sj}$, and 
\begin{equation}
Q_{\mu}=\frac{\mathop{\mathrm{det}}\mu_{ij}}{\det\tilde{\mathbf g}},\qquad C_{\mu}=\mu_{ij}\tilde{g}^{ij}.\label{eq:inv_C_Q}
\end{equation}

Choosing $\mu = \rho,\chi,\gamma$, we get four independent invariants
$C_{\rho},C_{\chi},Q_{\chi},Q_\gamma$, whereas $Q_\rho = 0$ and $C_{\gamma}=C_{\chi}-\frac{1}{4}C_{\rho}$.

Geometric meaning of $C_\chi$ is given in Proposition~\ref{GaussCurv}.

\begin{prop}
\label{prop:first_inv}
The functions $C_{\rho},C_{\chi},Q_{\chi}$ and 
$Q_{\gamma}$ are generically functionally independent first-order differential
invariants.
\end{prop}

\begin{proof}
Invariance follows from Lemma~\ref{lem:sigma}.
Functional independence follows from the fact that the rank of the Jacobian at a generic point 
of the jet space is equal to 4. 
Obviously, the last condition is easily checked by computing the rank of a numeric matrix.
\end{proof}

In coordinates, we have
\begin{enumerate}
\item $C_{\rho} 
 = {\displaystyle \frac{1}{(\det\mathbf h)^{2}}}(\det\mathbf h)_{,i}(\det\mathbf h)_{,j}\,\tilde{g}^{ij}$;
\item $C_{\chi}={\displaystyle {\displaystyle \frac{1}{(\det\mathbf h)}}}\,\tilde{g}^{ij}\left|\begin{array}{@{}cc@{}}
h_{11,i} & h_{12,j}\\
h_{12,i} & h_{22,j}
\end{array}\right|$;
\item $Q_{\chi}={\displaystyle \frac{\mathop{\mathrm{det}} \chi_{ij}}{\det\tilde{\mathbf g}}},$
\quad $\chi_{ij}={\displaystyle \frac{1}{2\,\det\mathbf h}}\left|\begin{array}{@{}cc@{}}
h_{11,i} & h_{12,j}\\
h_{21,i} & h_{22,j}
\end{array}\right|+{\displaystyle \frac{1}{2\,\det\mathbf h}}\left|\begin{array}{@{}cc@{}}
h_{11,j} & h_{12,i}\\
h_{21,j} & h_{22,i}
\end{array}\right|$;
\item $Q_{\gamma}={\displaystyle \frac{\mathop{\mathrm{det}}\gamma_{ij}}{\det\tilde{\mathbf g}}}={\displaystyle \frac{1}{4(\det\mathbf h)^{3}\det\tilde{\mathbf g}}}\left|\begin{array}{@{}ccc@{}}
h_{11} & h_{12} & h_{22}\\
h_{11,1} & h_{12,1} & h_{22,1}\\
h_{11,2} & h_{12,2} & h_{22,2}
\end{array}\right|^{2}$.
\end{enumerate}
Here comma denotes partial differentiation.


Following~\cite{M-S} again, we complete this section with a construction of two invariant first-order 
vector fields on $\mathcal S$.
In the jet space description, the $1$-form $\sigma$ 
is defined on $J^{1}\tau$
and horizontal with respect to $\text{\ensuremath{\pi}}_{1}:=\pi\circ\tau_{1}$.
We  denote by $\mathcal{X}$ and $\mathcal{X^\bot}$
the $\pi_{1}$-relative vector fields on $\mathcal{S}$ such that
$\sigma=\tilde{\mathbf{g}}(\mathcal{X},\text{--})$ and 
$\sigma=\mathcal{X}^\bot\mathop\lrcorner\mathop{\mathrm{vol}}_{\tilde{\mathbf{g}}}$,
respectively. 
Here 
$$
\mathop{\mathrm{vol}}\nolimits_{\tilde{\mathbf{g}}}
 =\sqrt{\left|\det \tilde{\mathbf{g}}\right|}\, dt^1 \wedge dt^2
$$
is the volume form of $(\mathcal S, \tilde{\mathbf{g}})$.

\begin{lem} \label{lem:XXbot}
Under the pseudogroup action, for any metric $\mathbf{g}$ of the form {\rm(\ref{eq:gG})},  the vector field
\begin{equation}
\mathcal{X}
 = \tilde{\mathbf{g}}^{is} \frac{(\det\mathbf h)_{,s}}{\det\mathbf h} \partial_{t^{i}}, \label{eq:X}
\end{equation}
is invariant, whereas the vector field
\begin{equation}
\mathcal{X}^\bot = \frac{(\det\mathbf h)_{,2}}{(\det\mathbf h)\sqrt{\left|\det\tilde{\mathbf g}\right|}} 
  \partial_{t^{1}}-\frac{(\det\mathbf h)_{,1}}{(\det\mathbf h)\sqrt{\left|\det\tilde{\mathbf g}\right|}} 
  \partial_{t^{2}}
  \label{eq:Xbot}
\end{equation}
 transforms as $\mathcal{X}^\bot  \mapsto \mathop{\mathrm{sgn}}(J_{\phi})\mathcal{X}^\bot$. Moreover
\begin{equation}
\label{eq:XXbot orthogonality}
\tilde{\mathbf{g}}(\mathcal{X},\mathcal{X}) = C_\rho, \quad
\tilde{\mathbf{g}}(\mathcal{X},\mathcal{X}^\bot) = 0, \quad
\tilde{\mathbf{g}}(\mathcal{X}^\bot,\mathcal{X}^\bot) =\pm_{\tilde{\mathbf{g}}} C_\rho.
\end{equation}
where $\pm_{\tilde{\mathbf{g}}} = \mathop{\mathrm{sgn}} \det\tilde{\mathbf{g}}$. Hence  $\{ \mathcal{X}, \mathcal{X}^\perp \}$ is a semi-invariant orthogonal  frame  on $\mathcal S$, when $C_{\rho}\neq 0$.

\end{lem}

\begin{proof}
In view of the invariance of $\sigma$ and the fact that $\mathop{\mathrm{vol}}_{\tilde{\mathbf{g}}}$
is invariant only up to a sign, $\mathcal{X}$ and $\mathcal{X}^\bot$ have the specified invariance 
properties.
Formulas~\eqref{eq:XXbot orthogonality} are routinely checked in adapted coordinates.
\end{proof}



\subsection{The semi-invariant vector field $\mathcal{C}$ and scalar invariant $\ell_{\mathcal{C}}$ }

The mapping $\pi:\mathcal{M}\to\mathcal{S}$ is a Riemannian submersion
\cite[9.12]{Bess}, with respect to metrics $\mathbf{g}$ and $\tilde{\mathbf{g}}$.
Relatively to this submersion, $\Xi$ will be referred to as the \textit{vertical
distribution}, whereas $\Xi^{\perp}$ as the\textit{ horizontal} \textit{distribution}.
Moreover, due to non-degeneracy condition (ii), the tangent bundle
to $\mathcal{M}$ decomposes as $T\mathcal{M}=\Xi\oplus\Xi^{\perp}$,
with $\Xi^{\perp}$ generated by the vector fields
\begin{equation}\label{eq:e}
\mathbf{e}_{j}=\frac{\partial}{\partial t^{j}}-f_{j}^{k}\frac{\partial}{\partial z^{k}},\qquad j=1,2.
\end{equation}
Recall that $f_{j}^{k}=f_{js}h^{sk}$.

In view of this decomposition one has the natural projections
$\mathop{\mathrm{ver}}=\mbox{pr}_{\Xi}:TM\rightarrow\Xi$ and $\mathop{\mathrm{hor}}=\mbox{pr}_{\Xi^{\perp}}:TM\rightarrow\Xi^{\perp}$
such that 
\[
\mathop{\mathrm{ver}}\left(\frac{\partial}{\partial t^{j}}\right)=\mathop{\mathrm{ver}}\left(\frac{\partial}{\partial t^{j}}-f_{j}^{k}\frac{\partial}{\partial z^{k}}+f_{j}^{k}\frac{\partial}{\partial z^{k}}\right)=f_{j}^{k}\frac{\partial}{\partial z^{k}},\qquad\mathop{\mathrm{ver}}\left(\frac{\partial}{\partial z^{k}}\right)=\frac{\partial}{\partial z^{k}},
\]
and 
\[
\mathop{\mathrm{hor}}\left(\frac{\partial}{\partial t^{j}}\right)=\mathop{\mathrm{hor}}\left(\frac{\partial}{\partial t^{j}}-f_{j}^{k}\frac{\partial}{\partial z^{k}}+f_{j}^{k}\frac{\partial}{\partial z^{k}}\right)=\frac{\partial}{\partial t^{j}}-f_{j}^{k}\frac{\partial}{\partial z^{k}},\qquad\mathop{\mathrm{hor}}\left(\frac{\partial}{\partial z^{k}}\right)=0.
\]
In the adapted coordinates the non-vanishing components of $\mathop{\mathrm{ver}}$
and $\mathop{\mathrm{hor}}$ are 
\[
\mathop{\mathrm{ver}}{}_{j}^{k^{*}}=f_{j}^{k},\quad\mathop{\mathrm{ver}}{}_{j^{*}}^{k^{*}}=\delta_{j}^{k},\qquad j,k=1,2,
\]
and 
\[
\mathop{\mathrm{hor}}{}_{j}^{k}=\delta_{j}^{k},\quad\mathop{\mathrm{hor}}{}_{j}^{k^{*}}=-f_{j}^{k},\qquad j,k=1,2,
\]
respectively, where we use the notation $k^{*}=k+2$.
\begin{rem}
Every (relative) vector field $X$ on $\mathcal{S}$ can be uniquely
lifted to an horizontal (relative) vector field $\widehat{X}$ on $\mathcal{M}$
which is $\pi$-related to $X$. In particular every invariant (relative)
vector field on $\mathcal{S}$ can be uniquely lifted to an invariant
(relative) vector field on $\mathcal{M}$. Moreover, the lift preserves the scalar product.
In coordinates, 
$$
\widehat{\frac{\partial}{\partial t^{i}}} = \frac{\partial}{\partial t^{i}}-f_{i}^{k}\frac{\partial}{\partial z^{k}} = \mathbf{e}_i,
\quad i = 1,2. 
$$

\end{rem}


The geometry of the Riemannian submersion $\pi:\mathcal{M}\to\mathcal{S}$
can be described by using the Ehresmann curvature and the O'Neill
tensors, which are naturally defined in terms of $\mathop{\mathrm{ver}}$
and $\mathop{\mathrm{hor}}$. 

The \textit{Ehresmann curvature }is the tensor $\mathbf{c}:\mathcal{D}(\mathcal{M})\otimes\mathcal{D}(\mathcal{M})\to\mathcal{D}(\mathcal{M})$
defined in terms of the Lie bracket by 
\[
\mathbf{\mathbf{c}}(W_{1},W_{2})=\mathop{\mathrm{ver}}\left[\mathop{\mathrm{hor}}\,W_{1},\mathop{\mathrm{hor}}\,W_{2}\right],
\]
for any two vector fields $W_{1},W_{2}\in\mathcal{D}(\mathcal{M})$.
This is an antisymmetric tensor whose nonzero components in adapted
coordinates are 
\begin{equation}
\mathbf{c}_{ij}^{k^{*}}=\partial_{j}f_{i}^{k}-\partial_{i}f_{j}^{k},\label{EqRef: curv c}
\end{equation}
\textcolor{black}{where} $k^{*}=k+2$.
\textcolor{black}{It is easily checked that $\mathbf{c}$ is traceless,
$\mathbf{c}_{ab}^{a}=0$. }Of course, $\mathbf{c}=0$ if and only
if $\Xi^{\perp}$ is involutive. 


The \textit{curvature vector field} \textit{$\mathcal{C}$}
is defined as 
\begin{equation}
\mathcal{C}=\frac{\mathbf{c}(\partial_{t_{1}},\partial_{t_{2}})}
  {\sqrt{\left|\det\tilde{\mathbf g}\right|}}. \label{curv}
\end{equation}
This is a semi-invariant vector field, since it transforms
as $\mathcal{C}\mapsto(\mathop{\mathrm{sgn}}J_{\phi})\mathcal{C}$ 
under pseudogroup transformations (\ref{eq:psgroup_fin}).
Indeed, the numerator and denominator of (\ref{curv}) transform as
$\mathbf{c}(\partial_{t_{1}},\partial_{t_{2}})
 = \mathop{\mathrm{ver}}
  \left[\mathop{\mathrm{hor}}\partial_{t_{1}},
    \mathop{\mathrm{hor}}\partial_{t_{2}}\right]
\mapsto\mathop{\mathrm{ver}} 
  \left[\mathop{\mathrm{hor}}\partial_{t_{1}},
    \mathop{\mathrm{hor}}\partial_{t_{2}}\right]/J_{\phi}$
and $\sqrt{\left|\det\tilde{\mathbf g}\right|}
\mapsto\sqrt{\left|\det\tilde{\mathbf g}\right|}/|J_{\phi}|$,
respectively. 
In coordinates,  
\begin{equation} \label{curv-vect}
\mathcal{C}=\mathcal{C}^{k}\frac\partial{\partial z^{k}},
\quad
\mathcal{C}^{k}=\frac{\partial_{t^{2}}f_{1}^{k}-\partial_{t^{1}}f_{2}^{k}}{\sqrt{\left|\det\tilde{\mathbf g}\right|}},\qquad k=1,2.
\end{equation}
Consider now the scalar invariant $\ell_{\mathcal{C}}=\mathbf{g}(\mathcal{C},\mathcal{C})$, i.e.,  
the squared length of $\mathcal{C}$.
Obviously from the coordinate formulas, $\ell_{\mathcal{C}}$ is 
given by a rather simple coordinate formula
$$
\ell_{\mathcal{C}} = \mathbf{g}(\mathcal{C},\mathcal{C})
 = h_{kl} \mathcal{C}^{k}\mathcal{C}^{l}
 = \frac{h_{kl} (\partial_{t^{2}}f_{1}^{k}-\partial_{t^{1}}f_{2}^{k})
                (\partial_{t^{2}}f_{1}^{l}-\partial_{t^{1}}f_{2}^{l})}
        {\left|\det\tilde{\mathbf g} \right|}.
$$

In the generic case, $\ell_{\mathcal{C}}$ is
functionally independent from the previous four invariants 
$C_{\rho},C_{\chi},C_{\gamma},Q_{\chi}$.
Consequently, $\ell_{\mathcal{C}}$ is the fifth scalar invariant sought.
Summarizing, we have the following, proposition.

\begin{lem}
For any metric $\mathbf{g}$ of the form {\rm(\ref{eq:gG})}, the curvature vector field $\mathcal{C}$ 
transforms as $\mathcal{C}\mapsto(\mathop{\mathrm{sgn}}J_{\phi})\mathcal{C}$ under the pseudogroup 
action~\eqref{eq:psgroup_fin}.
Therefore, $\ell_{\mathcal{C}} = \mathbf{g}(\mathcal{C},\mathcal{C})$ is a scalar differential invariant.
\end{lem}

We say that a metric belongs to the {\it Kundu class} when $\ell_{\mathcal{C}} \equiv 0$,
i.e., when $\mathcal{C}$ is null.
Vacuum Einstein metrics in this class have been studied by Kundu~\cite{Kun}. 


\subsection{O'Neill tensors $\mathbf A$ and $\mathbf T$. The invariant and semi-invariant vector fields $\mathcal{H}$ and $\mathcal{H}^\perp$ in the case when $C_\rho \neq 0$}

To construct further invariants, 
we  introduce also a semi-invariant orthogonal frame on $\Xi^\bot$ by employing the \textit{O'Neill tensors}  $\mathbf{A}$ and $\mathbf{T}$~\cite{ON,Bess}. These tensors are defined by 
\[
\mathbf{A}(W_{1},W_{2})=\mathbf{O}(W_{1},W_{2})+\mathbf{E}(W_{1},W_{2}),
\]
\[
\mathbf{T}(W_{1},W_{2})=\mathbf{N}(W_{1},W_{2})+\mathbf{L}(W_{1},W_{2}),
\]
where
\[
\mathbf{O}(W_{1},W_{2})=\mathop{\mathrm{ver}}\left(\nabla_{\mathop{\mathrm{hor}}\,W_{1}}\mathop{\mathrm{hor}}\,W_{2}\right),\quad\mathbf{E}(W_{1},W_{2})=\mathop{\mathrm{hor}}\left(\nabla_{\mathop{\mathrm{hor}}\,W_{1}}\mathop{\mathrm{ver}}\,W_{2}\right),
\]
\[
\mathbf{N}(W_{1},W_{2})=\mathop{\mathrm{ver}}\left(\nabla_{\mathop{\mathrm{ver}}\,W_{1}}\mathop{\mathrm{hor}}\,W_{2}\right),\quad\mathbf{L}(W_{1},W_{2})=\mathop{\mathrm{hor}}\left(\nabla_{\mathop{\mathrm{ver}}\,W_{1}}\mathop{\mathrm{ver}}\,W_{2}\right),
\]
for arbitrary vector fields $W_{1},W_{2}$ on $\mathcal{M}$.

As is well known, see \cite[\S 9.24]{Bess}, 
\[
\mathbf{A}\left(\mathop{\mathrm{hor}}\,W_{1},\mathop{\mathrm{hor}}\,W_{2}\right)=\tfrac{1}{2}\mathbf{c}(W_{1},W_{2}),
\]
while $\mathbf{A}\left(\mathop{\mathrm{hor}}\,W_{1},\mathop{\mathrm{hor}}\,W_{2}\right)=\mathbf{A}\left(W_{1},\mathop{\mathrm{hor}}\,W_{2}\right)=\mathbf{O}(W_{1},W_{2})$,
hence 
\[
\mathbf{O}(W_{1},W_{2})=\mathbf{A}\left(W_{1},\mathop{\mathrm{hor}}\,W_{2}\right)=\tfrac{1}{2}\mathbf{c}(W_{1},W_{2})
\]
meaning that in adapted coordinates components of $\mathbf{O}$ are
simply one half of those of $\mathbf{c}$ given by formulas (\ref{EqRef: curv c}).
\begin{rem}
The second fundamental form of the fibers of the submersion is defined
by $\left.\bar{\mathbf T}\right|_{\Xi}$, the restriction of $T$ to $\Xi$.
Hence, 
$\left.\bar{\mathbf T}\right|_{\Xi}=0$
if and only if the Riemannian submersion has totally geodesic fibers. Moreover,
$\Xi^{\perp}$ is completely integrable iff the restriction of $\mathbf A$
to $\Xi^{\perp}$ identically vanishes. In particular, when $\mathbf A=0$, 
then $\Xi^{\perp}$ is completely integrable.
\end{rem}



To construct a semi-invariant orthogonal frame in $\Xi^\bot$, we consider the \textit{mean-curvature vector field} $\mathcal{H}$,
defined as
\[
\mathcal{H}=\sum_{s=1}^{2} \mathbf{T}(\mathbf{v}_{s},\mathbf{v}_{s}),
\]
for any vertical orthonormal frame $\{\mathbf{v}_{1},\mathbf{v}_{2}\}$.
Obviously, $\mathcal{H}$ is invariant with respect
to the action of $\mathfrak{G}_{\tau}^{(1)}$. 
In adapted coordinates, $\mathcal{H}$ is the contraction 
$\mathcal{H}^a = \mathbf{g}^{kl}\mathbf{T}_{kl}^{a}$. Hence,
\[
\mathcal{H}=\mathcal{H}^{i}\left(\frac{\partial}{\partial t^{i}}-f_{i}^{k}\frac{\partial}{\partial z^{k}}\right), 
\]
where 
\begin{equation}
\mathcal{H}^{i} = -\frac12 \tilde{\mathbf{g}}^{is}\frac{(\det\mathbf h)_{,s}}{\det\mathbf h},
\quad i = 1,2.\label{H12}
\end{equation}
By comparing \eqref{eq:X} and (\ref{H12}), one sees
that $\mathcal{H} \in \Xi^\bot$ is a lifted vector field; 
more precisely, $\mathcal{H} = -\frac12 \hat{\mathcal{X}}$.
The squared length of $\mathcal{H}$ is easily seen to be 
$\ell_{\mathcal{H}} = \mathbf{g}(\mathcal{H},\mathcal{H})
 = \frac14 C_\rho$.
 

In the case when $C_\rho \neq 0$, to complete the sought semi-invariant orthogonal frame, we introduce the orthogonal complement 
$\mathcal{H}^\bot \in \Xi^\bot$ by lifting the vector field $-\frac12 \mathcal{X}^\bot$,
see formula \eqref{eq:Xbot}.
Then, since the lift preserves the scalar product, one has $\mathbf{g}(\mathcal{H},\mathcal{H}^\bot) = 0$ and 
\[\ell_{\mathcal{H}^\bot} = \mathbf{g}(\mathcal{H}^\bot,\mathcal{H}^\bot)
 =\mathbf{g}\left(-\tfrac{1}{2}\mathcal X^\bot,-\tfrac{1}{2}\mathcal X^\bot\right) 
 = \pm_{\tilde{\mathbf{g}}}\,\tfrac{1}{4}\,C_{\rho}=\pm_{\tilde{\mathbf{g}}}\, \ell_{\mathcal{H}}.
\]
In coordinates,
\[
\mathcal{H^\bot}=(\mathcal{H}^{\bot})^{i}\left(\frac{\partial}{\partial t^{i}}-f_{i}^{k}\frac{\partial}
  {\partial z^{k}}\right), 
\]
where
$$
(\mathcal{H}^{\bot})^{1}
 = -\frac12 \frac{(\det\mathbf h)_{,2}}{(\det\mathbf h) \sqrt{\left|\det\tilde{\mathbf g}\right|}},
\quad
(\mathcal{H}^{\bot})^{2}
 = \frac12 \frac{(\det\mathbf h)_{,1}}{(\det\mathbf h) \sqrt{\left|\det\tilde{\mathbf g}\right|}}.
$$
The pair $\mathcal{H},\mathcal{H}^{\bot}$ is the sought semi-invariant orthogonal frame in $\Xi^\bot$
when $C_{\rho}\neq 0$.

\subsection{The semi-invariant orthogonal frame $\{ \mathcal{H}, \mathcal{H}^\perp, \mathcal{C}, \mathcal{C}^\perp \}$ in the case when $C_\rho\,\ell_\mathcal{C}\neq 0 $}

In this sub-section we consider the case $\ell_\mathcal{C}\neq 0$ and construct a semi-invariant orthogonal frame on $\mathcal{M}$.


By construction, $\mathcal C \in \Xi$, where $\Xi$ is two-dimensional. 
Let $\mathcal{C}^\bot$ be the orthogonal complement of $\mathcal{C}$ in $\Xi$,
uniquely determined by the requirements 
$\mathbf{g}(\mathcal{C},\mathcal{C}^\bot) = 0$, $\mathop{\mathrm{vol}}\nolimits_{\mathbf{h}}(\mathcal{C},\mathcal{C}^\bot) > 0$,
$\ell_{\mathcal{C}^\bot} =\mathbf{g}(\mathcal{C}^\bot,\mathcal{C}^\bot)
 =\pm_{\mathbf{h}} \ell_{\mathcal{C}}$, where $\pm_{\mathbf{h}} = \mathop{\mathrm{sgn}} \det \mathbf{h}$, and 
 $$
\mathop{\mathrm{vol}}\nolimits_{\mathbf{h}}
 = \sqrt{\left| \det\mathbf h \right|}\,dz^1 \wedge dz^2
$$
is the $(t^1,t^2)$-dependent volume form of the orbits with metric $\mathbf{h}=h_{ij}\,dz^i\,dz^j$.

In coordinates, 
\begin{equation}
\label{eq:Cbot}
\mathcal{C}^\bot=\mathcal{C}^{\bot k}\frac\partial{\partial z^{k}},
\quad
\mathcal{C}^\bot{}^{1} = \frac{ h_{s2}\mathcal{C}^{s}}
{\sqrt{\left| \det\mathbf h \right|}},
\quad
\mathcal{C}^\bot{}^{2} = \frac{- h_{s1}\mathcal{C}^{s}}
{\sqrt{\left|\det\mathbf h\right|}}.
\end{equation}

The vector field $\mathcal{C}^\bot$ is semi-invariant,  since it transforms as $\mathcal{C}^\bot \mapsto (\mathop{\mathrm{sgn}}J_{\phi}) (\mathop{\mathrm{sgn}}\mathop{\mathrm{det}}\alpha^{i}_{j}) \mathcal{C}^\bot$ 
under pseudogroup transformations (\ref{eq:psgroup_fin}). Hence, when $\ell_\mathcal{C}\neq 0 $, the pair $\mathcal{C},\mathcal{C}^\bot$ defines a semi-invariant orthogonal frame on $\Xi$.


The following proposition summarizes the above results about the semi-invariant frame 
$\{ \mathcal{H}, \mathcal{H}^\bot$, $\mathcal{C},\mathcal{C}^\bot \}$.

\begin{prop} \label{g_frame}
 In the case when $C_\rho\,\ell_\mathcal{C}\neq 0$, the pairs of vector fields $\mathcal{H},\,\mathcal{H}^\bot \in \Xi^\bot$ and $\mathcal{C},\,\mathcal{C}^\bot \in \Xi$ form a semi-invariant orthogonal frame on $\mathcal{M}$. In particular, under the pseudo-group action \eqref{eq:psgroup_fin}, these fields transform as 

\begin{equation} 
\mathcal{H}\mapsto\mathcal{H},\quad \mathcal{H}^\bot \mapsto (\mathop{\mathrm{sgn}}J_{\phi}) \mathcal{H}^\bot,\quad \mathcal{C} \mapsto (\mathop{\mathrm{sgn}}J_{\phi}) \mathcal{C},\quad \mathcal{C}^\bot \mapsto (\mathop{\mathrm{sgn}}J_{\phi}) (\mathop{\mathrm{sgn}}\mathop{\mathrm{det}}\alpha^{i}_{j}) \mathcal{C}^\bot.
\label{transf_frame}
\end{equation}
 Moreover, the non-zero components of $\mathbf{g}$ in this frame are the invariants 
$$
 \mathbf{g}(\mathcal{H},\mathcal{H})=\ell_{\mathcal{H}}=\frac{1}{4}\,C_{\rho}, \quad \mathbf{g}(\mathcal{H}^\bot,\mathcal{H}^\bot)=\ell_{\mathcal{H}^\bot}= \pm_{\tilde{\mathbf{g}}}\,\frac{1}{4}\,C_{\rho},\quad \mathbf{g}(\mathcal{C},\mathcal{C})=\ell_{\mathcal{C}},\quad \mathbf{g}(\mathcal{C}^\bot,\mathcal{C}^\bot)=\ell_{\mathcal{C}^\bot}=\pm_{\mathbf{h}}\ell_{\mathcal{C}},
 $$ 
where 
$$
\pm_{\tilde{\mathbf{g}}} = \mathop{\mathrm{sgn}} \det\tilde{\mathbf g}, \quad
\pm_{\mathbf{h}} = \mathop{\mathrm{sgn}} \det \mathbf h.
$$

\end{prop}


\subsection{Further first-order scalar invariants}
\label{subsec:Further first-order scalar invariants}

In this section we discover three new semi-invariants 
$\Theta_{\rm I}$, $\Theta_{\rm II}$, $\Theta_{\rm III}$ by examining
the components of the O'Neill tensors in the frame 
$\{\mathcal{H},\mathcal{H}^{\bot},\mathcal{C},\mathcal{C}^\bot\}$.
This frame is well defined only when $C_{\rho}\ell_{\mathcal{C}} \neq 0$.

In the rest of the paper, the components of a tensor $\mathbf{W}$ with respect to the frame $\{\mathcal{H},\mathcal{H}^{\bot},\mathcal{C},\mathcal{C}^\bot\}$ will be denoted by $\mathbf{W}^{(a)(b)...}_{(c)(d)...}$. Thus, e.g.,  $\mathbf{g}_{(1)(1)} =\mathbf{g}\left(\mathcal{H},\mathcal{H}\right)$, $\mathbf{g}_{(1)(2)} =\mathbf{g}\left(\mathcal{H},\mathcal{H}^\bot\right)$, etc. 

Now, in view of Proposition \ref{g_frame} the nonzero components $\mathbf{g}_{(a)(b)}$ are scalar invariants, which coincide up to a sign with  $\ell_{\mathcal{H}}$ and $\ell_{\mathcal{C}}$.
Analogously, the only nonzero components $\mathbf{A}^{(a)}_{(b)(c)}$ of the O'Neill tensor $\mathbf{A}$ are
$\mathbf{A}^{(1)}_{(2)(3)} = -\tfrac12 \ell_{\mathcal{C}}$, 
$\mathbf{A}^{(2)}_{(1)(3)} = \pm_{\tilde{\mathbf{g}}} \tfrac12 \ell_{\mathcal{C}}$,
$\mathbf{A}^{(3)}_{(1)(2)} = -\tfrac12 \ell_{\mathcal{H}}$,
$\mathbf{A}^{(3)}_{(2)(1)} = \tfrac12 \ell_{\mathcal{H}}$,
yielding no new scalar invariant.

On the contrary, the nonzero components $\mathbf{T}^{(a)}_{(b)(c)}$ of the O'Neill tensor $\mathbf{T}$
are much more interesting. Indeed,  in order of increasing complexity the nonzero components of  $\mathbf{T}$ are


$$
\mathbf{T}^{(3)}_{(3)(2)}, \mathbf{T}^{(3)}_{(4)(2)}, \mathbf{T}^{(3)}_{(4)(1)}, \mathbf{T}^{(3)}_{(3)(1)},  
\mathbf{T}^{(4)}_{(4)(1)}.
$$ 
In  view of  (\ref{transf_frame}), the first three components of this quintuple are semi-invariants, whereas  $\mathbf{T}^{(3)}_{(3)(1)}$ and $\mathbf{T}^{(4)}_{(4)(1)}$ are invariants. Before exploring them in detail we also introduce two semi-invariant tensors 
of type $(1,1)$ defined by
$$
\mathbf{T_{\mathcal{C}}}(U) = \mathbf{T}(\mathcal{C}, U), \quad
\mathbf{T_{\mathcal{C}^\bot}}(U) = \mathbf{T}(\mathcal{C}^\bot, U) 
$$
for an arbitrary vector field $U$ on $\mathcal{M}$. 
Thus,
$$
(\mathbf{T_{\mathcal{C}}})^k_j = \mathbf{T}^k_{sj}\mathcal{C}^s, \quad
(\mathbf{T_{\mathcal{C}^\bot}})^k_j = \mathbf{T}^k_{sj}\mathcal{C}^{\bot s}.
$$
Of interest are
$$
\Theta_{\mathcal{C}} = \det \mathbf{T_{\mathcal{C}}}, \quad
\Theta_{\mathcal{C}^\bot} = \det \mathbf{T_{\mathcal{C}^\bot}},
$$
since they are invariants, while the traces
$(\mathbf{T_{\mathcal{C}}})^i_i = (\mathbf{T_{\mathcal{C}^\bot}})^i_i = 0$ vanish.
We could choose $\Theta_{\mathcal{C}}$ or $\Theta_{\mathcal{C}^\bot}$ 
as the missing sixth invariant, but there exist semi-invariants of lower complexity.
To find these semi-invariants, we first consider the vector fields 
$$
\mathcal{T} = \mathbf{T}_{(3)(2)} = \mathbf{T}(\mathcal{C},\mathcal{H}^\bot), \quad
\mathcal{T}^\bot = \mathbf{T}_{(4)(2)} = \mathbf{T}(\mathcal{C}^\bot,\mathcal{H}^\bot).
$$
They are respectively invariant and semi-invariant and constitute another orthogonal frame in $\Xi$. Moreover, denoting by  
$$
\ell_{\mathcal{T}} = \mathbf{g}(\mathcal{T},\mathcal{T}), \quad
\ell_{\mathcal{T}^\bot} = \mathbf{g}(\mathcal{T}^\bot,\mathcal{T}^\bot)
$$
the squared lengths, which are invariants, we have
$\ell_{\mathcal{T}^\bot} = \pm_{\mathbf{h}} \ell_{\mathcal{T}}$.
The angles between $\mathcal{T}$ or $\mathcal{T}^\bot$ and 
$\mathcal{C}$ or $\mathcal{C}^\bot$ (they all lie in $\Xi$)
are scalar semi-invariants, proportional to 
$\mathbf{g}(\mathcal{T},\mathcal{C}) = \ell_{\mathcal{C}} \mathbf{T}^{(3)}_{(3)(2)}$ and
$\mathbf{g}(\mathcal{T}^\bot,\mathcal{C}) = \ell_{\mathcal{C}} \mathbf{T}^{(3)}_{(4)(2)}$.

The coordinate description of the invariants and semi-invariants
introduced above involves the three relatively simple semi-invariants 
$$
\Theta_{\mathrm{I}} = 
\frac{1}{\left| \det\mathbf h \right|^{1/2} \left|\det\tilde{\mathbf g} \right|^{1/2}}
\left|\begin{array}{@{}ccc@{}}
h_{11,1} & h_{12,1} & h_{22,1} \\
h_{11,2} & h_{12,2} & h_{22,2} \\
(\mathcal{C}^2)^2 & -\mathcal{C}^1 \mathcal{C}^2 & (\mathcal{C}^1)^2
\end{array}\right|,
$$
$$
\Theta_{\mathrm{II}} =4\mathbf{g}(\mathcal{T},\mathcal{C})= 
\frac{1}{\det\mathbf h \left|\det\tilde{\mathbf g} \right|^{1/2}}
\left|\begin{array}{@{}ccc@{}}
h_{11,1} & h_{12,1} & h_{22,1} \\
h_{11,2} & h_{12,2} & h_{22,2} \\
-2 h_{1k} \mathcal{C}^k \mathcal{C}^2 &
   h_{11} (\mathcal{C}^1)^2 - h_{22} (\mathcal{C}^2)^2 &
   2 h_{2k} \mathcal{C}^1 \mathcal{C}^k
\end{array}\right|,
$$
and
$$
\Theta_{\mathrm{III}} = 
\frac{1}{\left| \det\mathbf h \right|^{3/2} \left|\det\tilde{\mathbf g} \right|^{1/2}}
\left|\begin{array}{@{}ccc@{}}
h_{11,1} & h_{12,1} & h_{22,1} \\
h_{11,2} & h_{12,2} & h_{22,2} \\
(h_{1k} \mathcal{C}^k)^2 &
   h_{1k} h_{2l} \mathcal{C}^k \mathcal{C}^l &
   (h_{2l} \mathcal{C}^l)^2
\end{array}\right|,
$$
which are such that
\begin{equation}\label{thetas_inv}
\Theta_{\mathrm{I}}^2=16\,\Theta_{\mathcal{C}}, \quad
\Theta_{\mathrm{II}} = 4 \ell_{\mathcal{C}} \mathbf{T}^{(3)}_{(3)(2)},\quad  \Theta_{\mathrm{III}}^2=\pm_{\tilde{\mathbf{g}} \mathbf{h}}16\,\Theta_{\mathcal{C}^\bot}.
\end{equation}


\begin{rem}Although all three semi-invariants $\Theta_{\rm I}$, $\Theta_{\rm II}$, $\Theta_{\rm III}$ exist (as determinants) independently of the condition $\ell_{\mathcal{C}} \neq 0$, it turns that they all vanish if $\ell_{\mathcal C} = 0$.
\end{rem} 


Summarizing,
$$
C_{\rho},
C_{\chi},
C_{\gamma},
Q_{\chi},
Q_{\gamma},
\Theta_{\mathcal{C}}, 
\Theta_{\mathcal{C}^\bot}, 
\ell_{\mathcal{C}},
\ell_{\mathcal{T}},
\mathbf{T}^{(3)}_{(3)(1)},  
\mathbf{T}^{(4)}_{(4)(1)}
$$
and 
$$
\Theta_{\mathrm{I}},
\Theta_{\mathrm{II}},
\Theta_{\mathrm{III}},
\mathbf{T}^{(3)}_{(3)(2)}, 
\mathbf{T}^{(3)}_{(4)(2)}, 
\mathbf{T}^{(3)}_{(4)(1)}, 
$$
are two (relatively simple) sets of first-order scalar invariants and semi-invariants, respectively.

As we already know, at most six first-order invariants can be functionally independent. 
Since the semi-invariants can only change their sign under the pseudogroup action \eqref{eq:psgroup_fin}, 
$$
C_{\rho},
C_{\chi},
Q_{\chi},
Q_{\gamma}, 
\ell_{\mathcal{C}},
(\Theta_{\mathrm{I}})^2
$$
turn out to be the simplest six functionally independent scalar invariants.

\begin{prop}
\label{prop:main_1ord}
The scalar differential invariants $I_{1}=C_{\rho}$, $I_{2}=C_{\chi}$,
$I_{3}=Q_{\chi}$, $I_{4}=Q_{\gamma}$, $I_{5}=\ell_{\mathcal{C}}$
and $I_{6}=(\Theta_{\mathrm{I}})^{2}$
form a maximal system of generically functionally independent scalar
differential invariants of the first order. 
\end{prop}

\begin{proof} The rank of the Jacobian at a generic point of the jet space is equal to 6.
\end{proof}

Thus, all invariants can be expressed in terms of
$C_{\rho},C_{\chi},Q_{\chi},Q_{\gamma}, \ell_{\mathcal{C}},(\Theta_{\mathrm{I}})^{2}$.
The simplest functional relations are provided by \eqref{thetas_inv} and

$$
\frac{\Theta_{\mathrm{II}}^2}{16\,\ell_{\mathcal{C}}^2}
 \pm_{\mathbf{h}}(\mathbf{T}^{(3)}_{(4)(2)})^2 
 \pm_{\tilde{\mathbf{g}}} \tfrac14 (Q_\chi - Q_\gamma) = 0.
$$
Moreover, the relations among
$Q_\chi, Q_\gamma, \ell_{\mathcal C}, \Theta_{\rm I}, \Theta_{\rm II}, \Theta_{\rm III}$ 
are
$$
-2 \ell_{\mathcal C} \sqrt{\smash[b]{|Q_\gamma|}} + \Theta_{\mathrm I} + \Theta_{\mathrm{III}} = 0
$$
and 
$$
\pm_{\tilde{\mathbf g}} 4 Q_\chi \ell_{\mathcal C}^2 
\mp_{\mathbf h} 8 \Theta_{\mathrm I} \sqrt{\smash[b]{|Q_\gamma|}} \ell_{\mathcal C}
\pm_{\mathbf h} 4 \Theta_{\mathrm I}^2 
+ \Theta_{\mathrm{II}}^2 = 0, 
$$
where $Q_\gamma$ is the scalar invariant defined in Section \ref{subsec:Invariants C and Q}.

\section{Additional second-order invariants}
\label{sec6}

Besides the fourteen second-order Carminati--McLenaghan invariants~\cite{C-M}
available for every four-dimensional metric, there are additional invariants originating in the 
submersion structure.

An infinite sequence of higher order scalar differential invariants
is obtained by repeatedly applying invariant or semi-invariant differentiations to the first 
order scalar invariants listed in Proposition~\ref{prop:main_1ord}.
Invariant differentiations correspond to vector fields on the orbit space
$\mathcal S$. We already introduced two such vector fields 
$\mathcal{X}, \mathcal{X}^\perp$ in Subsection~\ref{subsec:Invariants C and Q}, 
assuming that $C_\rho \ne 0$;
the corresponding invariant differentiations will be denoted $X,X^\perp$.
In coordinates, 
$$
X = \tilde{\mathbf{g}}^{is} \frac{(\det\mathbf h)_{,s}}{\det\mathbf h} D_{t^{i}}, 
$$
$$
X^\bot = \frac{(\det\mathbf h)_{,2}}{(\det\mathbf h)\sqrt{\left|\det\tilde{\mathbf g}\right|}} 
  D_{t^{1}}-\frac{(\det\mathbf h)_{,1}}{(\det\mathbf h)\sqrt{\left|\det\tilde{\mathbf g}\right|}} 
  D_{t^{2}},
$$
cf. Lemma~\ref{lem:XXbot}. 
Therefore, according to Proposition~\ref{prop:main_1ord}, we have 12 second-order invariants
$Z_{i}(I_{j})$, $i = 1,2$, $j = 1,\dots, 6$, where $Z_1 = X$ and $Z_2 = X^\perp$.

A related construction of higher-order invariants is as follows.
Let $\mathcal I^1, \mathcal I^2$ be two scalar invariants such that 
$$
\Delta = \left|\begin{array}{@{}cc@{}} X\mathcal I^1 & X\mathcal I^2 \\
      X^\bot\mathcal I^1 & X^\bot\mathcal I^2
    \end{array}\right|
\ne 0.
$$
For any other scalar invariant $\phi$, define scalar invariants 
$\phi_{\mathcal I^1}, \phi_{\mathcal I^2}$ by
$$
\phi_{\mathcal I^1} 
 = \frac1\Delta
   \left|\begin{array}{@{}cc@{}} X\phi & X\mathcal I^2 \\ 
        X^\bot\phi & X^\bot\mathcal I^2
    \end{array}\right|,
\quad
\phi_{\mathcal I^2}
 = \frac1\Delta 
    \left|\begin{array}{@{}cc@{}} X\mathcal I^1 & X\phi \\ 
        X^\bot\mathcal I^1 & X^\bot\phi 
    \end{array}\right|.
$$
When $\mathcal I^1, \mathcal I^2$ are of the first order and $\phi$ is of order $n \ge 1$, then 
$\phi_{\mathcal I^1}, \phi_{\mathcal I^2}$ are, in general, of order $n + 1$.
The invariants $\phi_{\mathcal I^1}, \phi_{\mathcal I^2}$ have an obvious geometric meaning. 
The scalar invariants $\mathcal I^1, \mathcal I^2$ restricted to the orbit space 
$\mathcal S$ (see Remark~\ref{not}) constitute a local coordinate system on $\mathcal S$ if they
are functionally independent or, equivalently, when $\Delta \ne 0$ (still assuming that $C_\rho \ne 0$).  
Let $\phi$ be any other invariant restricted to $\mathcal S$.
Solving $\{X \phi = X \mathcal I^i\,\partial\phi/\partial\mathcal I^i$,
$X^\bot \phi = X^\bot \mathcal I^i\,\partial\phi/\partial\mathcal I^i\}$
as a linear system for $\partial\phi/\partial\mathcal I^i$, we see that
the partial derivative 
$\partial\phi/\partial\mathcal I^i$ is equal to $\phi_{\mathcal I^i}$.

Additional invariants arise by means of formula (\ref{eq:inv_C_Q}) 
for suitable symmetric bilinear forms on the orbit space $\mathcal{S}$. 
For instance, denoting
$\mathop{\mathrm{ric}}=\mathop{\mathrm{Ric}}(\tilde{\mathbf{g}})$ the Ricci form of
$\mathcal S$, one has the invariants $Q_{\mathop{\mathrm{ric}}}$ and also $C_{\mathop{\mathrm{ric}}} = \mathrm{Sc}_{\mathcal S}$, 
the scalar curvature of $\mathcal{S}$.
Along the same line, in view of the invariance of $\sigma=d\ln |{\det\mathbf h}|$, 
the Hessian $\nu=\mathop{\mathrm{Hess}}(\ln |{\det\mathbf h}|)$,
defined by
\[
\nu(U,V) = \mathop{\mathrm{Hess}}(\ln |{\det\mathbf h}|)(U,V)
 = U\mathbin{\lrcorner} \nabla_{V}d\left(\ln |{\det\mathbf h}|\right)
 = V\mathbin{\lrcorner} \nabla_{U}d\left(\ln |{\det\mathbf h}|\right)
\]
for all vector fields $U,V$ on $\mathcal S$, is another symmetric bilinear form on $\mathcal{S}$. 
Hence, one obtains two additional
invariants $Q_{\nu}$ and 
$C_{\nu} = \Delta_{\mathcal S} \ln |{\det\mathbf h}|$, where $\Delta_{\mathcal S}$ is
the Laplace--Beltrami operator. It is worth mentioning here that
\[
C_{\nu}':=C_{\nu}-2C_{\chi}+C_{\rho}=g^{ij}h^{kl}h_{kl,ij}+g_{i}^{ij}h^{kl}h_{kl,j}+\frac{1}{2}g^{ij}g^{m,n}g_{mn,i}h^{kl}h_{kl,j}
\]
has a noteworthy simpler coordinate expression than $C_{\nu}$ itself.

\begin{prop} \label{prop:main_2ord}
The invariants $\mathcal{I}_{j}$, $Z_{i}(\mathcal{I}_{j})$, $i = 1,2$, $j = 1, \dots, 6$
(see Proposition~\ref{prop:main_1ord}), 
$C_{\mathop{\mathrm{ric}}}$ and $C_{\nu}$ (or $C_{\nu}'$) form a maximal
system of $20$ generically functionally independent scalar differential
invariants of order less than $3$. All other second-order invariants are
functionally dependent on these. 
\end{prop}

\begin{proof}
The rank of the Jacobian at a generic point of the jet space is equal to 20.
\end{proof}

For example, one can check that 
\[
\begin{aligned}
&Q_{\mathop{\mathrm{ric}}}=\tfrac{1}{4}\left(C_{\mathop{\mathrm{ric}}}\right)^{2},
\\
&4 C_\rho^2 Q_\nu + (XC_\rho)^2 \pm_{\tilde{\mathbf g}} (X^\bot C_\rho)^2 - 2 C_\nu  C_\rho  XC_\rho = 0.
\end{aligned}
\]
If $C_\rho \ne 0$, then the last formula allows us to express $Q_\nu$ in terms of 
$C_\rho$, $X C_\rho$, $X^\bot C_\rho$ and $C_\nu$.

To extend the set of geometrically meaningful invariants we consider the sectional curvatures
\[
K(\Xi)=\frac{g\left(R\left(\partial_{z^{1}},\partial_{z^{2}}\right)\partial_{z^{1}},\partial_{z^{2}}\right)}{g\left(\partial_{z^{1}},\partial_{z^{1}}\right)g\left(\partial_{z^{2}},\partial_{z^{2}}\right)-g\left(\partial_{z^{1}},\partial_{z^{2}}\right)^{2}}
\]
and 
\[
K(\Xi^{\perp})=\frac{g\left(R\left(\mathbf e_{1},\mathbf e_{2}\right)\mathbf e_{1},\mathbf e_{2}\right)}
{g\left(\mathbf e_{1},\mathbf e_{1}\right)g\left(\mathbf e_{2},\mathbf e_{2}\right)
 - g\left(\mathbf e_{1},\mathbf e_{2}\right)^{2}}
\]
of $\Xi$ and $\Xi^{\perp}$, respectively, with vectors $\mathbf e_i$ being given by formulas~\eqref{eq:e}.

\begin{prop} \label{GaussCurv}
We have
$$
K(\Xi)=-\tfrac{1}{4}C_{\chi},\quad
K(\Xi^{\perp})=\tfrac{1}{2}C_{\mathop{\mathrm{ric}}}\mp_{\tilde{\bf g}} \tfrac{3}{4}
  \ell_{\mathcal{C}}.
$$
\end{prop}

\begin{proof}
Both formulas are routinely checked in adapted coordinates.
\end{proof}

Finally, second-order invariants can be also obtained from the commutator $[\mathcal X, \mathcal X^\bot]$,
which lies in $\Xi$, hence is a linear combination of $\mathcal X$ and $\mathcal X^\bot$.
However, the coefficients are rather simple expressions in 
$C_\rho$, $X C_\rho$, $X^\bot C_\rho$ and $C_\nu$.

\begin{prop}
Let $C_\rho \ne 0$. Then the (semi-)invariant differentiations $\mathcal X$ and $\mathcal X^\bot$ 
satisfy the commutation relations
$$
[\mathcal X, \mathcal X^\bot] = J_1 \mathcal X + J_2 \mathcal X^\bot,
$$
where
$$
J_1 = -\frac{X^\bot C_\rho}{C_\rho}, \quad
J_2 = \frac{X C_\rho}{C_\rho} - C_\nu.
$$
\end{prop}

\begin{proof}
By orthogonality, we have
$$
J_1 = \frac{\tilde{\mathbf{g}}(\mathcal{X},[\mathcal{X},\mathcal{X}^\bot])}{C_\rho}, \quad
J_2 = \frac{\tilde{\mathbf{g}}(\mathcal{X^\bot},[\mathcal{X},\mathcal{X}^\bot])}{C_\rho}.
$$
Identities 
$\tilde{\mathbf{g}}(\mathcal{X},[\mathcal{X},\mathcal{X}^\bot]) = -\mathcal{X}^\bot(C_\rho)$
and 
$\tilde{\mathbf{g}}(\mathcal{X^\bot},[\mathcal{X},\mathcal{X}^\bot]) 
 = X^\bot C_\rho - C_\rho C_\nu$  
are routinely checked in adapted coordinates.
\end{proof}

\section{$\Lambda$-vacuum Einstein equations for $G_2$ metrics, and their solutions in the special cases $C_\rho=0$ and $\ell_\mathcal{C}=0$ }


\label{sect:Lambda vacuum eqs}

Vacuum Einstein equations for metrics with two commuting Killing fields
have been derived by
Geroch~\cite{Ger1,Ger2}, Gaffet~\cite[eq.~(3.15)]{Gaff}, 
Whelan and Romano~\cite{W-R}.
Here we look for $\Lambda$-vacuum equations.
We obtain a tractable system by choosing $\orbitmetric_{ij}, f_j^k, h_{kl}$
as dependent variables, i.e., substituting
$$
f_{il} = f^k_i h_{kl}, \quad
g_{ij} = \orbitmetric_{ij} + f_i^k f_j^l h_{kl}.
$$
This choice ensures that the components of the inverse matrix $\fourmetric^{\alpha\beta}$ 
are relatively simple.
Then, to  simplify the Einstein equations further, we exploit the fact that the metric $\orbitmetric$, 
being two-dimensional and nondegenerate, is conformally flat. Hence, depending on the position of 
the Killing leaves in the spacetime, $\orbitmetric$ is either
conformally Euclidean or conformally Minkowskian.
In addition, denoting by $H$ the symmetric $2 \times 2$ matrix with elements $h_{kl}$, 
it is useful to introduce the row vectors
$$
\begin{gathered}
F_i = (\begin{array}{@{}cc@{}}
\p{i}1 & \p{i}2 \end{array}),
\quad i = 1,2,
\\
P = (F_{1,2} - F_{2,1}) H,
\end{gathered}
$$
i.e., $P$ is a row vector obtained by multiplication of the row vector
$F_{1,2} - F_{2,1}$ by $H$ from the right.
By comparison with formula~\eqref{curv-vect}, $P = 0$ iff $\mathcal C = 0$ 
iff the metric is orthogonally transitive.

Below we derive the $\Lambda$-vacuum Einstein equations for $\mathrm G_2$-metrics.
We find their explicit solutions in the special cases $C_\rho=0$ and $\ell_\mathcal{C}=0$. 
In particular, we show that when $C_\rho = 0$, then the corresponding  $\Lambda$-vacuum Einstein metrics belong to the well-understood class of $pp$-waves, characterised by the presence of a
constant null vector, see~\cite[\S 25.5]{S-K-M-H-H} and references therein. In this special case all first-order invariants vanish.
In the case when $\ell_\mathcal{C}=0$, on the other hand, we show that the explicit vacuum solution originally presented by Kundu \cite{Kun} can be extended to the $\Lambda$-vacuum case. 
In particular, we find two new solutions (\ref{eq:Lambda Kundu solution}--\ref{eq:Lambda Kundu equation}) and 
(\ref{eq:Lambda Kundu solution c = 0}--\ref{eq:Lambda Kundu equation c = 0}).

\subsection{The case when $\tilde{\mathbf{g}}$ is Lorentzian and explicit solutions with $C_\rho=0$}




\begin{prop} \label{prop:Lorentz Crho nonzero}
Let the metric \eqref{eq:gG} be such that $C_\rho \ne 0$, with $\det \tilde{\mathbf{g}} < 0$ and  $\det H > 0$. Then, by writing the orbit metric in the conformally flat form
$$
\tilde{\mathbf{g}} = \left(
\begin{array}{@{}cc@{}}
0 & q \\ q & 0
\end{array}
\right),
\quad q=q(t^1,t^2) \ne 0,
$$
the $\Lambda$-vacuum Einstein equations
$\fourRicci_{\mu\nu} - \Lambda \fourmetric_{\mu\nu} = 0$
are equivalent to the compatible system of matrix and scalar equations
\begin{equation}
\label{EiEq hyp}
\begin{gathered}
(r H_{,1} H^{-1})_{,2} + (r H_{,2} H^{-1})_{,1}
 = 2 \Lambda q r E + \frac{q}{r} A^\top A H^{-1},
\\ (\ln q)_{,1}
 = (\ln (\ln r)_{,1})_{,1} + \frac{\tr (H_{,1} H^{-1} H_{,1} H^{-1})}{4 (\ln r)_{,1}},
\\ (\ln q)_{,2}
 = (\ln (\ln r)_{,2})_{,2} + \frac{\tr (H_{,2} H^{-1} H_{,2} H^{-1})}{4 (\ln r)_{,2}},
 \\ F_{1,2} - F_{2,1}=\frac{q}{r}A\,H^{-1},
\end{gathered}
\end{equation}
where $r = \sqrt{\mathop{\mathrm{det}}H}$, $E$ is the $2 \times 2$ unit matrix and
$$
A = (\begin{array}{@{}cc@{}}
a_1 & a_2 \end{array}),
\quad a_i = \text{const}
$$
denotes an arbitrary constant row vector, which is zero if and only if $\Xi^{\perp}$ is completely integrable. Moreover,
$$
r_{,12} = -\Lambda q r - \frac{q}{4 r} A H^{-1} A^\top
$$
as a consequence of first equation of \eqref{EiEq hyp}.
\end{prop}


\begin{proof}
By assumption $C_\rho \ne 0$, where
$$
C_\rho = 8 \frac{r_{,1} r_{,2}}{r^2 q}
 = 2 \frac{(\det H)_{,1} (\det H)_{,2}} {(\det H)^2 q}.
$$
Therefore, $(\det H)_{,1} \ne 0$, $(\det H)_{,2} \ne 0$.
Denote $\fourRicci$ the Ricci tensor of the metric $\fourmetric$.
Solving the Einstein equations $\fourRicci_{\mu\nu} - \Lambda \fourmetric_{\mu\nu}$
with respect to $H_{,12}, P_{,1}, P_{,2}, q_{,1}, q_{,2}, q_{,12}$, we obtain one $2\times2$ matrix equation
\begin{equation}
\label{EiEq hyp H}
H_{,12}  - \frac12 H_{,1} H^{-1} H_{,2} - \frac12 H_{,2} H^{-1} H_{,1}
 + \frac14 \frac{(\mathop{\mathrm{det}}H)_{,1}}{\mathop{\mathrm{det}}H} H_{,2} + \frac14 \frac{(\mathop{\mathrm{det}}H)_{,2}}{\mathop{\mathrm{det}}H} H_{,1}
 + \Lambda q H + \frac{1}{2q} P^\top P = 0,
\end{equation}
two vector equations
\begin{equation}
\label{EiEq hyp P}
P_{,1}
 = \left(\frac{(\mathop{\mathrm{det}}H)_{,11}}{(\mathop{\mathrm{det}}H)_{,1}} 
     - \frac{\mathop{\mathrm{det}}(H_{,1})}{(\mathop{\mathrm{det}}H)_{,1}} - \frac{(\mathop{\mathrm{det}}H)_{,1}}{\mathop{\mathrm{det}}H}\right) P,
\quad P_{,2}
 = \left(\frac{(\mathop{\mathrm{det}}H)_{,22}}{(\mathop{\mathrm{det}}H)_{,2}} 
     - \frac{\mathop{\mathrm{det}}(H_{,2})}{(\mathop{\mathrm{det}}H)_{,2}} - \frac{(\mathop{\mathrm{det}}H)_{,2}}{\mathop{\mathrm{det}}H}\right) P,
\end{equation}
and three scalar equations
\begin{equation}
\label{EiEq hyp q}
\begin{gathered}
q_{,1}
 = \left(\frac{(\mathop{\mathrm{det}}H)_{,11}}{(\mathop{\mathrm{det}}H)_{,1}} 
     - \frac{\mathop{\mathrm{det}}(H_{,1})}{(\mathop{\mathrm{det}}H)_{,1}} - \frac12 \frac{(\mathop{\mathrm{det}}H)_{,1}}{\mathop{\mathrm{det}}H}\right) q,
\quad q_{,2}
 = \left(\frac{(\mathop{\mathrm{det}}H)_{,22}}{(\mathop{\mathrm{det}}H)_{,2}} 
     - \frac{\mathop{\mathrm{det}}(H_{,2})}{(\mathop{\mathrm{det}}H)_{,2}} - \frac12 \frac{(\mathop{\mathrm{det}}H)_{,2}}{\mathop{\mathrm{det}}H}\right) q,
\\
(\ln q)_{,12}  = \frac{(\tr H)_{,1} (\tr H)_{,2} - \tr(H_{,1} H_{,2})}{4 \mathop{\mathrm{det}}H}
 + \frac{3}{4 q} P H^{-1} P^\top.
\end{gathered}
\end{equation}
By comparison of the cross derivatives $q_{,12}$ and $q_{,21}$, one sees that the third equation~\eqref{EiEq hyp q} 
is a differential consequence of the first two.

Conversely, if the five equations~\eqref{EiEq hyp H}, \eqref{EiEq hyp P} and~\eqref{EiEq hyp q} hold, then $\fourRicci_{\mu\nu} = \Lambda \fourmetric_{\mu\nu}$.
Compatibility of the equations is routinely checked.

As an easy consequence of equations~\eqref{EiEq hyp P} and~\eqref{EiEq hyp q} we obtain
$$
P_{,1}  = \left(\frac{q_{,1} }{q} - \frac12 \frac{(\mathop{\mathrm{det}}H)_{,1}}{\mathop{\mathrm{det}}H}\right) P,
\quad
P_{,2} = \left(\frac{q_{,2} }{q} - \frac12 \frac{(\mathop{\mathrm{det}}H)_{,2}}{\mathop{\mathrm{det}}H}\right) P.
$$
It follows that $r P/q$ is a constant vector (recall that $r = \sqrt{\mathop{\mathrm{det}}H}$).
Therefore, we can write
$$
P = \frac{q}{r} A,
$$
where $A$ is an arbitrary constant row vector.
Now the Einstein equations reduce to system~\eqref{EiEq hyp}.
\end{proof}

\begin{rem}
Notice that, when $A = 0$ and $\Lambda = 0$, equations~\eqref{EiEq hyp} reduce to the well-known
Belinsky--Zakharov~\cite{B-Z} formulation of the vacuum Einstein equations.
\end{rem}


\begin{prop} \label{prop:Lorentz Crho zero}
All $\Lambda$-vacuum Einstein metrics of the form \eqref{eq:gG}, with $C_\rho = 0$, $\det \tilde{\mathbf{g}} < 0$ and  $\det H > 0$, satisfy $\Lambda=0$ and in adapted coordinates can be written in the form

\begin{equation}
\label{pp-wave1}
\fourmetric = d t^1\,d t^2
 + R^2 (dz^1 + W\,dz^2)^2 + S^2(dz^2)^2,
\end{equation}
with $R$, $W$ and $S$ differentiable functions of\/ $t^1$ such that $R\,S\neq0$ and 

\begin{equation}
\label{eq_pp-wave1}
\left(W'\right)^2= \frac{2S^2}{R^2}\left( \frac{R''}{R}+\frac{S''}{S} \right).
\end{equation}
In particular these Ricci-flat metrics are such that $\mathcal{C}=0$ (then $\ell_{\mathcal{C}}=0$), hence are orthogonally transitive and, in addition, are pp-waves since $\partial_{t^2}$ is a null Killing vector field such that $\nabla \partial_{t^2}=0$.
\end{prop}


\begin{proof}
By assumption
$$
0 = C_\rho = 2 \frac{(\det H)_{,1} (\det H)_{,2}}{(\det H)^2 q}.
$$
Therefore,  $\det H$ is a function of either $t^1$ or $t^2$. We assume here that $\det H$ is a function of $t^1$.\\
On the other hand, since $h_{11}\neq 0$ can be always achieved by a linear change of coordinates 
$\{\bar z^{i}=\alpha^{i}_{j}\,z^{j}\}$, without loss of generality one can write $\mathbf{h}$ 
(the restriction of the metric to the Killing leaves $\Xi$) as 
$$
\mathbf{h}=h_{11}\left[  \left( dz^1 +\frac{h_{12}}{h_{11}}\,dz^2 \right)^2 +\frac{\left(h_{11}\,h_{22}-h_{12}^2\right)}{h_{11}^2}\, (dz^2)^2\right],
$$
i.e., in the  Weyl--Lewis--Papapetrou form \cite{Lew,Pa} 
$$
\mathbf{h}=\frac{r}{s}\left[  \left( dz^1 +w\,dz^2 \right)^2 \pm_H\,s^2\, (dz^2)^2\right],
$$
with
$$
w=\frac{h_{12}}{h_{11}},\qquad r = \sqrt{\mathop{\mathrm{det}}H}, \qquad s=\frac{r}{h_{11}}.
$$
In terms of  Weyl--Lewis--Papapetrou parameters $r,s,w$ the analysis of Einstein equations
$\mathbf L_{\mu\nu} = \fourRicci_{\mu\nu} - \Lambda \fourmetric_{\mu\nu} = 0$ simplifies noteworthy. 
Indeed, by computing the contravariant components $L^{\mu \nu}$, one finds that $ \mathbf L^{11}=0$ 
if and only if $s_{,2}^2 + w_{,2}^2=0$.
Since $w,s$ are real, it follows that they are functions of $t^1$.
Hence, all components $h_{ij}$ are functions of~$t^1$, which
substantially simplifies computation of the remaining components of $\mathbf L$.
In particular, we obtain
$$
0 = \mathbf L^{13} + f_2^1 \mathbf L^{12} = -\tfrac12 \mathcal C^3_{t^2},
\quad
0 = \mathbf L^{14} + f_2^2 \mathbf L^{12} = -\tfrac12 \mathcal C^4_{t^2}.
$$
Consequently, components of the curvature vector depend on $t^1$ only as well.
Continuing further, we obtain
$$
\begin{aligned}
0 &= \mathbf L_{33} + 3 h_{11} \mathbf L^{12}
 = \tfrac12 (\mathcal C^2)^2 \det H
  + \left(\frac{q_{,12}}{q^3} - \frac{q_{,1} q_{,2}}{q^2} \right) h_{11}
  + 4 \Lambda h_{11},
\\
0 &= \mathbf L_{34} + 3 h_{12} \mathbf L^{12}
 = \tfrac12 \mathcal C^3 \mathcal C^2 \det H
  + \left(\frac{q_{,12}}{q^3} - \frac{q_{,1} q_{,2}}{q^2} \right) h_{12}
  + 4 \Lambda h_{12},
\\
0 &= \mathbf L_{44} + 3 h_{22} \mathbf L^{12}
 = \tfrac12 (\mathcal C^3)^2 \det H
  + \left(\frac{q_{,12}}{q^3} - \frac{q_{,1} q_{,2}}{q^2} \right) h_{22}
  + 4 \Lambda h_{22}
\\
0 &= (4 q \det H) \mathbf L^{12}
  + h_{22} \mathbf L_{33} - 2 h_{12} \mathbf L_{34} + h_{11} \mathbf L_{44}
 = 2 \left(\frac{q_{,12}}{q^3} - \frac{q_{,1} q_{,2}}{q^2} + 3 \Lambda \right)
   \det H
\end{aligned}
$$
By the fourth equation,
$$
\frac{q_{,12}}{q^3} - \frac{q_{,1} q_{,2}}{q^2} = -3 \Lambda,
$$
then, by substituting into the remaining three equations and using
$(\mathcal C^3)^2 (\mathcal C^2)^2 = (\mathcal C^3 \mathcal C^2)^2$,
we obtain $\Lambda = 0$ and
\begin{equation}
\label{eq:hyp:q}
q_{,12} = \frac{q_{,1} q_{,2}}{q}.
\end{equation}
Hence $\mathcal C = 0$ and the metric $\fourmetric$ is orthogonally transitive.

On the other hand, equation~\eqref{eq:hyp:q} implies that $q(t^1,t^2)$ is a product,
$q = q_1(t^1) q_2(t^2)$. Therefore, by passing to new coordinates $\bar t^i = \int q_i\,dt^i$, 
the orbit metric reduces to $d\bar t^2\,d \bar t^2$ and the Einstein equations reduce to a 
single ordinary differential equation. Hence, by suitably rearranging the unknown functions, 
one can write the metric $\fourmetric$ and the corresponding Einstein equations in the form 
\eqref{pp-wave1} and  \eqref{eq_pp-wave1}, respectively.
The case when $\det H$ depends on $t^2$ is completely analogous.
\end{proof}

Obviously, the three Killing fields commute and, therefore, the metric has no unique two-dimensional commuting Killing algebra. Hence, it actually falls outside the class of metrics considered in this paper.


\subsection{ The case when $\tilde{\mathbf{g}}$ is Riemannian and explicit solutions with $C_\rho=0$}

In the case of conformally Euclidean orbit metric, we have
$\orbitmetric = q (\diff t^1)^2 + q (\diff t^2)^2$
and, therefore,
\begin{equation}
\label{gg xy ell}
\fourmetric = q ((\diff t^1)^2 + (\diff t^2)^2)
 + h_{kl} (\diff z^k + \p1k \diff t^1 + \p2k \diff t^2)(\diff z^l + \p1l \diff t^1 + \p2l \diff t^2).
\end{equation}
where $q, \p{i}{k}, h_{kl}$ are the unknown functions of $x,y$.
Clearly, $\mathop{\mathrm{det}}H < 0$.



\begin{prop} \label{prop:Eucl Crho nonzero}

Let the metric \eqref{eq:gG} be such that $C_\rho \ne 0$, with $\det \tilde{\mathbf{g}} > 0$ and  $\det H < 0$. Then, by writing the orbit metric in the conformally flat form
$$
\tilde{\mathbf{g}} = \left(
\begin{array}{@{}cc@{}}
q & 0 \\ 0 & q
\end{array}
\right),
\quad q=q(t^1,t^2) \ne 0.
$$
the $\Lambda$-vacuum Einstein equations
$\fourRicci_{\mu\nu} - \Lambda \fourmetric_{\mu\nu} = 0$
are equivalent to the compatible system of matrix and scalar equations

\begin{equation}
\label{EiEq ell}
\begin{gathered}
(r H_{,1} H^{-1})_{,1} + (r H_{,2} H^{-1})_{,2}
 = 2 \Lambda q r E + \frac{q}{r} A^\top A H^{-1},
\\
\left(\ln \frac{q}{r_{,1}^2 + r_{,2}^2}\right)_{,1} = - \frac{r_{,11} + r_{,22}}{r_{,1}^2 + r_{,2}^2} r_{,1}
 + \frac{\mathop{\mathrm{det}}(H_{,1}) - \mathop{\mathrm{det}}(H_{,2})}{r_{,1}^2 + r_{,2}^2} \frac{r_{,1}}{2 r}
 + \frac{(\tr H)_{,1} (\tr H)_{,2} - \tr (H_{,1} H_{,2})}{r_{,1}^2 + r_{,2}^2} \frac{r_{,2}}{2r} ,
\\
\left(\ln \frac{q}{r_{,1}^2 + r_{,2}^2}\right)_{,2} = - \frac{r_{,11} + r_{,22}}{r_{,1}^2 + r_{,2}^2} r_{,2}
 + \frac{\mathop{\mathrm{det}}(H_{,2}) - \mathop{\mathrm{det}}(H_{,1})}{r_{,1}^2 + r_{,2}^2} \frac{r_{,2}}{2 r}
 + \frac{(\tr H)_{,1} (\tr H)_{,2} - \tr (H_{,1} H_{,2})}{r_{,1}^2 + r_{,2}^2} \frac{r_{,1}}{2r} ,
  \\ F_{1,2} - F_{2,1}=\frac{q}{r}A\,H^{-1},
\end{gathered}
\end{equation}
where $r = \sqrt{-\mathop{\mathrm{det}}H}$, $E$ is the unit $2 \times 2$ matrix and
$$
A = (\begin{array}{@{}cc@{}}
a_1 & a_2 \end{array}),
\quad a_i = \text{const}
$$
is an arbitrary constant row vector, which  is zero if and only if\/ $\Xi^{\perp}$ is completely integrable.
Moreover,
$$
r_{,11} + r_{,22} = - 2 \Lambda q r + \frac{q}{2 r} A H^{-1} A^\top
$$
as a consequence of the first equation of \eqref{EiEq ell}.
\end{prop}

\begin{proof}
By assumption
$$
0 \ne C_\rho 
  = 4 \frac{r_{,1}{}^2 + r_{,2}{}^2}{q r^2}.
$$
Consequently, also $r_{,1}{}^2 + r_{,2}{}^2 \ne 0$.
Denote $\fourRicci$ the Ricci tensor of the metric~\eqref{gg xy ell}.
By tedious routine computations,  solving the Einstein equations
$\fourRicci_{\mu\nu} - \Lambda \fourmetric_{\mu\nu}$ with respect to
$H_{,22}, P_{,1}, P_{,2}, q_{,1}, q_{,2}, q_{,22}$, we obtain one $2\times2$ matrix equation
\begin{equation}
\label{EiEq ell H}
H_{,11} + H_{,22} - H_{,1} H^{-1} H_{,1} - H_{,2} H^{-1} H_{,2}
 + \frac12 \frac{(\mathop{\mathrm{det}}H)_{,1}}{\mathop{\mathrm{det}}H} H_{,1} + \frac12 \frac{(\mathop{\mathrm{det}}H)_{,2}}{\mathop{\mathrm{det}}H} H_{,2}
 + 2 \Lambda q H + \frac{1}{q} P^\top P = 0,
\end{equation}
two vector equations
\begin{equation}
\label{EiEq ell P}
P_{,1} = \left(\frac{q_{,1}}{q} - \frac12 \frac{(\mathop{\mathrm{det}}H)_{,1}}{\mathop{\mathrm{det}}H}\right) P,
\quad
P_{,2} = \left(\frac{q_{,2}}{q} - \frac12 \frac{(\mathop{\mathrm{det}}H)_{,2}}{\mathop{\mathrm{det}}H}\right) P,
\end{equation}
and three scalar equations
\begin{equation}
\label{EiEq ell q}
\begin{gathered}
(\mathop{\mathrm{det}}H)_{,2} \frac{q_{,1}}{q} + (\mathop{\mathrm{det}}H)_{,1} \frac{q_{,2}}{q}
 - 2 (\mathop{\mathrm{det}}H)_{,12} + \frac{(\mathop{\mathrm{det}}H)_{,1} (\mathop{\mathrm{det}}H)_{,2}}{\mathop{\mathrm{det}}H} + (\tr H)_{,1} (\tr H)_{,2} - \tr(H_{,1} H_{,2}) = 0,
\\
(\mathop{\mathrm{det}}H)_{,1} \frac{q_{,1}}{q} - (\mathop{\mathrm{det}}H)_{,2} \frac{q_{,2}}{q}
 + (\mathop{\mathrm{det}}H)_{,11} - (\mathop{\mathrm{det}}H)_{,22} - \frac{(\mathop{\mathrm{det}}H)_{,1}^2 + (\mathop{\mathrm{det}}H)_{,2}^2}{2\,\mathop{\mathrm{det}}H}
 - \mathop{\mathrm{det}}(H_{,1}) + \mathop{\mathrm{det}}(H_{,2}) = 0,
\\
(\ln q)_{,11} + (\ln q)_{,22} = \frac{\mathop{\mathrm{det}}(H_{,1}) + \mathop{\mathrm{det}}(H_{,2})}{2\,\mathop{\mathrm{det}}H} - \frac{3}{4 q} P H^{-1} P^\top.
\end{gathered}
\end{equation}
Again, the third equation~\eqref{EiEq ell q} is a differential consequence of the first two.
 
Conversely, if the five equations~\eqref{EiEq ell H}, \eqref{EiEq ell P} and~\eqref{EiEq ell q} hold, then $\fourRicci_{\mu\nu} = \Lambda \fourmetric_{\mu\nu}$.
Compatibility of the equations is routinely checked.

Again, $r P/q$ is a constant vector (recall that $r = \sqrt{-\mathop{\mathrm{det}}H}$) and we can write
$$
P = \frac{q}{r} A,
$$
where $A$ is a constant row vector.

The two scalar equations~\eqref{EiEq ell q} simplify to
$$
\begin{gathered}
r_{,2} \frac{q_{,1}}{q} + r_{,1} \frac{q_{,2}}{q}
 - 2 r_{,12} - \frac{(\tr H)_{,1} (\tr H)_{,2}}{2r} + \frac{\tr(H_{,1} H_{,2})}{2r} = 0,
\\
-r_{,1} \frac{q_{,1}}{q} + r_{,2} \frac{q_{,2}}{q}
 + r_{,11} - r_{,22}
 + \frac{\mathop{\mathrm{det}}(H_{,2}) - \mathop{\mathrm{det}}(H_{,2})}{2r} = 0
\end{gathered}
$$

Then the Einstein equations reduce to system~\eqref{EiEq ell}.
\end{proof}


\begin{prop} \label{prop:Eucl Crho zero} 
All $\Lambda$-vacuum Einstein metrics of the form \eqref{eq:gG}, with $C_\rho = 0$, $\det \tilde{\mathbf{g}} > 0$ and  $\det H < 0$, satisfy $\Lambda=0$ and in adapted coordinates can be written either in the form 
\begin{equation}
\label{pp-wave2}
\fourmetric = (d t^1)^2+(d t^2)^2+ \psi (dz^1)^2 +2\left(c\,t^1 d t^2+d z^2\right )\,d z^1,
\end{equation}
with $c\in\mathbb{R}$ and $\psi=\psi(t^1,t^2)$ a differentiable function such that $\psi_{,11}+\psi_{,22}=c^2$, or in the form
\begin{equation}
\label{pp-wave3}
\fourmetric = e^{t^1}(d t^1)^2+e^{t^1}(d t^2)^2+ \psi (dz^1)^2 +2\left(c\,e^{t^1} d t^2+d z^2\right )\,d z^1,
\end{equation}
with $c\in\mathbb{R}$ and $\psi=\psi(t^1,t^2)$ a differentiable function such that $\psi_{,11}+\psi_{,22}=e^{t^1}c^2$.\\
In particular, these Ricci-flat metrics are such that $\ell_{\mathcal{C}}=0$ and, in addition, are $pp$-waves since $\partial_{z^2}$ is a null Killing vector field such that $\nabla \partial_{z^2}=0$; moreover these metrics are orthogonally transitive if, and only if, $c=0$. 

\end{prop}

\begin{proof} 
By assumption
$$
0 = C_\rho = \frac{(\det H_{,1})^2 + (\det H_{,2})^2}{q\,(\det H)^2}.
$$
Consequently, $\det H$ is a constant. On the other hand, by considering $\mathbf h$ in the  
Weyl--Lewis--Papapetrou form, like in the proof of Proposition \ref{prop:Lorentz Crho zero}, the 
analysis of Einstein equations $\mathbf L_{\mu\nu} = \fourRicci_{\mu\nu} - \Lambda \fourmetric_{\mu\nu} = 0$
simplifies noteworthy. Indeed, in terms of the Weyl--Lewis--Papapetrou parameters $r,s,w$, one has 
$\det H = -r^2$ and, without loss in generality, one can assume $r=1$, because this can always be achieved 
by a coordinate transformation $z^{i}\rightarrow z^{i}/\sqrt{|r|}$, $i=1,2$, and rearranging the sign of 
$s$ whenever $r<0$.  Moreover, by equating contravariant components $\mathbf L^{12}$ and 
$\mathbf L^{11} - \mathbf L^{22}$ to zero modulo
$\det H = \mathrm{const}$, we obtain
$$
\begin{aligned}
0 
 = h_{11,2} h_{22,1} - 2 h_{12,1} h_{12,2} + h_{11,1} h_{22,2}
 &= s_{,1} s_{,2} - w_{,1} w_{,2},
\\
0 
 = h_{11,1} h_{22,1} - h_{12,1}{}^2 + h_{12,2}{}^2 - h_{11,2} h_{22,2}
 &= s_{,1}{}^2 - s_{,2}{}^2 - w_{,1}{}^2 + w_{,2}{}^2,
\end{aligned}
$$
The latter algebraic system has two real solutions
$$
s_{,1} = \pm w_{,1}, \quad s_{,2} = \pm w_{,2}
$$
and also a complex solution
$s_{,2} = \mathrm i\, w_{,1}$, $s_{,1} = -\mathrm i\, w_{,2}$,
which gives $s = \mathrm{const}$, $w = \mathrm{const}$ as the unique real subcase.
Altogether we obtain
$$
s = \pm w + c_1, \quad r = 1,
$$
Hence
$$
H = \left(\begin{array}{@{}cc@{}}
\displaystyle{\frac{1}{\pm w+c_1}} & \displaystyle{ \frac{w}{\pm w+c_1}} \vspace{10pt}\\
\displaystyle{ \frac{w}{\pm w+c_1}} & \displaystyle{\frac{-c_1^2\mp 2c_1 w}{\pm w+c_1}}
\end{array}\right)
\sim
\left(\begin{array}{@{}cc@{}}
\displaystyle{\frac{1}{\pm w+c_1}} & 1 \vspace{10pt}\\
1 & 0
\end{array}\right),
$$
where the matrix congruence $H \sim Q^\top H Q$ is with respect to
the transition matrix
$$
Q = 
\left(\begin{array}{@{}cc@{}}
\pm1 & \pm c_1 \\
0 & 1
\end{array}\right).
$$
Otherwise said, we can take 
$$
H = \left(\begin{array}{@{}cc@{}}
\psi & 1\\
1 & 0
\end{array}\right)
$$
with $\psi=\psi(t^1,t^2)$ a differentiable function.
This simplifies $\mathbf L$ further.
From $\mathbf L^{44} = 0$ we get $\mathcal C^3 = 0$, i.e., ${f_{1,2}^{1}}-{f_{2,1}^{1}}=0$, and by $\mathbf L_{34} = 0$ we get $\Lambda = 0$. Then
$\mathbf L^{11} = 0$ is equivalent to
$q_{,11} + q_{,22} = (q_{,1}{}^2 + q_{,2}{}^2)/q$, which transforms to
the Laplace equation $\phi_{,11} + \phi_{,22} = 0$ under $q = e^\phi$, and by $L^{14}=L^{24}=0$ one gets that $\mathcal C^4 = \text{const}$, i.e., $f_{2,1}^2 - f_{1,2}^2= c e^\phi$, with $c\in\mathbb{R}$.
It follows that the remaining equations $L^{\mu \nu}=0$ are satisfied if, and only if, 
$$
\psi_{,11} + \psi_{,22} = c^2 e^\phi.
$$
Thus, when $\phi=c_0$, $c_0\in\mathbb{R}$, one has
$$
f_1^1=\phi_{21,1}+\phi_1,\quad f_1^2=\phi_{12,1},\quad f_2^1=\phi_{21,2},\quad f_2^2=\phi_{12,2}+c\,e^{c_0}\,t^1+\phi_2,
$$
with $\phi_{12}=\phi_{12}(t^1,t^2)$, $\phi_{21}=\phi_{21}(t^1,t^2)$, $\phi_{1}=\phi_{1}(t^1)$, $\phi_{2}=\phi_{2}(t^2)$ arbitrary differentiable functions and, by choosing new adapted coordinates
$$
\bar {t}^1=e^{c_0/2}t^1,\quad \bar {t}^2=e^{c_0/2}t^2,\quad \bar {z}^1=\phi_{21}+\phi_1+z^1,\quad \bar {z}^2=\phi_{12}+\phi_2+z^2,
$$
we get
$$
\fourmetric = (d \bar{t}^1)^2+(d \bar{t}^2)^2+ \psi (d \bar{z}^1)^2 +2(c\,\bar{t}^1 d \bar{t}^2+d \bar{z}^2)\,d \bar{z}^1,
$$
where $\partial_{\bar{t}^1}^2\psi + \partial_{\bar{t}^2}^2\psi = c^2$.\\
On the other hand,  if $\phi$ is non-constant, then, being a harmonic function, $\phi$ can be chosen for $t^1$ and the conjugate harmonic function for $t^2$. Then, one has
$$
f_1^1=\phi_{21,1}+\phi_1,\quad f_1^2=\phi_{12,1},\quad f_2^1=\phi_{21,2},\quad f_2^2=\phi_{12,2}+c\,e^{t^1}+\phi_2,
$$
with $\phi_{12}=\phi_{12}(t^1,t^2)$, $\phi_{21}=\phi_{21}(t^1,t^2)$, $\phi_{1}=\phi_{1}(t^1)$, $\phi_{2}=\phi_{2}(t^2)$ arbitrary differentiable functions and, by choosing new adapted coordinates
$$
\bar {t}^1=t^1,\quad \bar {t}^2=t^2,\quad \bar {z}^1=\phi_{21}+\phi_1+z^1,\quad \bar {z}^2=\phi_{12}+\phi_2+z^2,
$$
we get
$$
\fourmetric = e^{\bar{t}^1}(d \bar{t}^1)^2+e^{\bar{t}^1}(d \bar{t}^2)^2+ \psi (d \bar{z}^1)^2 +2(c\,e^{\bar{t}^1} d \bar{t}^2+d \bar{z}^2)\,d \bar{z}^1,
$$
where $\partial_{\bar{t}^1}^2\psi + \partial_{\bar{t}^2}^2\psi =e^{\bar{t}^1} c^2$.

In any case these Ricci-flat metrics are $pp$-waves since $\partial/\partial \bar{z}^2$ is a covariantly constant and null Killing vector.

\end{proof}






\subsection{Exact solutions in the case when $\ell_{\mathcal{C}}=0$} 

In the paper \cite{Kun} Kundu looked for solutions of the vacuum Einstein equations satisfying the condition $h^{kl} c_k c_l = 0$, where the scalars
$$
c_i = \epsilon^{\alpha\beta\rho\sigma} \xi_{(1)\alpha} \xi_{(2)\beta} \xi_{(i)\rho;\sigma}
$$
measure the orthogonal intransitivity
(cf.~\cite{Ger1}), and  $\xi_{(i)}=\partial_{z^i}$, $i=1,2$, are the Killing vectors. Kundu presented all solutions satisfying this condition, but without proof.
We reconstruct the proof below and extend his result to $\Lambda$-vacuum metrics.

We first notice that the Kundu condition $h^{kl} c_k c_l = 0$ is equivalent to 
$\mathbf{c}^{\alpha}_{12} \mathbf{c}^{\beta}_{12} \fourmetric_{\alpha\beta}
 = \mathbf{c}^{k^*}_{12} \mathbf{c}^{l^*}_{12} h_{kl} = 0$, i.e., 
in invariant terms,
$$
\ell_{\mathcal C} = 0.
$$


\begin{lem} \label{prop:inv through A}
When $C_\rho \ne 0$, the semi-invariant vector field $\mathcal C$ and the invariant $\ell_{\mathcal C}$ can be written as 
$$
\mathcal C = \mathop{\mathrm{sgn}}(q) \frac{1}{\sqrt{\mp\mathop{\mathrm{det}}H}} A H^{-1}, \qquad 
 \ell_{\mathcal C} = \frac{1}{\mp\mathop{\mathrm{det}}H} A H^{-1} A^\top,
$$
where $A$ is the constant vector introduced in Propositions \ref{prop:Lorentz Crho nonzero} and 
\ref{prop:Eucl Crho nonzero}.
Moreover, under  transformations $\partial/\partial z^{j}=\alpha_{j}^{i}\,\partial/\partial\bar{z}^{\,i}$,
with $(\alpha_{j}^{i})\in \mathrm{GL}(2, \mathbb{R})$, we have 
\begin{equation}\label{C sym}
H \to \alpha^\top H \alpha, \quad A \to A \alpha, \quad q \to q, \quad P \to P
\end{equation}
and, whenever $A$ is nonzero, it can be always normalised to any prescribed nonzero vector by 
means of transformation~\eqref{C sym}.
\end{lem}

\begin{proof}
This is easily checked using Propositions \ref{prop:Lorentz Crho nonzero} and \ref{prop:Eucl Crho nonzero}.
\end{proof}

Then  the $\Lambda$-vacuum Einstein metrics with $\ell_{\mathcal{C}}=0$ are described by the following

\begin{thm}
Every Lorentzian $\Lambda$-vacuum metric of the form \eqref{eq:gG} that satisfies the 
condition $\ell_{\mathcal C} = 0$ has one of the following forms:

{\rm 1.}  $pp$-waves with $C_\rho = 0$, described by 
Propositions~\ref{prop:Lorentz Crho zero} and~\ref{prop:Eucl Crho zero};


{\rm 2.} Petrov type II vacuum metrics of Kundu~{\rm\cite{Kun}}
\begin{equation}
\label{eq:Kundu solution}
\frac1{\sqrt{x}} \bigl(dx^2 + dy^2 + (x^{3/2} \psi + 1)\,du^2 + 2\,dy\,du\bigr),
\end{equation}
where $\psi$ solves the cylindrical Laplace equation
$\psi_{xx} - 2 \psi_x/x + \psi_{yy} = 0$;


{\rm 3.} Petrov type II $\Lambda$-vacuum metrics
\begin{equation}
\label{eq:Lambda Kundu solution}
\frac{3}{\Lambda} \frac{c^2 r}{c^2 r^3 + 1} \,dr^2 + \frac{c^2 r^3 + 1}{r}\,dy^2
+ \frac{2}{r}\,dy\,du  
+ \frac{r^3 \psi + 1}{r} \,du^2
+ \frac{2}{r^2}\,du\,dv
\end{equation}
where $c,\Lambda$ are nonzero constants and $\psi(r,y)$ is a solution of the 
separable linear equation 
\begin{equation}
\label{eq:Lambda Kundu equation}
\frac{(c^2 r^3 + 1)^2}{r^2} \psi_{rr}
 + \frac{(c^2 r^3 + 1)(4 c^2 r^3 + 1)}{r^3} \psi_r
+ 3 \frac{c^2}{\Lambda} \psi_{yy} = 0;
\end{equation}


{\rm 4.} Petrov type III $\Lambda$-vacuum metrics
\begin{equation}
\label{eq:Lambda Kundu solution c = 0}
\frac{3}{\Lambda  x^2}\,dx^2 + \frac1{x^2}dy^2
+ 2 x\,dy\,du  
+ \frac{2}{x^2}\,du\,dv
+ \frac{x^6 + \psi}{2  x^2} \,du^2
\end{equation}
where $\Lambda \ne 0$ and $\psi(x,y)$ satisfies the cylindrical Laplace equation
\begin{equation} \label{eq:Lambda Kundu equation c = 0}
\psi_{xx} - \frac{2}{x} \psi_x + \frac3{\Lambda} \psi_{yy} = 0.
\end{equation}
\end{thm}

\begin{proof} 
The case of $C_\rho = 0$ was completed in Section 7, since all metrics found in 
Propositions~\ref{prop:Lorentz Crho zero} and~\ref{prop:Eucl Crho zero}
satisfy $\ell_{\mathcal C} = 0$.
Assume henceforth that $C_\rho \ne 0$, so that we can use 
Propositions~\ref{prop:Lorentz Crho nonzero} and~\ref{prop:Eucl Crho nonzero}.

We consider the $\Lambda$-vacuum Einstein equations
augmented with condition $\ell_{\mathcal C} = 0$.
According to Lemma~\ref{prop:inv through A},
the scalar invariant $\ell_{\mathcal C}$ equals
\begin{equation} \label{eq:AHA}
\ell_{\mathcal C} = \frac{1}{\mp\det H} A H^{-1} A^\top.
\end{equation}
Therefore, condition $\ell_{\mathcal C} = 0$ implies that the vector $\mathcal C$ 
is null with respect to the matrix $H^{-1}$; then necessarily $\det H < 0$,
so that Proposition~\ref{prop:Lorentz Crho nonzero} is applicable.
Rewriting $\Lambda$-vacuum Einstein equations in terms of 
Weyl--Lewis--Papapetrou parameters and renaming $t^1,t^2$ to $x,y$ for brevity, 
we obtain
\begin{equation}
\label{EiEq qrsw ell}
\begin{gathered}
r_{xx} = \frac{1}{4} \frac{r_x^2 - r_y^2}{r}
+ \frac{1}{4} r \frac{w_x^2 - w_y^2 - s_x^2 + s_y^2}{s^2} 
+ \frac{1}{2} \frac{r_x q_x - r_y q_y}{q}
+ \frac{1}{4} r q \ell_{\mathcal C}
- \Lambda  r q,\\
r_{xy} = \frac{1}{2} \frac{r_x r_y}{r}
+ \frac{1}{2} r \frac{w_x w_y - s_x s_y}{s^2} 
+ \frac{1}{2} \frac{r_x q_y + r_y q_x}{q}, \\
r_{yy} = \frac{1}{4} \frac{r_y^2 - r_x^2}{r}
+ \frac{1}{4} r \frac{w_y^2 - w_x^2 - s_y^2 + s_x^2}{s^2} 
+ \frac{1}{2} \frac{r_y q_y - r_x q_x}{q}
+ \frac{1}{4} r q \ell_{\mathcal C}
- \Lambda  r q, \\
q_{xx} + q_{yy} = \frac{q_x^2 + q_y^2}{q}
+ \frac{1}{2} q \frac{r_x^2 + r_y^2}{r^2}
+ \frac{1}{2} q \frac{w_y^2 + w_x^2 - s_y^2 - s_x^2}{s^2} 
- \frac{3}{2} q^2 \ell_{\mathcal C}, \\
s_{xx} + s_{yy} = -\frac{r_x s_x + r_y s_y}{r}
+ \frac{w_x^2 + w_y^2 + s_x^2 + s_y^2}{s} 
- \frac{1}{2} q \frac{a_1^2 (s^2 + w^2) - 2 a_1 a_2 w + a_2^2}{r^3}, \\
w_{xx} + w_{yy} = -\frac{r_x w_x + r_y w_y}{r}
+ 2 \frac{w_x s_x + w_y s_y}{s} 
- \frac{s q}{r^3} a_1 (a_1 w - a_2),
\end{gathered}
\end{equation}
where
\begin{equation}
\label{KC qrsw ell}
\ell_{\mathcal C} = \frac{a_1^2 (s^2 - w^2) + 2 a_1 a_2 w - a_2^2}{s r^3}
= \frac{(a_1 s - a_1 w + a_2) (a_1 s + a_1 w - a_2)}{s r^3},
\end{equation}
according to eq.~\eqref{eq:AHA}.
If system~\eqref{EiEq qrsw ell} is solved, then the components $f_i^j$ can 
be found from the underdetermined system
\begin{equation}
\label{kundu_fij}
f_{1,y}^1 - f_{2,x}^1 = (a_1 s^2 - a_1 w^2 + a_2 w) \frac{q}{s r^2}, \\ \quad
f_{1,y}^2 - f_{2,x}^2 = (a_1 w - a_2) \frac{q}{s r^2}.
\end{equation}
The condition $\ell_{\mathcal C} = 0$ implies that $a_1 \ne 0$, since
otherwise $a_1 = a_2 = 0$, contradicting the assumption that the metric is
non-orthogonally transitive.
Then one has $w = \pm s + a_2/a_1$. 
One can choose the upper sign without loss in generality since the equations 
are invariant with respect to the transformation $w \to -w$, $a_2 \to -a_2$, 
$f_i^2 \to -f_i^2$.
Thus, we let
$$
w = s + \frac{a_2}{a_1}.
$$
Equations~\eqref{EiEq qrsw ell} turn into
\begin{equation}
\label{Kundu system}
\begin{gathered}
r_{xx} = \frac{1}{4} \frac{r_x^2 - r_y^2}{r}
+ \frac{1}{2} \frac{r_x q_x - r_y q_y}{q} - \Lambda  r q, \\
r_{xy} = \frac{1}{2} \frac{r_x r_y}{r} + \frac{1}{2} \frac{r_x q_y + r_y q_x}{q}, \\
r_{yy} = \frac{1}{4} \frac{r_y^2 - r_x^2}{r} + \frac{1}{2} \frac{r_y q_y
- r_x q_x}{q} - \Lambda  r q, \\
q_{xx} + q_{yy} = \frac{1}{2} q \frac{r_x^2 + r_y^2}{r^2} + \frac{q_x^2 + q_y^2}{q}, \\
s_{xx} + s_{yy} = -\frac{r_x s_x + r_y s_y}{r} + 2 \frac{s_x^2 + s_y^2}{s} - \frac{a_1^2 q s^2}{r^3},
\end{gathered}
\end{equation}
whence, 
\begin{equation}
\label{Kundu r Lambda}
r_{xx} + r_{yy} = -2 \Lambda r q.
\end{equation}
Moreover, the system \eqref{kundu_fij} reduces to 
\begin{equation}
\label{Kundu sys f}
-f_{1,y}^1 + f_{2,x}^1 + a_2 \frac{q}{r^2} = 0, \quad
-f_{1,y}^2 + f_{2,x}^2 - a_1 \frac{q}{r^2} = 0.\ \\
\end{equation}
Now, system~\eqref{Kundu system} is preserved under the coordinate 
transformations (isometries of the orbit metric)
$$
x \to \tilde x, \quad y \to \tilde y, \quad q \to \tilde q
$$
($r,s,w$ being unchanged), where
 $\tilde x(x,y)$, $\tilde y(x,y)$
are arbitrary functionally independent conjugate harmonic functions, i.e.,
$\tilde x_x = \tilde y_y$, $\tilde x_y = -\tilde y_x$, and
$$
\tilde q = \frac q J, \quad 
J = \frac{\partial(\tilde x,\tilde y)}{\partial(x,y)} = \tilde x_x^2 + \tilde x_y^2 \ne 0.
$$

In order to reproduce Kundu's result, assume that $\Lambda = 0$.
Then $r$ is harmonic by equation~\eqref{Kundu r Lambda}.
Moreover, $r$ is non-constant, since otherwise $C_\rho = 8 r_{x} r_{y}/r^2 q = 0$, which we excluded at the
beginning of the proof.
Therefore, $r$ can be chosen for $\tilde x$. 
Transforming back to the coordinates $x,y$, we thus identify $r = x$. 
Next we put $a_1 = 3/2$, $s = 1/(S + x^{-3/2})$, and $u = z^1$, $v = z^2$
to get solution~\eqref{eq:Kundu solution}, which is
easily identified with the Kundu non-orthogonally transitive 
solution~\cite[eq.~(4), $\alpha = 1$]{Kun}.

To cover the remaining two cases, assume that $\Lambda \ne 0$.
Then we can express
$$
q = -\frac{r_{xx} + r_{yy}}{2 \Lambda r}
$$
and substitute back into system~\eqref{Kundu system}, obtaining two third-order 
equations on $r$.
These are equivalent to
$$
\left(\frac{r_x^2 + r_y^2}{r_{xx} + r_{yy}} r^{1/2} - \frac23 r^{3/2}\right)_x = 0, 
\quad
\left(\frac{r_x^2 + r_y^2}{r_{xx} + r_{yy}} r^{1/2} - \frac23 r^{3/2}\right)_y = 0, 
$$
i.e.,
\begin{equation}
\label{Kundu r}
r^{1/2} (r_x^2 + r_y^2) = (\tfrac23 r^{3/2} + c) ({r_{xx} + r_{yy}}),
\end{equation}
where $c$ is an arbitrary constant.
Equation~\eqref{Kundu r} is equivalent to 
\begin{equation}
\label{rho harmonic}
\rho_{xx} + \rho_{yy} = 0
\end{equation}
under substitution $\rho = \rho(r)$, where $\rho(r)$ satisfies
\begin{equation}
\label{Kundu rho}
(\tfrac{2}{3} r^{3/2} + c) \frac{\partial^2 \rho}{\partial r^2} 
+ r^{1/2} \frac{\partial\rho}{\partial r} = 0.
\end{equation}
The last equation is easily integrated,
$$
\rho = \int \frac{\diff r}{r^{3/2} + c},
$$
which yields $r(\rho)$.
The integration constants are suppressed, since they correspond to point symmetries 
$\rho \to b_1 \rho + b_0$ of the Laplace equation~\eqref{rho harmonic} and as such
they are inessential.

Moreover, $r$ is non-constant, since otherwise $C_\rho = 8 r_{x} r_{y}/r^2 q = 0$, which we excluded at the
beginning of the proof. 
Then $\rho$ is non-constant as well and we are free to choose coordinates $x,y$ in such a way that 
$\rho = x$.
Otherwise said, we are free to assume that $r = r(x)$ is given by 
\begin{equation}
\label{r(x)}
x = \int \frac{\diff r}{r^{3/2} + c}.
\end{equation}
Now, the above expression for $q$ evaluates to
$$
q = -\frac3{4} \frac{r^{3/2} + c}{4 \Lambda}.
$$
To solve equations~\eqref{Kundu sys f}, we choose
$$
f_1^1 = 0, \quad f_2^1 = \mp\frac{a_2}{2 \Lambda r^{3/2}},
\\
f_1^2 = 0, \quad f_2^2 = \pm\frac{a_1}{2 \Lambda r^{3/2}}.
$$
Next steps differ according to whether $c = 0$ or not.

Assume that $c \ne 0$.
With $\rho$ being an arbitrary harmonic function, the equation for $s$ becomes 
$$
s_{xx} + s_{yy} - 2 \frac{s_x^2 + s_y^2}{s}
+ \frac{(r^{\frac{3}{2}} + c) (s_x \rho_x+ s_y \rho_y)}{r}
- \frac{3}{4} a_1^2 \frac{(r^2 + c \sqrt{r}) s^2 (\rho_x^2 + \rho_y^2)}{\Lambda  r^4}
= 0,
$$
which is linearizable in terms of the variable
$$
S = \frac{1}{s} + \frac{a_1^2}{3 \Lambda  c r^{3/2}},
\quad \text{i.e.,} \quad
s = (S - \frac{a_1^2}{3 \Lambda  c r^{3/2}})^{-1},
$$
giving 
\begin{equation}
\label{eq:S rho}
S_{xx} + S_{yy} + \frac{r^{3/2} + c}{r} (\rho_x S_x + \rho_y S_y) = 0.
\end{equation}
With $\rho = x$, equation~\eqref{eq:S rho} simplifies to
\begin{equation}
\label{eq:S}
S_{xx} + S_{yy} + \frac{r^{3/2} + c}{r} S_x = 0,
\end{equation}
and the metric becomes
\begin{equation}
- \frac{3}{4} \frac{r^{3/2} + c}{\sqrt{r} \Lambda} (dx^2 + dy^2)
- 2 \frac{\,dy  \,du}{\sqrt{r} \Lambda} 
- 2  r \,du  \,dv 
+ \frac{4}{3} \frac{\psi  r^2 - \sqrt{r}}{\Lambda  c r} \,du^2.
\end{equation}
Choosing $r,y,u,v$ as coordinates, we get the 
solution~\eqref{eq:Lambda Kundu solution} 
and equation~\eqref{eq:Lambda Kundu equation}
for the unknown function $\psi$.

Finally, assume that $c = 0$.
Then $x = -2/\sqrt r$, so that $r = 4/x^2$
and easy computation gives the metric~\ref{eq:Lambda Kundu solution c = 0}.
\end{proof}

\begin{rem} \rm
Let us remark that not only the cylindrical Laplace equation, but also the linear 
equation~\eqref{eq:Lambda Kundu equation} is separable by the substitution 
$\psi(r,y) = R(r) Y(y)$.
The $y$-dependent factor $Y(y)$ is easy to find from
$$
Y'' = \frac{\Lambda}{3 c^2} C Y,
$$
where $C$ is an arbitrary constant.
The difficult part is the equation for $R(r)$, which is
$$
\frac{(c^2 r^3 + 1)^2}{r^2} R_{rr}
 + \frac{(c^2 r^3 + 1)(4 c^2 r^3 + 1)}{r^3} R_r
 - C R = 0.
$$
It is easily solvable for $C = 0$, but to apply the linear superposition principle for solutions
we need enough solutions for $C \ne 0$, too. 
Since $c \ne 0$ by assumption, we can set it to $1$ by substitution $r \to r/c^{2/3}$,
obtaining 
\begin{equation}
\label{eq:Lambda Kundu equation R}
R'' + \frac{4 r^3 + 1}{r (r^3 + 1)} R' - \frac{C r^2}{(r^3 + 1)^2} R = 0.
\end{equation}
Equation~\eqref{eq:Lambda Kundu equation R} has five regular singular points given by
$r = 0$, $r^3 + 1 = 0$, $r = \infty$; at these points it is amenable to convergent series solution.
For $C \ne 0$ it has the first integral $I(r) =$ const of the form $I = I_1 R' + I_0 R$, given by
$$
I_0 = \frac{(r^3 + 1)^2}{C r^2} I_1,
\qquad
I_1'' + \frac{3}{r (r^3 + 1)} I_1' + \frac{C r^2}{(r^3 + 1)} I_1 = 0.
$$
\end{rem}

\section{Differential invariants for $\Lambda$-vacuum Einstein metrics}
\label{sec8}

In this section we answer the question of how many invariants are functionally independent 
on solutions of the $\Lambda$-vacuum Einstein equations.
Recall that every system of partial differential equations induces a proper subset $\mathcal E^{(k)}$
in each $k$th-order jet space, where $k$ is greater or equal to the order of the equation.
The 20 invariants given in Proposition~\ref{prop:main_2ord} can be easily restricted to 
$\mathcal E^{(k)}$. 
The easiest way to do this is to solve the equations with respect to a suitable set of the highest 
order variables, and use them as substitutions (i.e., treat the $\Lambda$-vacuum Einstein equations 
as an orthonomic system). See equations~\eqref{EiEq qrsw ell} for example.

\begin{prop}
\label{prop:Einst_independent_2ord}
The ten (semi-)invariants $C_{\rho}$, $C_{\chi}$,
$Q_{\chi}$, $Q_{\gamma}$, 
$\ell_{\mathcal{C}}$, $\Theta_{\rm I}$,
$X C_{\chi}$, $X Q_{\gamma}$,
$X^\bot C_{\chi}$, $X^\bot Q_{\gamma}$  
constitute a maximal set of scalar differential (semi-)invariants of order $\le 2$ 
functionally independent on generic solutions of the $\Lambda$-vacuum Einstein equations.
\end{prop}

\begin{proof} 
The rank of the Jacobian at a generic point of $\mathcal E^{(2)}$ is equal to 10.
\end{proof}

The simplest six relations are 
\[
\begin{aligned}
& C_{\mathrm{ric}} 
 + \tfrac{1}{2} C_\chi
 \mp_{\tilde{\bf g}} \tfrac{3}{2}  \ell_{\mathcal C} 
 = 0, 
\\
&C_\nu 
 \mp_{\tilde{\bf g}} \ell_{\mathcal C}
 + 4 \Lambda
 + \tfrac{1}{2} C_\rho = 0, 
\\
&{\pm_{\tilde{\bf g}}} (X^\bot C_\rho)^2
 + 4 Q_\chi C_\rho^2 
 - 16 (Q_\chi - Q_\gamma) C_\chi C_\rho 
 + 64 (Q_\chi - Q_\gamma)^2
 = 0, 
\\
&{\pm_{\tilde{\bf g}}} (X^\bot\ell_{\mathcal C})^2 
  \pm_{\tilde{\bf g} \bf h} 4 (\Theta_{\mathrm I} - 2 \ell_{\mathcal C} \sqrt{Q_\gamma})
    \Theta_{\mathrm I}
  + 4 \ell_{\mathcal C}^2 Q_\chi
 = 0,
\\
&XC_\rho
 + (C_\rho - C_\chi + 4 \Lambda \mp_{\tilde{\bf g}} \ell_{\mathcal C}) C_\rho
 + 8 (Q_\chi - Q_\gamma) = 0,
\\
&(Q_\chi - Q_\gamma) X\ell_{\mathcal C}^2
\mp_{\tilde{\bf g} \bf h} C_\rho   \sqrt{Q_\gamma}  \Theta_{\mathrm I}  X\ell_{\mathcal C}
 + (3 Q_\chi - 2 Q_\gamma) C_\rho  \ell_{\mathcal C}  X\ell_{\mathcal C} 
\\
&\quad + (C_\chi  C_\rho  Q_\chi + 2 C_\rho^2 Q_\chi - C_\rho^2 Q_\gamma - 4 Q_\chi^2 + 4 Q_\chi  Q_\gamma) \ell_{\mathcal C}^2 
\\
&\quad \mp_{\tilde{\bf g} \bf h} (2 C_\chi  C_\rho + C_\rho^2 - 8 Q_\chi)   \sqrt{Q_\gamma}  \Theta_{\mathrm I} \ell_{\mathcal C} 
 - 8  \sqrt{Q_\gamma}^3 \Theta_{\mathrm I} \ell_{\mathcal C}
\\
&\quad \pm_{\tilde{\bf g} \bf h} (C_\chi  C_\rho - \tfrac{1}{4} C_\rho^2 - 4 Q_\chi + 4 Q_\gamma)   \Theta_{\mathrm I}^2
 = 0.
\end{aligned}
\]
Here $\pm_{\tilde{\bf g} \bf h} = \mathop{\rm sgn}(\det \tilde{\bf g} \det \bf h)$,
$\mp_{\tilde{\bf g} \bf h} = -\mathop{\rm sgn}(\det \tilde{\bf g} \det \bf h)$ and $\pm_{\tilde{\bf g}} = \mathop{\rm sgn}(\det \tilde{\bf g})$.

\begin{prop}
The $\Lambda$-vacuum Einstein equations imply that\/ $\Xi$ and\/ $\Xi^\bot$ have the same 
Gaussian curvatures.
\end{prop}

\begin{proof}
The relation $C_{\mathrm{ric}} 
 + \tfrac{1}{2} C_\chi
 \mp_{\tilde{\bf g}} \tfrac{3}{2}  \ell_{\mathcal C} 
 = 0$ 
is equivalent to
$K(\Xi^{\perp})=K(\Xi)$.
\end{proof}

\section{The equivalence problem in the case $\ell_{\mathcal{C}}C_{\rho}\protect\neq0$}

\label{sect:equivalence} 


Let $\{I_{1},\dots,I_{6}\}$ be a maximal system of generically functionally
independent scalar differential invariants for $\mathfrak{G}^{(1)}$
on $J^{1}(\tau)$. For any metric $\mathbf{g}$, which is a section
of $\tau:E\rightarrow\mathcal{M}$, the restrictions $\{I_{i}[\mathbf{g}]\}$
of these invariants to the first-order prolongation of $\mathbf{g}$
provide at most two functionally independent differential invariants
on $\mathcal{S}$. The functional relations between these restricted
invariants are necessary conditions for any other metric $\bar{\mathbf{g}}$
being equivalent to $\mathbf{g}$. Here we discuss two alternative
methods for the solution of the equivalence problem for metrics 
satisfying
$\ell_{\mathcal{C}}C_{\rho}\neq0$ and possessing at least two functionally 
independent scalar invariants.

\subsection{The first method }

Let $\mathbf{g}$ and $\bar{\mathbf{g}}$ be two generic metrics which,
in adapted coordinates $(t^{1},t^{2},z^{1},z^{2})$ and $(\bar{t}^{1},\bar{t}^{2},\bar{z}^{1},\bar{z}^{2}),$
are written as
\begin{equation}
\mathbf{g}=b_{ij}(t^{1},t^{2})\,dt^{i}\,dt^{j}+2f_{ik}(t^{1},t^{2})\,dt^{i}\,dz^{k}+h_{kl}(t^{1},t^{2})\,dz^{k}\,dz^{l},\label{primag}
\end{equation}
and 
\begin{equation}
\bar{\mathbf{g}}=\bar{b}_{mn}(\bar{t}^{1},\bar{t}^{2})\,d\bar{t}^{m}\,d\bar{t}^{n}+2\bar{f}_{mr}(\bar{t}^{1},\bar{t}^{2})\,d\bar{t}^{m}\,d\bar{z}^{r}+\bar{h}_{rs}(\bar{t}^{1},\bar{t}^{2})\,d\bar{z}^{r}\,d\bar{z}^{s},\label{secondag}
\end{equation}
respectively. \\
 If $\mathbf{g}$ and $\bar{\mathbf{g}}$ are equivalent, then there
is a pair of indexes $a,b\in\{1,2,..,6\}$ such that $\{ I_{a}[\mathbf{g}](t^{1},t^{2}),$
$I_{b}[\mathbf{g}](t^{1},t^{2})\} $ and $\{ \bar{I}_{a}[\bar{\mathbf{g}}](\bar{t}^{1},\bar{t}^{2}),$
$\bar{I}_{b}[\bar{\mathbf{g}}](\bar{t}^{1},\bar{t}^{2})\} $
are two systems of functionally independent invariants on $\mathcal{S}$.
For ease of notation, we will denote by $\{\mathcal{I}^{1}(t^{1},t^{2}),\mathcal{I}^{2}(t^{1},t^{2})\}$
and $\{\bar{\mathcal{I}}^{1}(\bar{t}^{1},\bar{t}^{2}),\bar{\mathcal{I}}^{2}(\bar{t}^{1},\bar{t}^{2})\}$,
respectively, these two systems.

Then, by implicit function theorem, the system 
\[
\mathcal{I}^{1}(t^{1},t^{2})=\bar{\mathcal{I}}^{1}(\bar{t}^{1},\bar{t}^{2}),\qquad\mathcal{I}^{2}(t^{1},t^{2})=\bar{\mathcal{I}}^{2}(\bar{t}^{1},\bar{t}^{2})
\]
locally defines the $(t^{1},t^{2})$-part of coordinate
transformation (\ref{eq:psgroup_fin}), i.e., 
\begin{equation}
\bar{t}^{i}=\phi^{i}(t^{1},t^{2}),\qquad i=1,2.\label{eq:equiv1}
\end{equation}
On the other hand, under a coordinate transformation $P$ of the form
\eqref{eq:psgroup_fin}, the coordinate vector fields transform as
\[
P_{*}(\partial_{z^{j}})=\alpha_{j}^{i}\partial_{\bar{z}^{i}},\qquad P_{*}(\partial_{t^{j}})=\left(\frac{\partial\psi^{i}}{\partial t^{j}}\circ P^{-1}\right)\partial_{\bar{z}^{i}}+\left(\frac{\partial\phi^{i}}{\partial t^{j}}\circ P^{-1}\right)\partial_{\bar{t}^{i}}.
\]
Hence, in view of relations (see Proposition \ref{g_frame})
\[
P_{*}\left(\mathcal{C}\right)=\epsilon_{1}\,\bar{\mathcal{C}},\qquad P_{*}\left(\mathcal{C}^{\perp}\right)=\epsilon_{1}\epsilon_{2}\,\bar{\mathcal{C}}^{\perp},
\]
where $\epsilon_{1}=\mathop{\mathrm{sgn}}\left(J_{\phi}\right)$ and
$\epsilon_{2}=\mathop{\mathrm{sgn}}\left(\det\left(\alpha_{j}^{i}\right)\right)$,
one readily gets that 
\begin{equation}
\left(\begin{array}{@{}ll@{}}
\alpha_{1}^{1}\  & \alpha_{2}^{1}\\
\alpha_{1}^{2}\  & \alpha_{2}^{2}
\end{array}\right)=\epsilon_{1}\left(\begin{array}{@{}ll@{}}
\bar{\mathcal{C}}^{1}\  & \epsilon_{2}\left(\bar{\mathcal{C}}^{\perp}\right)^{1}\vspace{5pt}\\
\bar{\mathcal{C}}^{2}\  & \epsilon_{2}\left(\bar{\mathcal{C}}^{\perp}\right)^{2}
\end{array}\right)\left(\begin{array}{@{}ll@{}}
\mathcal{C}^{1}\  & \left(\mathcal{C}^{\perp}\right)^{1}\vspace{5pt}\\
\mathcal{C}^{2}\  & \left(\mathcal{C}^{\perp}\right)^{2}
\end{array}\right)^{-1},\label{eq:alpha_matrix}
\end{equation}
where $\bar{t}^{i}=\phi^{i}(t^{1},t^{2})$, $i=1,2$.

Analogously, in view of relations 
\[
P_{*}\left(\mathcal{H}\right)=\bar{\mathcal{H}},\qquad P_{*}\left(\mathcal{H}^{\perp}\right)=\epsilon_{1}\,\bar{\mathcal{H}}^{\perp},
\]
one also gets under substitution (\ref{eq:psgroup_fin})  that 
\begin{equation}
\left(\begin{array}{@{}ll@{}}
{\displaystyle \frac{\partial\phi^{1}}{\partial t^{1}}} & {\displaystyle \frac{\partial\phi^{1}}{\partial t^{2}}}\vspace{10pt}\\
{\displaystyle \frac{\partial\phi^{2}}{\partial t^{1}}} & {\displaystyle \frac{\partial\phi^{2}}{\partial t^{2}}}
\end{array}\right)=\left(\begin{array}{@{}ll@{}}
\bar{\mathcal{H}}^{1}\vspace{5pt}\  & \epsilon_{1}\left(\bar{\mathcal{H}}^{\perp}\right)^{1}\\
\bar{\mathcal{H}}^{2} & \epsilon_{1}\left(\bar{\mathcal{H}}^{\perp}\right)^{2}
\end{array}\right)\left(\begin{array}{@{}ll@{}}
\mathcal{H}^{1}\vspace{5pt}\  & \left(\mathcal{H}^{\perp}\right)^{1}\\
\mathcal{H}^{2} & \left(\mathcal{H}^{\perp}\right)^{2}
\end{array}\right)^{-1}\label{eq:Jacobi_matr_phi}
\end{equation}
and 
\begin{equation}
\begin{array}{@{}ll@{}}
\left(\begin{array}{@{}ll@{}}
{\displaystyle \frac{\partial\psi^{1}}{\partial t^{1}}} & {\displaystyle \frac{\partial\psi^{1}}{\partial t^{2}}}\vspace{10pt}\\
{\displaystyle \frac{\partial\psi^{2}}{\partial t^{1}}} & {\displaystyle \frac{\partial\psi^{2}}{\partial t^{2}}}
\end{array}\right)= & \left(\begin{array}{@{}ll@{}}
\alpha_{1}^{1}\  & \alpha_{2}^{1}\vspace{5pt}\\
\alpha_{1}^{2}\  & \alpha_{2}^{2}
\end{array}\right)\left(\begin{array}{@{}ll@{}}
f_{1}^{1}\  & f_{2}^{1}\vspace{5pt}\\
f_{1}^{2}\  & f_{2}^{2}
\end{array}\right)-\\
 & -\left(\begin{array}{@{}ll@{}}
\bar{f}_{1}^{1}\  & \bar{f}_{2}^{1}\vspace{5pt}\\
\bar{f}_{1}^{2}\  & \bar{f}_{2}^{2}
\end{array}\right)\left(\begin{array}{@{}ll@{}}
\bar{\mathcal{H}}^{1}\  & \epsilon_{1}\left(\bar{\mathcal{H}}^{\perp}\right)^{1}\vspace{5pt}\\
\bar{\mathcal{H}}^{2}\  & \epsilon_{1}\left(\bar{\mathcal{H}}^{\perp}\right)^{2}
\end{array}\right)\left(\begin{array}{@{}ll@{}}
\vspace{5pt}\mathcal{H}^{1} & \ \left(\mathcal{H}^{\perp}\right)^{1}\\
\mathcal{H}^{2} & \ \left(\mathcal{H}^{\perp}\right)^{2}
\end{array}\right)^{-1},
\end{array}\label{eq:Jacobi_matr_psi}
\end{equation}
where $f_{j}^{k}=f_{js}h^{sk}$ and $\bar{f}_{j}^{k}=\bar{f}_{js}\bar{h}^{sk}$
with $h^{sk}$ and $\bar{h}^{sk}$ denoting the elements of the inverse
matrix $(h_{ij})^{-1}$ and $(\bar{h}_{ij})^{-1}$, respectively.

Thus one gets the following 
\begin{thm}
The metrics 
\[
\mathbf{g}=b_{ij}(t^{1},t^{2})\,dt^{i}\,dt^{j}+2f_{ik}(t^{1},t^{2})\,dt^{i}\,dz^{k}+h_{kl}(t^{1},t^{2})\,dz^{k}\,dz^{l}
\]
and 
\[
\bar{\mathbf{g}}=\bar{b}_{mn}(\bar{t}^{1},\bar{t}^{2})\,d\bar{t}^{m}\,d\bar{t}^{n}+2\bar{f}_{mr}(\bar{t}^{1},\bar{t}^{2})\,d\bar{t}^{m}\,d\bar{z}^{r}+\bar{h}_{rs}(\bar{t}^{1},\bar{t}^{2})\,d\bar{z}^{r}\,d\bar{z}^{s},
\]
with $\ell_{\mathcal{C}}C_{\rho}\neq0$, are equivalent
if, and only if, there exists a coordinate transformation 
\[
P:\qquad\bar{t}^{\,i}=\phi^{i}(t^{1},t^{2}),\qquad\bar{z}^{\,i}=\alpha_{j}^{i}\,z^{j}+\psi^{i}(t^{1},t^{2}),\qquad(\alpha_{j}^{i})\in\mathrm{GL}(2, \mathbb{R})
\]
satisfying the following conditions: 
\begin{enumerate}
\item [\rm{{(i)}}] $\{ \mathcal{I}^{1}=I_{a}[\mathbf{g}],\mathcal{I}^{2}=I_{b}[\mathbf{g}]\} $
and $\{ \bar{\mathcal{I}}^{1}=\bar{I}_{a}[\bar{\mathbf{g}}],\bar{\mathcal{I}}^{2}=\bar{I}_{b}[\bar{\mathbf{g}}]\} $,
for some $a,b\in\{1,2,...,6\}$, are two systems of functionally independent
scalar invariants on $\mathcal{S}$ and the coordinate transformation
$\left\{ \bar{t}^{1}=\phi^{1}(t^{1},t^{2}),\,\bar{t}^{2}=\phi^{2}(t^{1},t^{2})\right\} $
is implicitly defined by 
\[
\mathcal{I}^{1}-\bar{\mathcal{I}}^{1}=0,\qquad\mathcal{I}^{2}-\bar{\mathcal{I}}^{2}=0;
\]
\item [\rm{{(ii)}}] the right hand side of \eqref{eq:alpha_matrix}, with $\epsilon_{1}=\mathop{\mathrm{sgn}}\left(J_{\phi}\right)$
and $\epsilon_{2}=\mathop{\mathrm{sgn}}\left(\det\left(\alpha_{j}^{i}\right)\right)$,
is a constant matrix coinciding with $(\alpha_{j}^{i})\in\mathrm{GL}(2, \mathbb{R})$;
\item [\rm{{(iii)}}] the transformation $\left\{ \bar{t}^{1}=\phi^{1}(t^{1},t^{2}),\,\bar{t}^{2}=\phi^{2}(t^{1},t^{2})\right\} $,
defined in {\rm(i)}, satisfies \eqref{eq:Jacobi_matr_phi};
\item [\rm{{(iv)}}] the functions $\psi^{i}=\psi^{i}(t^{1},t^{2})$, $i=1,2$,
are solutions of an integrable system of first-order partial differential
equations defined by \eqref{eq:Jacobi_matr_psi};
\item [\rm{{(v)}}] the matrix $(\alpha_{j}^{i})$
and the derivatives of $\phi^{i}$ and
$\psi^{i}$ satisfy the system \eqref{cond_equiv} for
$\mathbf{g}$ and $\bar{\mathbf{g}}$.
\end{enumerate}
\end{thm}

\subsection{The second method}

Let $\mathbf{g}$ and $\bar{\mathbf{g}}$ be two metrics which, in
adapted coordinates $(t^{1},t^{2},z^{1},z^{2})$ and $(\bar{t}^{1},\bar{t}^{2},\bar{z}^{1},\bar{z}^{2}),$
read as \eqref{primag} and \eqref{secondag}, respectively. Under
the assumption $\ell_{\mathcal{C}}C_{\rho}\neq0$, by using the semi-invariant
orthogonal frames (see Proposition \ref{g_frame})
\[
\left\{ \mathcal{Y}_{1}=\mathcal{H},\mathcal{Y}_{2}=\mathcal{H}^{\bot},\mathcal{Y}_{3}=\mathcal{C},\mathcal{Y}_{4}=\mathcal{C}^{\bot}\right\} 
\]
and 
\[
\left\{ \bar{\mathcal{Y}}_{1}=\bar{\mathcal{H}},\bar{\mathcal{Y}}_{2}=\bar{\mathcal{H}}^{\bot},\bar{\mathcal{Y}}_{3}=\bar{\mathcal{C}},\bar{\mathcal{Y}}_{4}=\bar{\mathcal{C}}^{\bot}\right\} ,
\]
and the corresponding semi-invariant dual co-frames $\left\{ \omega_{1},\omega_{2},\omega_{3},\omega_{4}\right\} $
and $\left\{ \bar{\omega}_{1},\bar{\omega}_{2},\bar{\omega}_{3},\bar{\omega}_{4}\right\} $,
$\mathbf{g}$ and $\bar{\mathbf{g}}$ can be written as 
\[
\mathbf{g}=\ell_{\mathcal{H}}\omega_{1}^{2}+\ell_{\mathcal{H}^{\perp}}\omega_{2}^{2}+\ell_{\mathcal{C}}\omega_{3}^{2}+\ell_{\mathcal{C}^{\perp}}\omega_{4}^{2}
\]
and 
\[
\bar{\mathbf{g}}=\bar{\ell}_{\bar{\mathcal{H}}}\bar{\omega}_{1}^{2}+\bar{\ell}_{\bar{\mathcal{H}}^{\perp}}\bar{\omega}_{2}^{2}+\bar{\ell}_{\bar{\mathcal{C}}}\bar{\omega}_{3}^{2}+\bar{\ell}_{\bar{\mathcal{C}}^{\perp}}\bar{\omega}_{4}^{2},
\]
respectively.

We notice that, in view of (\ref{transf_frame}), under the pseudo-group
action (\ref{eq:psgroup_fin}) the co-frame transforms as
\begin{equation}
\omega_{1}\mapsto\omega_{1},\quad\omega_{2}\mapsto(\mathop{\mathrm{sgn}}J_{\phi})\omega_{2},\quad\omega_{3}\mapsto(\mathop{\mathrm{sgn}}J_{\phi})\omega_{3},\quad\omega_{4}\mapsto(\mathop{\mathrm{sgn}}J_{\phi})(\mathop{\mathrm{sgn}}\mathop{\mathrm{det}}\alpha_{j}^{i})\omega_{4}.\label{eq:coframe_transf}
\end{equation}
Moreover,in terms of local adapted coordinates $\{t^{1},t^{2},z^{1},z^{2}\}$,
one has 
\begin{equation}
\left(\begin{array}{@{}c@{}}
\omega_{1}\vspace{5pt}\\
\omega_{2}
\end{array}\right)=\left(\begin{array}{@{}ll@{}}
\mathcal{H}^{1}\vspace{5pt}\: & \left(\mathcal{H}^{\perp}\right)^{1}\\
\mathcal{H}^{2} & \left(\mathcal{H}^{\perp}\right)^{2}
\end{array}\right)^{-1}\left(\begin{array}{@{}c@{}}
dt^{1}\vspace{5pt}\\
dt^{2}
\end{array}\right)\label{omega12}
\end{equation}
and 
\begin{equation}
\begin{array}{l}
\left(\begin{array}{@{}c@{}}
\omega_{3}\vspace{5pt}\\
\omega_{4}
\end{array}\right)=\left(\begin{array}{@{}ll@{}}
\mathcal{C}^{1}\vspace{5pt}\: & \left(\mathcal{C}^{\perp}\right)^{1}\\
\mathcal{C}^{2} & \left(\mathcal{C}^{\perp}\right)^{2}
\end{array}\right)^{-1}\left[\left(\begin{array}{@{}ll@{}}
f_{1}^{1}\vspace{5pt}\: & f_{2}^{1}\\
f_{1}^{2} & f_{2}^{2}
\end{array}\right)\left(\begin{array}{@{}c@{}}
dt^{1}\vspace{5pt}\\
dt^{2}
\end{array}\right)+\left(\begin{array}{@{}c@{}}
dz^{1}\vspace{5pt}\\
dz^{2}
\end{array}\right)\right]\end{array}.\label{omega34}
\end{equation}

From now on, we assume that there is a pair of indexes
$a,b\in\{1,2,..,6\}$ such that $\left\{ I_{a}[\mathbf{g}](t^{1},t^{2}),\right.$
$\left.I_{b}[\mathbf{g}](t^{1},t^{2})\right\} $ and $\left\{ \bar{I}_{a}[\bar{\mathbf{g}}](\bar{t}^{1},\bar{t}^{2}),\right.$
$\left.\bar{I}_{b}[\bar{\mathbf{g}}](\bar{t}^{1},\bar{t}^{2})\right\} $
are two systems of functionally independent invariants on $\mathcal{S}$.
For ease of notation, we will denote these two systems by $\{\mathcal{I}^{1}(t^{1},t^{2}),\mathcal{I}^{2}(t^{1},t^{2})\}$
and $\{\bar{\mathcal{I}}^{1}(\bar{t}^{1},\bar{t}^{2}),\bar{\mathcal{I}}^{2}(\bar{t}^{1},\bar{t}^{2})\}$,
respectively. In view
of the functional independence, $\left\{ \mathcal{I}^{1},\mathcal{I}^{2},z^{1},z^{2}\right\} $
and $\left\{ \bar{\mathcal{I}}^{1},\bar{\mathcal{I}}^{2},\bar{z}^{1},\bar{z}^{2}\right\} $
define new adapted coordinates for $\mathbf{g}$ and $\bar{\mathbf{g}}$,
respectively. Therefore, by the implicit function theorem,
$\left\{ \mathcal{I}^{1}=\mathcal{I}^{1}(t^{1},t^{2}),\,\mathcal{I}^{2}=\mathcal{I}^{2}(t^{1},t^{2})\right\} $
and $\left\{ \bar{\mathcal{I}}^{1}=\bar{\mathcal{I}}^{1}(\bar{t}^{1},\bar{t}^{2}),\,\bar{\mathcal{I}}^{2}=\bar{\mathcal{I}}^{2}(\bar{t}^{1},\bar{t}^{2})\right\}$
define local coordinate transformations 

\[
t^{i}=m^{i}(\mathcal{I}^{1},\mathcal{I}^{2}),\qquad i=1,2
\]
and 
\[
\bar{t}^{i}=\bar{m}^{i}(\bar{\mathcal{I}}^{1},\bar{\mathcal{I}}^{2}),\qquad i=1,2.
\]
This entails that, according to (\ref{omega12})
and (\ref{omega34}), $\left\{ \omega_{1},\omega_{2},\omega_{3},\omega_{4}\right\} $
can be written in terms of new adapted coordinates $\left\{ \mathcal{I}^{1},\,\mathcal{I}^{2},\,z^{1},\,z^{2}\right\} $
as follows 

\[
\omega_{h}=a_{hi}\,d\mathcal{I}^{i}+p_{hi}\,dz^{i},
\]
where $a_{hi}=a_{hi}\left(\mathcal{I}^{1},\mathcal{I}^{2}\right)$,
$p_{hi}=p_{hi}\left(\mathcal{I}^{1},\mathcal{I}^{2}\right)$ and

\[
\det\left(\begin{array}{@{}ll@{}}
a_{11} & a_{12}\\
a_{21} & a_{22}
\end{array}\right)\neq0,\qquad\det\left(\begin{array}{@{}ll@{}}
p_{31} & p_{32}\\
p_{41} & p_{42}
\end{array}\right)\neq0,\qquad p_{1i}=p_{2i}=0.
\]
In particular the coefficients $a_{hi}$ and $p_{hi}$ can be computed
by using the following identities 
\begin{equation}
\left(\begin{array}{@{}c@{}}
\omega_{1}\vspace{5pt}\\
\omega_{2}
\end{array}\right)=\left(\begin{array}{@{}ll@{}}
\mathcal{H}^{1}\vspace{5pt}\: & \left(\mathcal{H}^{\perp}\right)^{1}\\
\mathcal{H}^{2} & \left(\mathcal{H}^{\perp}\right)^{2}
\end{array}\right)_{\mathrm{invar.}}^{-1}\left(\begin{array}{@{}ll@{}}
{\displaystyle \frac{\partial m^{1}}{\partial\mathcal{I}^{1}}}\vspace{7pt}\: & {\displaystyle \frac{\partial m^{1}}{\partial\mathcal{I}^{2}}}\\
{\displaystyle \frac{\partial m^{2}}{\partial\mathcal{I}^{1}}} & {\displaystyle \frac{\partial m^{2}}{\partial\mathcal{I}^{2}}}
\end{array}\right)\left(\begin{array}{@{}c@{}}
d\mathcal{I}^{1}\vspace{5pt}\\
d\mathcal{I}^{2}
\end{array}\right),\label{eq:equiv_omega1-2}
\end{equation}
and 
\begin{equation}
\begin{array}{l}
\left(\begin{array}{@{}c@{}}
\omega_{3}\vspace{5pt}\\
\omega_{4}
\end{array}\right)=\left(\begin{array}{@{}ll@{}}
\mathcal{C}^{1}\vspace{5pt}\: & \left(\mathcal{C}^{\perp}\right)^{1}\\
\mathcal{C}^{2} & \left(\mathcal{C}^{\perp}\right)^{2}
\end{array}\right)_{\mathrm{invar.}}^{-1}\left[\left(\begin{array}{@{}ll@{}}
f_{1}^{1}\vspace{5pt}\: & f_{2}^{1}\\
f_{1}^{2} & f_{2}^{2}
\end{array}\right)_{\mathrm{invar.}}\left(\begin{array}{@{}ll@{}}
{\displaystyle \frac{\partial m^{1}}{\partial\mathcal{I}^{1}}}\vspace{7pt}\: & {\displaystyle \frac{\partial m^{1}}{\partial\mathcal{I}^{2}}}\\
{\displaystyle \frac{\partial m^{2}}{\partial\mathcal{I}^{1}}} & {\displaystyle \frac{\partial m^{2}}{\partial\mathcal{I}^{2}}}
\end{array}\right)\left(\begin{array}{@{}c@{}}
d\mathcal{I}^{1}\vspace{5pt}\\
d\mathcal{I}^{2}
\end{array}\right)+\left(\begin{array}{@{}c@{}}
dz^{1}\vspace{5pt}\\
dz^{2}
\end{array}\right)\right],\end{array}\label{eq:equiv_omega3-4}
\end{equation}
where \lq\lq{}invar.\rq\rq{} means the restriction to $\{t^{i}=m^{i}\left(\mathcal{I}^{1},\mathcal{I}^{2}\right)\}$.

Analogously one can write $\{ \bar{\omega}_{1},\bar{\omega}_{2},\bar{\omega}_{3},\bar{\omega}_{4}\} $
in terms of the adapted coordinates $\{ \bar{\mathcal{I}}^{1},\,\bar{\mathcal{I}}^{2},\,\bar{z}^{1},\,\bar{z}^{2}\} $
as 
\[
\bar{\omega}_{i}=\bar{a}_{hi}\,d\bar{\mathcal{I}}^{i}+\bar{p}_{hi}\,d\bar{z}^{i},
\]
where the coefficients $\bar{a}_{hi}=\bar{a}_{hi}(\bar{\mathcal{I}}^{1},\bar{\mathcal{I}}^{2})$,
$\bar{p}_{hi}=\bar{p}_{hi}(\mathcal{\bar{I}}^{1},\bar{\mathcal{I}}^{2})$
are such that
\[
\det\left(\begin{array}{@{}ll@{}}
\bar{a}_{11} & \bar{a}_{12}\\
\bar{a}_{21} & \bar{a}_{22}
\end{array}\right)\neq0,\qquad\det\left(\begin{array}{@{}ll@{}}
\bar{p}_{31} & \bar{p}_{32}\\
\bar{p}_{41} & \bar{p}_{42}
\end{array}\right)\neq0,\qquad\bar{p}_{1i}=\bar{p}_{2i}=0
\]
and can be computed by using formulas analogous to (\ref{eq:equiv_omega1-2})
and (\ref{eq:equiv_omega3-4}) (where in this case \lq\lq{}invar.\rq\rq{}
means the restriction to $\{ \bar{t}^{i}=\bar{m}^{i}(\bar{\mathcal{I}}^{1},\bar{\mathcal{I}}^{2})\} $).

\begin{lem}
\label{lema_equiv} Under the pseudo-group action \eqref{eq:psgroup_fin},
the coefficients $a_{ij}=\omega_{i}(\partial_{\mathcal{I}^{j}})$
transform according to the following formulas 
\begin{equation}
\bar{a}_{1i}=a_{1i},\qquad\bar{a}_{2i}=\epsilon_{1}\,a_{2i},\qquad\bar{a}_{3i}=\epsilon_{1}\,a_{3i},\qquad\bar{a}_{4i}=\epsilon_{1}\epsilon_{2}\,a_{4i},\label{semi-inv-aij}
\end{equation}
with $\epsilon_{1}=\mathop{\mathrm{sgn}}J_{\phi}$ and $\epsilon_{2}=\mathop{\mathrm{sgn}}\mathop{\mathrm{det}}\alpha_{j}^{i}$,
whereas the coefficients $p_{hi}=\omega_{h}(\partial_{z^{i}})$ transform
as 
\begin{equation}
p_{3i}=\epsilon_{1}\,\bar{p}_{3s}\alpha_{i}^{s},\qquad p_{4i}=\epsilon_{1}\epsilon_{2}\,\bar{p}_{4s}\alpha_{i}^{s},\label{eq:semi-inv-bij}
\end{equation}
with $(\alpha_{j}^{i})\in \mathrm{GL}(2,\mathbb{R})$ such that 
\begin{equation}
\left(\begin{array}{@{}ll@{}}
\alpha_{1}^{1}\: & \alpha_{2}^{1}\vspace{5pt}\\
\alpha_{1}^{2}\: & \alpha_{2}^{2}
\end{array}\right)=\epsilon_{1}\,\left(\begin{array}{@{}ll@{}}
\bar{\mathcal{C}}^{1}\: & \epsilon_{2}\left(\bar{\mathcal{C}}^{\perp}\right)^{1}\vspace{5pt}\\
\bar{\mathcal{C}}^{2}\: & \epsilon_{2}\left(\bar{\mathcal{C}}^{\perp}\right)^{2}
\end{array}\right)_{\mathrm{invar.}}\left(\begin{array}{@{}ll@{}}
\mathcal{C}^{1} & \:\left(\mathcal{C}^{\perp}\right)^{1}\vspace{5pt}\\
\mathcal{C}^{2} & \:\left(\mathcal{C}^{\perp}\right)^{2}
\end{array}\right)_{\mathrm{invar.}}^{-1}.\label{cond_matr_const}
\end{equation}
\end{lem}
\begin{proof}
Equations (\ref{semi-inv-aij}) and (\ref{eq:semi-inv-bij}) readily
follow by (\ref{eq:coframe_transf}). On the other hand, when $\bar{\omega}_{3}$,
$\bar{\omega}_{4}$ are obtained by $\omega_{3}$, $\omega_{4}$ through
the pseudo-group action (\ref{eq:psgroup_fin}), equations (\ref{semi-inv-aij})
and (\ref{eq:semi-inv-bij}) entail (\ref{cond_matr_const}).
\end{proof}
\medskip{}

Notice that (\ref{cond_matr_const})
is the same condition (\ref{eq:alpha_matrix})
obtained in the first method. Here one also has the following

\begin{lem}
\label{Rem_equiv} The fact that right hand side of \eqref{cond_matr_const}
is an element of\/ $\mathrm{GL}(2,\mathbb{R})$ is equivalent
to the following condition 
\begin{equation}
\begin{array}{l}
\left(\begin{array}{@{}ll@{}}
\bar{\mathcal{C}}^{1}\: & \epsilon_{2}\left(\bar{\mathcal{C}}^{\perp}\right)^{1}\vspace{5pt}\\
\bar{\mathcal{C}}^{2}\; & \epsilon_{2}\left(\bar{\mathcal{C}}^{\perp}\right)^{2}
\end{array}\right)_{\mathrm{invar.}}^{-1}\dfrac{\partial}{\partial{\mathcal{I}^{s}}}\left(\begin{array}{@{}ll@{}}
\bar{\mathcal{C}}^{1}\; & \epsilon_{2}\left(\bar{\mathcal{C}}^{\perp}\right)^{1}\vspace{5pt}\\
\bar{\mathcal{C}}^{2}\: & \epsilon_{2}\left(\bar{\mathcal{C}}^{\perp}\right)^{2}
\end{array}\right)_{\mathrm{invar.}}=\vspace{10pt}\\
\qquad\qquad\qquad\qquad=\left(\begin{array}{@{}ll@{}}
\mathcal{C}^{1}\: & \left(\mathcal{C}^{\perp}\right)^{1}\vspace{5pt}\\
\mathcal{C}^{2}\: & \left(\mathcal{C}^{\perp}\right)^{2}
\end{array}\right)_{\mathrm{invar.}}^{-1}\dfrac{\partial}{\partial{\mathcal{I}^{s}}}\left(\begin{array}{@{}ll@{}}
\mathcal{C}^{1}\: & \left(\mathcal{C}^{\perp}\right)^{1}\vspace{5pt}\\
\mathcal{C}^{2}\: & \left(\mathcal{C}^{\perp}\right)^{2}
\end{array}\right)_{\mathrm{invar.}},\qquad s=1,2.
\end{array}\label{invar_epsilon_matr}
\end{equation}
In particular, under transformations preserving the orientations of
the leaves of $\Xi$, \eqref{invar_epsilon_matr} is equivalent to
the invariance of the matrices 
\begin{equation}
\left(\begin{array}{@{}ll@{}}
\mathcal{C}^{1}\: & \left(\mathcal{C}^{\perp}\right)^{1}\vspace{5pt}\\
\mathcal{C}^{2}\: & \left(\mathcal{C}^{\perp}\right)^{2}
\end{array}\right)_{\mathrm{invar.}}^{-1}\frac{\partial}{\partial{\mathcal{I}^{s}}}\left(\begin{array}{@{}ll@{}}
\mathcal{C}^{1}\: & \left(\mathcal{C}^{\perp}\right)^{1}\vspace{5pt}\\
\mathcal{C}^{2}\: & \left(\mathcal{C}^{\perp}\right)^{2}
\end{array}\right)_{\mathrm{invar.}},\qquad s=1,2.\label{invar_matr}
\end{equation}
Moreover, by introducing the functions 
\[
\mathfrak{c}_{11}^{s}=\frac{\left(\mathcal{C}^{\perp}\right)^{2}{\displaystyle \frac{\partial\mathcal{C}^{1}}{\partial\mathcal{I}^{s}}}-\left(\mathcal{C}^{\perp}\right)^{1}{\displaystyle \frac{\partial\mathcal{C}^{2}}{\partial\mathcal{I}^{s}}}}{\mathcal{C}^{1}\left(\mathcal{C}^{\perp}\right)^{2}-\mathcal{C}^{2}\left(\mathcal{C}^{\perp}\right)^{1}},\qquad\mathfrak{c}_{22}^{s}=\frac{\mathcal{C}^{1}{\displaystyle \frac{\partial\left(\mathcal{C}^{\perp}\right)^{2}}{\partial\mathcal{I}^{s}}}-\mathcal{C}^{2}{\displaystyle \frac{\partial\left(\mathcal{C}^{\perp}\right)^{1}}{\partial\mathcal{I}^{s}}}}{\mathcal{C}^{1}\left(\mathcal{C}^{\perp}\right)^{2}-\mathcal{C}^{2}\left(\mathcal{C}^{\perp}\right)^{1}},
\]
\[
\mathfrak{c}_{12}^{s}=\frac{\left(\mathcal{C}^{\perp}\right)^{2}{\displaystyle \frac{\partial\left(\mathcal{C}^{\perp}\right)^{1}}{\partial\mathcal{I}^{s}}}-\left(\mathcal{C}^{\perp}\right)^{1}{\displaystyle \frac{\partial\left(\mathcal{C}^{\perp}\right)^{2}}{\partial\mathcal{I}^{s}}}}{\mathcal{C}^{1}\left(\mathcal{C}^{\perp}\right)^{2}-\mathcal{C}^{2}\left(\mathcal{C}^{\perp}\right)^{1}},\qquad\mathfrak{c}_{21}^{s}=\frac{\mathcal{C}^{1}{\displaystyle \frac{\partial\mathcal{C}^{2}}{\partial\mathcal{I}^{s}}}-\mathcal{C}^{2}{\displaystyle \frac{\partial\mathcal{C}^{1}}{\partial\mathcal{I}^{s}}}}{\mathcal{C}^{1}\left(\mathcal{C}^{\perp}\right)^{2}-\mathcal{C}^{2}\left(\mathcal{C}^{\perp}\right)^{1}}.
\]
where $s=1,2$, condition \eqref{invar_epsilon_matr} is equivalent to the identities
\begin{equation}
\label{cijs_inv}
\mathfrak{c}_{11}^{s}=\bar{\mathfrak{c}}_{11}^{s},\qquad\mathfrak{c}_{22}^{s}=\bar{\mathfrak{c}}_{22}^{s},\qquad\mathfrak{c}_{12}^{s}=\epsilon_{2}\bar{\mathfrak{c}}_{12}^{s},\qquad\mathfrak{c}_{21}^{s}=\epsilon_{2}\bar{\mathfrak{c}}_{21}^{s},
\end{equation}

\end{lem}
\begin{proof}
By differentiating (\ref{cond_matr_const}) with respect to
$\mathcal{I}^{s}$ and observing that $\mathcal{I}^{s}=\bar{\mathcal{I}}^{s}$, one gets (\ref{invar_epsilon_matr}). Conversely, (\ref{cond_matr_const}) follows by integrating (\ref{invar_epsilon_matr}). The rest of the proof follows by straightforward computations.\end{proof}

Now the following remark is in order.

\begin{rem}
\label{rem:RMK}In view of (\ref{cond_matr_const}), transformation
formulas (\ref{eq:semi-inv-bij}) are equivalent to the following identities
\[
\mathfrak{p}_{3}=\epsilon_{1}\bar{\mathfrak{p}}_{3},\qquad\mathfrak{p}_{4}=\epsilon_{2}\bar{\mathfrak{p}}_{4},\qquad\mathfrak{p}_{3}^{\perp}=\epsilon_{1}\epsilon_{2}\bar{\mathfrak{p}}_{3}^{\perp},\qquad\mathfrak{p}_{4}^{\perp}=\bar{\mathfrak{p}}_{4}^{\perp},
\]
where $\mathfrak{p}_{a}=p_{as}\mathcal{C}^{s}$, $\mathfrak{p}_{a}^{\perp}=p_{as}\left(\mathcal{C}^{\perp}\right)^{s}$,
$a=3,4$. On the other hand, \eqref{invar_epsilon_matr} is equivalent to the identities \eqref{cijs_inv}.
Thus, one gets that (\ref{semi-inv-aij}),
(\ref{eq:semi-inv-bij}) and \eqref{cond_matr_const} are equivalent to the invariance of the $18$
functions 
\begin{equation}
\begin{array}{lllllll}
a_{1i},\quad & \mathfrak{p_{4}^{\perp},\quad} & \left(a_{2i}\right)^{2},\quad & \left(a_{3i}\right)^{2},\quad & \left(a_{4i}\right)^{2},\quad & a_{2i}a_{3i,\quad} & \left(\mathfrak{p}_{3}\right)^{2},\vspace{5pt}\\
\left(\mathfrak{p}_{3}^{\perp}\right)^{2},\quad & \left(\mathfrak{p}_{4}\right)^{2},\quad & \mathfrak{p}_{3}a_{2i},\quad & \mathfrak{p}_{3}a_{3i},\quad & \mathfrak{p}_{3}^{\perp}a_{4i},\quad & \mathfrak{p}_{3}\mathfrak{p}_{4}a_{4i},\quad & \mathfrak{p}_{3}\mathfrak{p}_{3}^{\perp}\mathfrak{p}_{4},\vspace{5pt}\\
\left(\mathfrak{c}_{12}^{s}\right)^{2},\quad & \left(\mathfrak{c}_{21}^{s}\right)^{2},\quad& \mathfrak{c}_{12}^{s}\mathfrak{p}_{4},\quad &\mathfrak{c}_{21}^{s}\mathfrak{p}_{4}

\end{array}\label{eq:invarsfin}
\end{equation}
with $i,s=1,2$.
\end{rem}
\bigskip{}

Then one has the following
\begin{thm}
Two metrics 
\[
\mathbf{g}=b_{ij}(t^{1},t^{2})\,dt^{i}\,dt^{j}+2f_{ik}(t^{1},t^{2})\,dt^{i}\,dz^{k}+h_{kl}(t^{1},t^{2})\,dz^{k}\,dz^{l}
\]
and 
\[
\bar{\mathbf{g}}=\bar{b}_{mn}(\bar{t}^{1},\bar{t}^{2})\,d\bar{t}^{m}\,d\bar{t}^{n}+2\bar{f}_{mr}(\bar{t}^{1},\bar{t}^{2})\,d\bar{t}^{m}\,d\bar{z}^{r}+\bar{h}_{rs}(\bar{t}^{1},\bar{t}^{2})\,d\bar{z}^{r}\,d\bar{z}^{s},
\]
with $\ell_{\mathcal{C}}C_{\rho}\neq0$ and $\bar{\ell}_{\bar{\mathcal{C}}}\bar{C}_{\bar{\rho}}\neq0$, are equivalent
if and only if there are two systems $\{\mathcal{I}^{1}(t^{1},t^{2}),\mathcal{I}^{2}(t^{1},t^{2})\}$
and $\{\bar{\mathcal{I}}^{1}(\bar{t}^{1},\bar{t}^{2}),\bar{\mathcal{I}}^{2}(\bar{t}^{1},\bar{t}^{2})\}$
of functionally independent first-order scalar differential invariants
on $\mathcal{S}$, such that the six fundamental first-order scalar differential invariants
of $\mathbf{g}$ depend on $(\mathcal{I}^{1},\mathcal{I}^{2})$ in
the same way as the corresponding six fundamental invariants of $\bar{\mathbf{g}}$
depend on $(\bar{\mathcal{I}}^{1},\bar{\mathcal{I}}^{2})$.
\end{thm}
\begin{proof}
If $\mathbf{g}$ and $\bar{\mathbf{g}}$ are two
equivalent metrics such that $\ell_{\mathcal{C}}C_{\rho}\neq0$ and $\bar{\ell}_{\bar{\mathcal{C}}}\bar{C}_{\bar{\rho}}\neq0$,
then there are certainly two systems $\{\mathcal{I}^{1}(t^{1},t^{2}),\mathcal{I}^{2}(t^{1},t^{2})\}$
and $\{\bar{\mathcal{I}}^{1}(\bar{t}^{1},\bar{t}^{2}),\bar{\mathcal{I}}^{2}(\bar{t}^{1},\bar{t}^{2})\}$
of functionally independent first-order scalar differential invariants
on $\mathcal{S}$ such that the six fundamental first-order scalar differential invariants of $\mathbf{g}$ depend on $(\mathcal{I}^{1},\mathcal{I}^{2})$ in the same way as the corresponding six fundamental invariants of $\bar{\mathbf{g}}$ depend on $(\bar{\mathcal{I}}^{1},\bar{\mathcal{I}}^{2})$.

For the proof of the converse one needs first to observe that, in the
generic case, all first-order scalar differential invariants are functions
of the six fundamental scalar differential invariants. Hence, when
restricting to a metric $\mathbf{g}$ with two functionally independent
scalar differential invariants $(\mathcal{I}^{1},\mathcal{I}^{2})$,
all first-order scalar differential invariants become functions of
$(\mathcal{I}^{1},\mathcal{I}^{2})$. Then, under the given assumptions, all
first-order scalar differential invariants of $\mathbf{g}$ depend
on $(\mathcal{I}^{1},\mathcal{I}^{2})$ in the same way as all the
corresponding first-order scalar differential invariants of $\bar{\mathbf{g}}$
depend on $(\bar{\mathcal{I}}^{1},\bar{\mathcal{I}}^{2})$. As a consequence, under the transformation $\{\bar{\mathcal{I}}^{1}=\mathcal{I}^{1},\, \bar{\mathcal{I}}^{2}=\mathcal{I}^{2}\}$,
 the invariants (\ref{eq:invarsfin}) for $\mathbf{g}$ transform
to the corresponding invariants for $\bar{\mathbf{g}}$. Hence, in view of Lemma \ref{lema_equiv}, Lemma \ref{Rem_equiv} and Remark \ref{rem:RMK}, the matrix $(\alpha_{j}^{i})$ defined by \eqref{cond_matr_const} belongs to $\mathrm{GL}(2,\mathbb{R})$ and defines the adapted coordinate transformation 
\[
P:\;\left\{ \begin{array}{l}
\bar{\mathcal{I}}^{1}=\mathcal{I}^{1}\\
\bar{\mathcal{I}}^{2}=\mathcal{I}^{2}\\
\bar{z}^{i}=\alpha_{j}^{i}z^{j}
\end{array}\right.
\]
by which $\left\{ \omega_{1},\omega_{2},\omega_{3},\omega_{4}\right\} $ transforms to $\left\{ \bar{\omega}_{1},\bar{\omega}_{2},\bar{\omega}_{3},\bar{\omega}_{4}\right\} $. It follows that $P$ is an
isometry between $\mathbf{g}$ and $\bar{\mathbf{g}}$: 
\[
P^{*}(\bar{\mathbf{g}})=P^{*}\left(\bar{\ell}_{\bar{\mathcal{H}}}\bar{\omega}_{1}^{2}+\bar{\ell}_{\bar{\mathcal{H}}^{\perp}}\bar{\omega}_{2}^{2}+\bar{\ell}_{\bar{\mathcal{C}}}\bar{\omega}_{3}^{2}+\bar{\ell}_{\bar{\mathcal{C}}^{\perp}}\bar{\omega}_{4}^{2}\right)=\ell_{\mathcal{H}}\omega_{1}^{2}+\ell_{\mathcal{H}^{\perp}}\omega_{2}^{2}+\ell_{\mathcal{C}}\omega_{3}^{2}+\ell_{\mathcal{C}^{\perp}}\omega_{4}^{2}=\mathbf{g}.
\]
\end{proof}

\begin{cor}
The equivalence class of a metric $\mathbf{g}$ such that $\ell_{\mathcal{C}}C_{\rho}\neq0$ is completely characterised
by the way the six fundamental first-order scalar differential invariants
$I_{1}=C_{\rho}$,$I_{2}=C_{\chi}$, $I_{3}=Q_{\chi}$, $I_{4}=Q_{\gamma}$,
$I_{5}=\ell_{\mathcal{C}}$, $I_{6}=(\Theta_{\mathrm{I}})^{2}$ 
depend on two functionally independent first-order scalar differential
invariants $(\mathcal{I}^{1},\mathcal{I}^{2})$. 
\end{cor}
\medskip{}

\subsubsection{Example}

We consider here the Van den Bergh metric 
\[
\begin{array}{l}
\mathbf{g}=\cosh\left(\sqrt{6}\,t^{1}\right)\left\{ \sinh^{4}\left({t^{2}}\right)\left[\left(d{t^{1}}\right)^{2}-\left(d{t^{2}}\right)^{2}\right]+2\,\sinh^{2}\left({t^{2}}\right)\left[d{z^{2}}+\cosh\left({t^{2}}\right)d{t^{1}}\right]^{2}\right\} \vspace{10pt}\\
\qquad\qquad+{\displaystyle \frac{12}{\cosh\left(\sqrt{6}\,t^{1}\right)}}\,\left[d{z^{1}}+\cosh\left({t^{2}}\right)d{z^{2}}+1/2\,\cosh^{2}\left({t^{2}}\right)d{t^{1}}\right]^{2}.
\end{array}
\]
This is a Ricci-flat metric with two Killing vector fields $\partial_{z^{1}}$
and $\partial_{z^{2}}$ and an orthogonally intransitive $\Xi$.

In this case the six fundamental first-order scalar differential invariants
are 
\begin{equation}
\begin{array}{l}
C_{\rho}=-4\,{\displaystyle \frac{\cosh^{2}\left({t^{2}}\right)}{\cosh\left(\sqrt{6}\,t^{1}\right)\sinh^{6}\left({t^{2}}\right)}},\vspace{5pt}\\
C_{\chi}=-6\,{\displaystyle \frac{\sinh^{2}\left(\sqrt{6}\,t^{1}\right)-1}{\cosh^{3}\left(\sqrt{6}\,t^{1}\right)\sinh^{4}\left({t^{2}}\right)}},\vspace{5pt}\\
Q_{\chi}=6\,{\displaystyle \frac{\sinh^{2}\left(\sqrt{6}\,t^{1}\right)\,\left[-6\,\sinh^{2}\left(t^{2}\right)+\cosh^{2}\left({t^{2}}\right)\cosh^{2}\left(\sqrt{6}\,t^{1}\right)\right]}{\cosh^{6}\left(\sqrt{6}\,t^{1}\right)\sinh^{10}\left({t^{2}}\right)}},\vspace{5pt}\\
Q_{\gamma}=-36\,{\displaystyle \frac{\sinh^{2}\left(\sqrt{6}\,t^{1}\right)}{\cosh^{6}\left(\sqrt{6}\,t^{1}\right)\sinh^{8}\left({t^{2}}\right)}},\vspace{5pt}\\
\ell_{\mathcal{C}}=2\,{\displaystyle \frac{1}{\cosh\left(\sqrt{6}\,t^{1}\right)\sinh^{4}\left({t^{2}}\right)}},\vspace{5pt}\\
\left(\Theta_{1}\right)^{2}=144\,{\displaystyle \frac{\sinh^{2}\left(\sqrt{6}\,t^{1}\right)}{\sinh^{16}\left({t^{2}}\right)\cosh^{8}\left(\sqrt{6}\,t^{1}\right)}}.
\end{array}\label{fund_invar_vanderberg}
\end{equation}
Then, by choosing 
\[
\mathcal{I}^{1}=C_{\rho},\qquad\mathcal{I}^{2}=\ell_{\mathcal{C}},
\]
one can write 
\begin{equation}
t^{1}=\frac{1}{\sqrt{6}}\text{arccosh}\left(\frac{2}{\ell_{\mathcal{C}}}\left(\frac{C_{\rho}}{2\ell_{\mathcal{C}}}+1\right)^{2}\right),\qquad t^{2}=\text{arctanh}\left(\frac{1}{\sqrt{1-{\displaystyle \frac{1}{2}\frac{C_{\rho}}{\ell_{\mathcal{C}}}}}}\right).\label{eq:t_vanderberg}
\end{equation}
It follows that, by substituting (\ref{eq:t_vanderberg}) in (\ref{fund_invar_vanderberg}),
the remaining $4$ fundamental first-order scalar differential invariants
reduce to the following functions of $\mathcal{I}^{1}=C_{\rho}$ and
$\mathcal{I}^{2}=\ell_{\mathcal{C}}$: 
\begin{equation}
\left\{ \begin{array}{l}
C_{\chi}={\displaystyle \frac{-3\,{\ell_{\mathcal{C}}}\,\left(-8\,\ell_{\mathcal{C}}^{6}+\left(C_{\rho}^{2}+4C_{\rho}\ell_{\mathcal{C}}+4\ell_{\mathcal{C}}^{2}\right)^{2}\right)}{\left({C_{\rho}}+2\,{\ell_{\mathcal{C}}}\right)^{4}},}\vspace{5pt}\\
{\displaystyle Q_{\chi}=\frac{-3\,{\ell_{\mathcal{C}}}\,\left(48\,\ell_{\mathcal{C}}^{7}+C_{\rho}\,\left(C_{\rho}^{2}+4C_{\rho}\ell_{\mathcal{C}}+4\ell_{\mathcal{C}}^{2}\right)^{2}\right)\left(\left(C_{\rho}^{2}+4C_{\rho}\ell_{\mathcal{C}}+4\ell_{\mathcal{C}}^{2}\right)^{2}-4\ell_{\mathcal{C}}^{6}\right)}{4\,\left({C_{\rho}}+2\,{\ell_{\mathcal{C}}}\right)^{8}},\vspace{5pt}}\\
{\displaystyle Q_{\gamma}=\frac{-36\,\ell_{\mathcal{C}}^{8}\left(\left(C_{\rho}^{2}+4C_{\rho}\ell_{\mathcal{C}}+4\ell_{\mathcal{C}}^{2}\right)^{2}-4\ell_{\mathcal{C}}^{6}\right)}{\left({C_{\rho}}+2\,{\ell_{\mathcal{C}}}\right)^{8}},\vspace{5pt}}\\
(\Theta_{\mathrm{I}})^{2}=-\ell_{\mathcal{C}}^{2}\,Q_{\gamma}.
\end{array}\right.\label{sinvariant_char_VanDerBerg}
\end{equation}
Conditions $C_{\rho}=\mathcal{I}^{1}$, $\ell_{\mathcal{C}}=\mathcal{I}^{2}$
and (\ref{sinvariant_char_VanDerBerg}) 
give an invariant characterisation of the equivalence class of metrics
equivalent to the Van den Bergh metric.

\section{Conclusions}

Considering metrics with two commuting Killing vectors, referred to as $\mathrm G_2$-metrics, we introduced scalar differential
invariants of the first and second order with respect to the pseudogroup of transformations
preserving the Riemannian submersion structure. 
The set of (semi-)invariants is sufficient for the
solution of the equivalence problem in the generic case, which was our first goal. 
Our (semi-)invariants are designed to have tractable coordinate expressions, which is particularly suitable for the equivalence problem. 
The next goal is to look for relations satisfied by known metrics or classes thereof. 
By computing all metrics that satisfy these relations, 
one can, in principle, either extend the set of known solutions or prove an invariant 
characterization of a class of metrics in the spirit of~\cite{F-S}.
To provide an example, we extended the Kundu class of metrics, 
defined by $\ell_{\mathcal C} = 0$,  to the $\Lambda$-vacuum case.
A multitude of such relations have been already identified and will be studied elsewhere.

\section*{Acknowledgements}
The first author was partially supported by CNPq, grant 310577/2015-2 and grant 
422906/2016-6. The second author was supported by GA\v{C}R under
project P201/12/G028.



\begin{thebibliography}{10}
\small 

\bibitem{Aleks} G.A. Alekseev,
Thirty years of studies of integrable reductions of Einstein's field equations, in:
T. Damour and R.T. Jantzen, eds.,\textit{
The Twelfth Marcel Grossmann Meeting} (World Scientific, Singapore, 2012) 645--666.

{\bibitem{A-V-L} D.V. Alekseevsky, A.M. Vinogradov, V.V Lychagin,
}\textit{Basic Ideas and Concepts of Differential Geometry,}{
Encyclopaedia Math. Sci. 28 (Springer, Berlin, 1991).}{}

{\bibitem{A-Ka} J.E. \AA man and A. Karlhede, An algorithmic
classification of geometries in general relativity, in: }\textit{Proceedings
of the Fourth ACM Symposium on Symbolic and Algebraic Computation,}{
Snowbird, Utah, USA (ACM, 1981) 79\textendash 84.}{}


{\bibitem{B-Z} V.A. Belinski\u{\i}\ and V.E. Zakharov, Integration
of the Einstein equations by means of the inverse scattering problem,
}\textit{Soviet Phys. JETP}{ }\textbf{75}{
(1978) (6) 1955\textendash 1971.}{}

{\bibitem{Bess} A.L. Besse, }\textit{Einstein Manifolds}{
(Springer, Berlin, 1987).}


{\bibitem{B-M} M. Bradley and M. Marklund, Finding solutions
to Einstein's equations in terms of invariant objects, }\textit{Class.
Quantum Grav.}{ }\textbf{13}{ (1996) 3021\textendash 3037.}

\bibitem{C-M}
J. Carminati and R.G. McLenaghan,
Algebraic invariants of the Riemann tensor in a four-dimensional Lorentzian space,
{\it J. Math. Phys.} {\bf 32} (1991) 3135--3140.





\bibitem{Ca2} B.~Carter, Killing horizons and orthogonally
transitive groups in space-time, \textit{J. Math. Phys.}
\textbf{10} (1969) 70\textendash 81.


{
}{}


{
}{}

{
}{}


\bibitem{C-H-P}
A. Coley, S. Hervik and N. Pelavas,
Spacetimes characterized by their scalar curvature invariants,
{\it Class. Quantum Grav.} {\bf 26} (2009) 025013 (33pp).


{\bibitem{C1} C.M.~Cosgrove, A new formulation of the field
equations for the stationary axisymmetric gravitational field, I.
General theory, }\textit{J. Phys. A: Math. Gen.}{
}\textbf{11}{ (1978) 2389\textendash 2404.}{}

{\bibitem{C2} C.M.~Cosgrove, A new formulation of the field
equations for the stationary axisymmetric gravitational field, II.
Separable solutions, }\textit{J. Phys. A: Math. Gen.}{
}\textbf{11}{ (1978) 2405\textendash 2430.}{}


{
}{}



{
}{}

{
}{}

{
}{}



\bibitem{F-S}
J.J. Ferrando and J.A. S\'aez,
An intrinsic characterization of the Kerr metric,
{\it Class. Quantum Grav.} {\bf 26} (2009) 075013 (13pp).

\bibitem{Gaff}
B.~Gaffet,
The Einstein equations with two commuting Killing vectors,
{\it Class. Quantum Grav.} {\bf 7} (1990) 2017--2044.


{\bibitem{Ger1} R.~Geroch, A method for generating solutions
of Einstein's equations, }\textit{J. Math. Phys.}{
}\textbf{12}{ (1971) 918\textendash 924.}{}

{\bibitem{Ger2} R.~Geroch, A method for generating new solutions
of Einstein's equations. II, }\textit{J. Math. Phys.}{
}\textbf{13}{ (1972) 394\textendash 404.}{}




{
}{}


{
}{}
{\bibitem{Ka} A. Karlhede, A review of the geometrical equivalence
of metrics in general relativity, }\textit{Gen. Rel. Grav.}{
}\textbf{12}{ (1980) 693\textendash 707.}{}

{\bibitem{Ka2006} A. Karlhede, The equivalence problem, }\textit{Gen.
Rel. Grav.}{ }\textbf{38}{ (2006) 1109\textendash 1114.}{}

{\bibitem{Ka-M} A. Karlhede and M.A.H. MacCallum, On determining
the isometry group of a Riemannian space. }\textit{Gen. Rel.
Grav.}{ }\textbf{14}{ (1982) 673\textendash 682.}{}  
 





{
}{}

{
}{}

\bibitem{K-R} 
C. Klein and O. Richter, 
{\it Ernst Equation and Riemann Surfaces} (Springer, Berlin et al., 2005).


\bibitem{K-S} 
C. A. Kolassis and N. O. Santos, Spacetimes with a preferred null direction and a two-dimensional group of isometries: the null dust case,
{\it  Class. Quantum Grav.} {\bf 4} (1987) 599--618



{\bibitem{Kun} P. Kundu, 
Class of ``noncanonical'' vacuum metrics with two commuting Killing vectors,
{\it Phys. Rev. Lett.} {\bf 42} (1979) 416--417.\par}

{\bibitem{Lew} T.~Lewis, Some special solutions of the equations
of axially symmetric gravitational fields, }\textit{Proc.
Roy. Soc. London A}{ }\textbf{136}{ (1932)
176\textendash 192.}{}


\bibitem{L-Yu 1}
V. Lychagin and V. Yumaguzhin, 
Differential invariants and exact solutions of the Einstein--Maxwell equation,
{\it Anal. Math. Phys.} {\bf 7} (2017) 19--29. 

\bibitem{L-Yu 2}
V. Lychagin and V. Yumaguzhin, 
Differential invariants and exact solutions of the Einstein equations,
{\it Anal. Math. Phys.} {\bf 7} (2017) 107--115.


{
}{}

{
}{}

{
}{}

{
}{}

{\bibitem{M-S}
M.~Marvan and O. Stol\'{\i}n, 
On local equivalence problem of
spacetimes with two orthogonally transitive commuting Killing fields,
{\it J. Math. Phys.} {\bf 49} (2008), no. 2, 022503, 17 pp.}{}

\bibitem{M-M-C}
R. Milson, D. McNutt and A. Coley,
Invariant classification of vacuum pp-waves,
{\it J. Math. Phys.} {\bf 54} (2013) 022502.


{\bibitem{Ol} P.J. Olver, }\textit{Equivalence, Invariants
and Symmetry}{ (Cambridge University Press, New York, 1995).}{}

{\bibitem{ON} B. O'Neill, The fundamental equations of a submersion,
{\it Michigan Math. J.} {\bf 13} (1966) 459--469.}{}

{\bibitem{Ov} L.V.~Ovsiannikov, }\textit{Group Analysis
of Differential Equations}{ (Academic Press, New York, 1982).}

{\bibitem{Pet} A.Z. Petrov, }\textit{Einstein Spaces}{
(Pergamon, New York, 1969).}

{
}{}

{
}{}
{\bibitem{Pa} A. Papapetrou, Champs gravitationnels stationnaires 
\`a sym\'etrie axiale, }\textit{Ann. Inst. H. Poincar\'e A}{
}\textbf{4}{ (1966) (2) 83\textendash 105.}{}

{\bibitem{P-S-dI3} D. Pollney, J.E.F. Skea and R.A. d'Inverno,
Classifying geometries in general relativity: III. Classification
in practice, }\textit{Class. Quantum Grav.}{ }\textbf{17}{
(2000) 2885\textendash 2902.}{}   


{\bibitem{P-P-C-M}
V. Pravda, A. Pravdov\'a, A. Coley and R. Milson, 
All spacetimes with vanishing curvature invariants,
{\it Class. Quantum Grav.} {\bf 19} (2002) 6213--6236.
}{}

{
}{}
{\bibitem{Sk} J.E.F. Skea, A spacetime whose invariant classification
requires the fifth covariant derivative of the Riemann tensor, }\textit{Class.
Quantum Grav.}{ }\textbf{17}{ (2000) L69\textendash L74.}{}


{\bibitem{S-K-M-H-H} H. Stephani, D. Kramer, M. MacCallum,
C. Hoenselaers and E. Herlt, }\textit{Exact Solutions of Einstein's
Field Equations,}{ 2nd ed. (Cambridge University Press, Cambridge,
2003).}{}

\bibitem{Ve} E.~Verdaguer, Soliton solutions in spacetime
with two spacelike Killing fields, \textit{Phys. Rep.}
\textbf{229} (1993) (1\textendash 2) 1\textendash 80.

{
}{}

{
}{}

\bibitem{Bergh}
N. Van den Bergh,  A class of inhomogeneous cosmological models with separable metrics, {\it Class. Quantum Grav.} {\bf 5} (1988) 167-177. 



\bibitem{W-R}
J.T. Whelan and J.D.Romano, 
Quasistationary binary inspiral. I. Einstein equations for the two Killing vector spacetime,
{\it Phys. Rev. D } {\bf 60} (1999) 084009.


\bibitem{WCL}
Woei Chet Lim, 
Non-orthogonally transitive G2 spike solution,
{\it Class. Quantum Grav.  } {\bf 32} (2015) 162001 (8pp).

{\bibitem{Zor} K. \.{Z}orawski, On deformation invariants.
An application of Lie's theory of groups. }\textit{Acta Math.}{
}\textbf{16}{ (1892) 1\textendash 64 (in German). }{}
\end{thebibliography}
\end{document}